\documentclass[10pt, reqno]{amsart}
\usepackage{graphicx, amssymb, amsmath, amsthm}
\numberwithin{equation}{section}
\usepackage{tikz, caption, subcaption}
\usepackage[font=small,labelfont=bf]{caption}
\usepackage{color}
\usepackage{bbm}
\usepackage{mathrsfs}
\usepackage{comment}

\newcommand{\CS}{Y}

\newcommand{\mk}{\mathfrak}
\newcommand{\conR}{\mathsf{R}_{\mathsf{Conj}}}

\usepackage{hyperref}  
\hypersetup{colorlinks=true, linkcolor=blue, anchorcolor=blue, 
citecolor=red, filecolor=blue, menucolor=blue,
urlcolor=blue}
\hypersetup{colorlinks=true, linkcolor=blue, anchorcolor=blue, 
citecolor=red, filecolor=blue, menucolor=blue,
urlcolor=blue}

\let\Im=\undefined\DeclareMathOperator*{\Im}{Im}

\newcommand{\R}{\mathbb{R}}

\newcommand{\Z}{\mathbb{Z}}

\newcommand\CC{\mathcal{C}}

\newtheorem{theorem}{Theorem}[section]

\newtheorem{lemma}[theorem]{Lemma}

\newtheorem{corollary}[theorem]{Corollary}

\newtheorem{proposition}[theorem]{Proposition}

\theoremstyle{definition}
\newtheorem{definition}[theorem]{Definition}
\newtheorem{remark}[theorem]{Remark}

\makeatletter
\newcommand{\Extend}[5]{\ext@arrow0099{\arrowfill@#1#2#3}{#4}{#5}}
\makeatother

\begin{document}
\title[Dispersive estimates for Schr\"odinger and wave ]{Pointwise dispersive estimates for Schr\"odinger and wave equations on conical singular spaces}

\author{Qiuye Jia}
\address{Mathematical Sciences Institute, the Australian National University; }
\email{Qiuye.Jia@anu.edu.au; }

\author{Junyong Zhang}
\address{School of Mathematics and Physics, University of Science and Technology
Beijing, Beijing 100083, China; Department of Mathematics, Beijing Institute of Technology, Beijing 100081, China;}
\email{zhang\_junyong@bit.edu.cn; }

\begin{abstract}
We study the pointwise decay estimates for the  Schr\"odinger and wave equations on a product cone $(X,g)$, where the metric $g=dr^2+r^2 h$ and $X=C(Y)=(0,\infty)\times Y$ is a product cone over the closed Riemannian manifold $(Y,h)$ with metric $h$. 
Under the assumption that the {conjugate radius} $\conR$ of $Y$ satisfies $\conR>\pi$, we prove the pointwise dispersive estimates for the Schr\"odinger and half-wave propagators in this setting. The key ingredient is the modified Hadamard parametrix on $Y$ in which the role of the conjugate points does not come into play if $\conR>\pi$. 
A new finding is that a threshold of the {conjugate radius} of $Y$ for the pointwise dispersive estimates in this setting is the magical number $\pi$.
\end{abstract}

 \maketitle

\begin{center}
 \begin{minipage}{120mm}
 \end{minipage}
 \end{center}

 \tableofcontents

\section{Introduction and main results}
In this paper, we study the pointwise dispersive estimates for the Schr\"odinger and wave equations, continuing the investigations about Strichartz estimates carried out in \cite{HZ,Z, ZZ1,ZZ2}, on the product cone $(X,g)$, where the metric $g=dr^2+r^2 h$ and $X=C(Y)=(0,\infty)\times Y$ is an $n$-dimensional product cone over the closed Riemannian manifold $(Y,h)$ of dimension $n-1$ with metric $h$. Let $\Delta_g$ be the positive Laplace-Beltrami operator  on $X$, which is the Friedrichs self-adjoint extension from the domain $\CC_c^\infty(X)$ that consists of the compactly supported smooth functions on the
interior of the cone. Consider the Schr\"odinger operator 
\begin{equation} \label{oper:S}
H=\Delta_g+V_0(y)r^{-2}
\end{equation}
in the coordinates $(r,y)$ of the above product cone $(X,g)$. The purpose of this paper is to study the pointwise decay estimates of the associated Schr\"odinger equation
\begin{equation}
\begin{cases}
i\partial_t u(t, r, y)+H u(t, r, y)=0,\\
u|_{t=0}=f(r,y).
\end{cases}
\end{equation}
As is well known, the free Schr\"odinger equation in Euclidean space $\R^n$ without potential obeys the decay estimate
\begin{equation}\label{est:classical-dis}
\|e^{it\Delta} f\|_{L^\infty(\R^n)}\leq C |t|^{-\frac n2} \|f\|_{L^1(\R^n)},\quad t\neq 0,
\end{equation}
where the constant $C$ is independent of $f$ and $t$. Therefore, the Strichartz estimates on Euclidean space (e.g. see \cite{KT}) read
\begin{equation}\label{est:Stri}
\|e^{it\Delta}f\|_{L^pL^q(\mathbb{R}\times \R^n)}\leq
C\|f\|_{L^2(\R^n)},
\end{equation}
where $(p, q)$ is an \emph{admissible pair}, i.e.
\begin{equation}\label{1.1}
2\leq p,q\leq\infty, \quad 2/p+n/q=n/2,\quad (p,q,n)\neq(2,\infty,2).
\end{equation}
It has been known that the geometry (e.g. trapping geodesics, conjugate points) of the setting plays an important role in the study of dispersive solutions to evolution equations. For example, the Strichartz estimates in \cite{BGT} on compact manifolds are local-in-time and have a loss of regularity due to the elliptic trapped geodesics. However, the loss of regularity of the local-in-time Strichartz estimate can be recovered in \cite{BGH} if the trapped geodesic is hyperbolic, and further extended to be global-in-time in \cite{ZZ17}.  
From the results of \cite{HZ,ZZ1,ZZ2}, the conjugate points have no effect on the Strichartz estimates even though one needs elaborate microlocal arguments. However, the pointwise decay estimates are more delicate than the Strichartz estimates. It is known that there is an interesting phenomenon that the usual Strichartz estimates are still true even though the classical pointwise decay estimates fail, which is illustrated by \cite{BPSS, FFFP} for the inverse-square potential and by \cite{HW, HZ,ZZ1,ZZ2} for the conjugate points.\vspace{0.2cm}

In this paper, we study the pointwise decay estimates for the solutions of the Schr\"odinger and wave equations associated with the conical singular operator $H$ given in \eqref{oper:S}. 
More precisely, we aim to detect the quantitative influence of the conjugate points and the inverse-square potential on the decay rate of dispersive estimates, 
which is the motivation of this sequence of papers. This operator $H$ has attracted researchers' interest from different disciplines such as geometry, analysis and physics. Even for the operator without potential,  the diffractive phenomenon of the wave on conical manifolds was studied by Cheeger and Taylor \cite{CT1,CT2}, and later was generalized to general cones with several conical ends by Ford and Wunsch \cite{FW}.  M\"uller and Seeger \cite{MS1} studied the regularity properties of wave propagation.
  For the case with the inverse-square potential, the asymptotic behavior of the Schr\"odinger propagator was considered in \cite{Carron,wang} and the Riesz transform was studied in \cite{HL}.  

There are also several other related studies on the pointwise decay estimates on cones in the literature. 
In \cite{SSS1, SSS2}, Schlag, Soffer and Staubach proved decay estimates (depending on the angular momentum) for Schr\"odinger and wave equations on manifolds with conical ends.
In \cite{KM}, Keeler and Marzuola studied the pointwise dispersive estimates (also depending on the angular momentum) for the Schr\"odinger equation on product cones, which are hard to sum in the
angular momentum. In \cite{Chen}, Chen proved the local-in-time dispersive and Strichartz estimates on a general conic manifold without conjugate points. We also refer the reader to the survey \cite{S} by Schlag for more about the dispersive estimates.
In particular, when $Y=\mathbb{S}_\sigma^1=\R/2\pi \sigma\Z$ with radius $\sigma>0$, 
this is close to the Euclidean cone of cone angle $2\pi\sigma$, $C_\alpha=[0,\infty)_r\times (\R/2\pi\sigma \Z)_\theta$.  This setting $X=C(\mathbb{S}^1_\sigma)$ is a 2D flat Euclidean cone, in which there are no conjugate points. 
 The difficulties in summing angular momentum are simplified by the straightforward structure of $Y=\mathbb{S}_\sigma^1$, in which  the eigenfunctions and eigenvalues on $Y$ are explicit. 
In \cite{Ford}, Ford proved the dispersive estimates \eqref{est:classical-dis} for the Schr\"odinger equation on the flat cone $C(\mathbb{S}^1_\sigma)$. For the wave equation on $C(\mathbb{S}^1_\sigma)$,  Blair, Ford and Marzuola \cite{BFM} proved the decay estimates for $\sin(t\sqrt{\Delta_g})/\sqrt{\Delta_g}$  while  in \cite[(1.7), Conjecture 1.1]{BFM} they conjectured a pointwise decay estimate for $\cos(t\sqrt{\Delta_g})$.  Very recently, the last author \cite{Z} constructed the Schwartz kernels of resolvent and spectral measure for the Laplacian on the 2D flat Euclidean cone, and proved the dispersive estimates for the Schr\"odinger and half-wave propagators,
which verifies \cite[(1.7), Conjecture 1.1]{BFM} for wave and provides a simple proof of the results in \cite{Ford} for Schr\"odinger. However, since the pointwise dispersive decay estimates are very sensitive to the geometric properties and the scaling critical potential perturbation, to the best of our knowledge, there are few results about the pointwise decay estimates in a general conical setting.

Motivated by this observation, we aim to study the pointwise decay estimates for the dispersive equations associated with the operator $H$ on the product cone $X=C(Y)$ with a more general closed manifold $Y$. One challenge is the potential presence of conjugate points within our general cone settings. 
In view of the conjugate points, Hassell and Wunsch \cite{HW1} pointed out that the Schr\"odinger propagator $U(t)(z_1,z_2)$ may fail to satisfy the classical pointwise dispersive estimate $|U(t)(z_1,z_2)|\leq C|t|^{-\frac n2}$ at some
pair of conjugate points. In addition, as mentioned above, the perturbation of the inverse-square potential is non-trivial since
the inverse-square decay of the potential has the same scaling as the Laplacian operator.  Fanelli, Felli, Fontelos and Primo  \cite{FFFP} proved a weighted decay estimate when $V_0(y)\equiv a\in [-1/4, 0)$ on $\R^3$, and they also addressed an open problem about decay estimates for more general $V_0(y)$ and higher dimensions $n\geq4$ in \cite[Remark 1.12]{FFFP}.

  \vspace{0.2cm}

In this paper, we focus on a general product cone $X=C(Y)$ over $Y$ whose conjugate radius $\conR>\pi$, where the conjugate radius $\conR$ is defined by
\begin{equation*}
\conR=\inf \{d(y_1,y_2): \text{$(y_1, y_2)$ are conjugate point pairs}\}.
\end{equation*}
When there are no conjugate point pairs, we set $\conR = +\infty$. Notice that we have
\begin{equation} \label{eq: cojR vs injR}
\conR \geq \mathrm{inj}(\CS),
\end{equation}
where $\mathrm{inj}(\CS)$ is the injectivity radius of $\CS$. The possible strict inequality is because $\conR$ only detects when the exponential map degenerates and allows it to be a covering map, while $\mathrm{inj}(\CS)$ requires the injectivity of the exponential map as well. 
For example, when $Y$ is a flat torus (or any other compact manifold with non-positive sectional curvature), $\mathrm{inj}(Y)$ is finite while $\conR$ is infinite.

In fact, we expect the dispersive estimate \eqref{est:dispersive} below to fail generically in its current form when $\conR < \pi$. This is because the geodesic flow on $X$ is expected to govern the propagation phenomena of $\Delta_g$. Thus the dichotomy according to the existence of a conjugate point pair within distance $\pi$ can be seen from the structure of the geodesic flow on metric cones. Let $x = r^{-1}$ and $y$ still be a coordinate system on $Y$. Suppose  (see \cite[Section~2, 3]{MZ96} for more details)
\begin{equation*}
(x,y,\tau,\mu)
\end{equation*}
are coordinates of the scattering cotangent bundle $^{\mathrm{sc}}T^*X$ of $X$, then the rescaled geodesic flow of $g=\frac{dx^2}{x^4}+\frac{h}{x^2}$ takes the form:
\begin{equation}\begin{aligned}\label{eq:conic-bichar}
    &    x=\frac{x_0}{\sin s_0}\sin (s+s_0),\ \tau=\cos (s+s_0),\ |\mu|=\sin (s+s_0),\\
    & (y,\hat\mu)=\exp(sH_{\frac{1}{2} h^{-1} })(y_0,\hat\mu_0),\ s\in(-s_0,-s_0+\pi),
\end{aligned}\end{equation}
where $\hat{\mu} = |\mu|_{h^{-1}}^{-1}\mu$. In particular, this rescaled flow has a global source-sink structure with the location where $s+s_0=0$ being the source and the location where $s+s_0=\pi$ being the sink. 
The important feature of this rescaling is that on the one hand the flow has unit speed on $Y$, while on the other hand the entire travel time of this flow is always $\pi$. Thus, the geometric information on $Y$ that can be detected through the geodesic flow on $X$ is `within distance $\pi$'. And the geometric information is $\exp(sH_{ \frac{1}{2} h^{-1} })$, whose non-degeneracy, which is equivalent to our no conjugate point assumption, is crucial in the Hadamard parametrix construction. Though one can still construct a parametrix in the presence of conjugate points in the calculus of Lagrangian distributions, and this degeneracy is harmless to $L^2-$based estimates, it is a general phenomenon in the theory of the boundedness of Fourier integral operators that this type of degeneracy (which essentially is the degeneracy of the projection from the Lagrangian submanifold defined in \eqref{eq: propagating lagrangians} to the base manifold) is fatal to general $L^p-$estimates. In the case where $Y=\mathbb{S}^{n-1}_{\sigma}$ with $n \geq 3$, Taira \cite{Ta} subsequently proved that the dispersive estimates fail when the radius $\sigma<1$, which corresponds to a conjugate radius $\conR<\pi$. 
Concretely, with $\sigma<1, n \geq 3$, the estimate of the kernel $e^{itH}$ will have a loss of the form $(\frac{t}{r_1r_2})^{\frac{n-2}{2}}$ when $\frac{t}{r_1r_2} \gg 1$, which in turn causes loss in the long time decay rate and also introduces spatial weights in the dispersive estimate.
When $n=2$, the work of Ford \cite{Ford} mentioned above shows that there is no loss in the estimate of $e^{itH}$ and the dispersive estimate. Geometrically this corresponds to the fact that there are no conjugate points on $\mathbb{S}^{1}$ and $\conR = +\infty$. 
Our results below can't recover the result of Ford due to technical reasons. For example, the sum of the right hand side of \eqref{beta-j} (and other analogous terms in the proof) does not converge when $n=2$.

\vspace{0.2cm}
 Now we state our main results.
\begin{theorem} [Pointwise estimates for Schr\"odinger propagator] 
\label{thm1:dispersiveS}
Let $z_1=(r_1,y_1)$ and $z_2=(r_2,y_2)$ be in the product cone $X=C(Y)$ of dimension $n\geq 3$ and let $H$ be the Schr\"odinger operator given in \eqref{oper:S}, where $V_0(y)\in\CC^\infty(Y)$ is such that $P=\Delta_h+V_0(y)+(n-2)^2/4$ is a strictly positive operator on $L^2(Y)$. Assume that the conjugate radius $\conR$ of $Y$ satisfies $\conR>\pi$. Then for $t\neq 0$, the Schwartz kernel of the Schr\"odinger propagator  $e^{itH}(z_1,z_2)$ satisfies
\begin{equation}\label{est:dispersive}
 \begin{split}
\big| e^{itH}(z_1,z_2)\big|\leq C|t|^{-\frac n2}\times \begin{cases} \Big(\frac{r_1r_2}{2|t|}\Big)^{-\frac{n-2}2+\nu_0},\quad & \frac{r_1r_2}{2|t|}\lesssim 1;\\
1,\qquad \quad  &\frac{r_1r_2}{2|t|}\gg1,
\end{cases}
\end{split}
\end{equation}
where $\nu_0$ is the positive square root of the smallest eigenvalue of the positive operator $P$ on the closed manifold $Y$.
\end{theorem} 

\begin{remark} In particular, the result applies when $Y$ is a sphere with radius larger than $1$, or any closed Riemannian manifold with non-positive sectional curvature, or their product.
\end{remark}

\begin{remark} It would be interesting to study the same problem when $Y$ is the unit sphere $\mathbb{S}^{n-1}$ whose conjugate radius equals $\pi$. This is closely related to the Schr\"odinger operator with inverse-square potentials $-\Delta+V_0(y)r^{-2}$ (where $y\in \mathbb{S}^{n-1}$) in the Euclidean space $\R^n$. There is an analogue of the open problem addressed in \cite[Remark 1.12]{FFFP}.
Although the global pointwise dispersive estimate is expected to fail generically in this case, we are still able to prove microlocalized decay estimates and global Strichartz estimates for the scaling critical electromagnetic Schr\"odinger equation in \cite{JZ2}.
\end{remark}

For the discussion below, it is convenient to introduce a different parameterization of the operator $H$
\begin{equation}\label{def:alpha}
\alpha=-(n-2)/2+ \nu_0,
\end{equation}
where $\nu_0$ is given in Theorem \ref{thm1:dispersiveS}, the positive square root of the smallest eigenvalue of the positive operator
$P=\Delta_h+V_0(y)+(n-2)^2/4$ on the closed manifold $Y$. Define 
\begin{equation}\label{def:q-alpha}
q(\alpha)=
\begin{cases}
\infty,\quad \alpha\geq 0;\\
-\frac{n}{\alpha}, \quad -(n-2)/2<\alpha< 0,
\end{cases}
\end{equation}
and let $q'(\alpha)$ be the dual exponent of  $q(\alpha)$ such that
$$\frac1{q(\alpha)}+\frac1{q'(\alpha)}=1.$$
\vspace{0.2cm}

As a direct consequence of Theorem \ref{thm1:dispersiveS}, we have the following results.
\begin{corollary}\label{cor:L1Linfty}Let $\alpha$ be given in \eqref{def:alpha} and $t\neq 0$. If $\alpha\geq0$, then there exists a constant $C$ such that
\begin{equation}\label{est:dis-cl}
\| e^{itH}\|_{L^1(X)\to L^\infty(X)}\leq C |t|^{-\frac n2},
\end{equation}
and
\begin{equation}\label{est:dis-weight}
\| r_1^{-\alpha}e^{itH}r_2^{-\alpha}\|_{L^1(X)\to L^\infty(X)}\leq C |t|^{-\frac n2-\alpha}.
\end{equation}
If $-\frac{n-2}2<\alpha <0$, then
\begin{equation}\label{est:dis-weight'}
\| (1+r_1^{\alpha})^{-1}e^{itH}(1+r_2^{\alpha})^{-1}\|_{L^1(X)\to L^\infty(X)}\leq C |t|^{-\frac n2}(1+|t|^{-\alpha}).
\end{equation}
\end{corollary}

\begin{remark} If the potential $V_0$ is positive, then $\alpha\geq0$, hence one has the classical dispersive estimates \eqref{est:dis-cl} and gains more decay in \eqref{est:dis-weight} by compensating with some weight. 
\end{remark}

\begin{theorem}\label{thm:LqLq'} Let $\alpha$ be given in \eqref{def:alpha} and $t\neq 0$. If $\alpha\geq0$, then there exists a constant $C$ such that
\begin{equation}\label{est:LqLq'}
\| e^{itH}\|_{L^{q'}(X)\to L^q(X)}\leq C |t|^{-\frac n2(1-\frac2q)},\quad q\in [2, +\infty].
\end{equation}
If $-\frac{n-2}2<\alpha <0$, then
\begin{equation}\label{est:LqLq'0}
\| e^{itH}\|_{L^{q'}(X)\to L^q(X)}\leq C |t|^{-\frac n2(1-\frac2q)},\quad q\in [2, q(\alpha)).
\end{equation}
\end{theorem}

\begin{remark} 
The first estimate \eqref{est:LqLq'} has been proved by directly interpolating \eqref{est:dis-cl} and the $L^2$-estimates.
In contrast to the direct interpolation result, the second estimate \eqref{est:LqLq'0} is improved by removing the weight. Thus, for \eqref{est:LqLq'0}, we need an additional argument beyond interpolation, see Proposition \ref{prop:est-qq'}.
\end{remark}

\begin{remark} An analogue of \eqref{est:LqLq'0} was proved by Miao, Su and Zheng \cite{MSZ} for the Schr\"odinger operator with inverse-square potentials $-\Delta+V_0(y)r^{-2}$ with $Y= \mathbb{S}^{n-1}$ and $V_0(y)\equiv a\in [-(n-2)^2/4,0)$ in the Euclidean space $\R^n$. 
\end{remark}

\begin{remark} 
One can produce the Strichartz estimates by using the above decay estimates and Keel-Tao's abstract methods in \cite{KT}. The Strichartz estimates for the Schr\"odinger and wave equations in a general conical setting (without assumption on the conjugate radius of $Y$) have been proved by Zheng and the last author in \cite{ZZ1, ZZ2}. The method used to study the pointwise decay estimates here is quite different from the one therein. \end{remark}

Next we state our results for the wave equation. \vspace{0.1cm}

 Let $\varphi\in C_c^\infty(\mathbb{R}\setminus\{0\})$, with $0\leq\varphi\leq 1$, $\text{supp}\,\varphi\subset[1/2,2]$, and
\begin{equation}\label{LP-dp}
\sum_{j\in\Z}\varphi(2^{-j}\lambda)=1,\quad \varphi_j(\lambda):=\varphi(2^{-j}\lambda), \, j\in\Z,\quad \phi_0(\lambda):=\sum_{j\leq0}\varphi(2^{-j}\lambda).
\end{equation}

\begin{definition}[Besov spaces associated with $H$]\label{def:besov} For $s\in\R$ and $1\leq p,r<\infty$, the homogeneous Besov norm  $\|\cdot\|_{\dot{\mathcal{B}}^s_{p,r}(X)}$ is defined by
\begin{equation}\label{Besov}
\|f\|_{\dot{\mathcal{B}}^s_{p,r}(X)}=\Big(\sum_{j\in\Z}2^{jsr}\|\varphi_j(\sqrt{H})f\|_{L^p(X)}^r\Big)^{1/r}.
\end{equation}
In the case $p=r=2$, we denote the Sobolev norm
\begin{equation}\label{Sobolev1}
\begin{split}
\|f\|_{\dot{\mathcal{H}}^s(X)}:=\|f\|_{\dot{\mathcal{B}}^s_{2,2}(X)}.
\end{split}
\end{equation}
\end{definition}

\begin{theorem} [Decay estimates for half-wave propagator]\label{thm2:dispersiveW}
Let $z_1=(r_1,y_1)$ and $z_2=(r_2,y_2)$ be in the product cone $X=C(Y)$ of dimension $n\geq 3$ and let $H$ be the Schr\"odinger operator of Theorem \ref{thm1:dispersiveS}. Assume that the {conjugate radius} $\conR$ of $Y$ satisfies $\conR>\pi$. Then for $t\neq0$, there exists a constant $C$ such that
\begin{equation}\label{dis-w}
\|e^{it\sqrt{H}}f\|_{L^\infty(X)}\leq C  |t|^{-\frac{n-1}2}\|f\|_{\dot{\mathcal{B}}^{\frac{n+1}2}_{1,1}(X)},
\end{equation}
provided  that $\alpha\geq 0$. If $-(n-2)/2<\alpha<0$, for $2\leq q<q(\alpha)$, then
\begin{equation}\label{dis-w'}
\|e^{it\sqrt{H}}f\|_{L^q(X)}\leq C  |t|^{-\frac{n-1}2(1-\frac2q)}\|f\|_{\dot{\mathcal{B}}^{\frac{n+1}2(1-\frac2q)}_{q',2}(X)}.
\end{equation}
\end{theorem} 

\begin{remark} 
When $Y=\mathbb{S}_\sigma^1$, in which there are no conjugate points, Blair, Ford and Marzuola \cite{BFM} proved the decay estimates for $\sin(t\sqrt{\Delta_g})/\sqrt{\Delta_g}$,  while  in \cite[(1.7), Conjecture 1.1]{BFM} they conjectured a pointwise decay estimate for $\cos(t\sqrt{\Delta_g})$. This result extends their result to the half-wave operator $e^{it\sqrt{H}}$ in higher dimensions. 
\end{remark}

The structure of the paper is as follows. Section \ref{sec:r-s} is devoted to the construction of the kernel of the Schr\"odinger propagator, while in Section \ref{sec: parametrix} we provide the proof of the parametrix construction. In Section \ref{sec: proof of dispersive S}, we prove the main Theorem \ref{thm1:dispersiveS}.  The Littlewood-Paley theory associated with the Schr\"odinger operator $H$ is established in Section \ref{sec:LP} and the decay estimates in Corollary \ref{cor:L1Linfty} and Theorem \ref{thm:LqLq'} are proved in Section \ref{sec:decayS}. Finally, we prove the decay estimates for wave in Section \ref{sec:decayW}.
Without loss of generality, in the rest of the paper, we assume $t>0$. 
\medskip

\textbf{Acknowledgments.}  The authors would like to thank Andrew Hassell for his
helpful discussions and encouragement. The last author is grateful for the hospitality of the Australian National University when he was visiting Andrew Hassell at ANU.
J. Zhang was supported by the Beijing Natural Science Foundation (1242011) and National Natural Science Foundation of China (12171031, 12531005);
Q. Jia was supported by the Australian Research Council through grant FL220100072.

\section{The construction of the Schr\"odinger propagator}\label{sec:r-s}

In this section, we construct the representation of the Schr\"odinger propagator inspired by Cheeger-Taylor \cite{CT1,CT2}. More precisely, we prove
\begin{proposition}[Schr\"odinger kernel]\label{prop:Sch-pro} Let $H$ be the Schr\"odinger operator given in \eqref{oper:S} and let $z_1=(r_1,y_1)\in X$ and $z_2=(r_2,y_2)\in X$. 
Then the kernel of the Schr\"odinger propagator can be written as
\begin{equation}\label{S-kernel} 
\begin{split}
e^{itH}(z_1,z_2)&=e^{itH}(r_1,y_1,r_2, y_2)\\
 &=\big(r_1 r_2\big)^{-\frac{n-2}2}\frac{e^{-\frac{r_1^2+r_2^2}{4it}}}{2it}
  \Big(\frac1{\pi}\int_0^\pi e^{\frac{r_1r_2}{2it} \cos(s)} \cos(s\sqrt{P})(y_1,y_2) ds\\
  &-\frac{\sin(\pi\sqrt{P})}{\pi}\int_0^\infty e^{-\frac{r_1r_2}{2it} \cosh s} e^{-s\sqrt{P}}(y_1,y_2) ds\Big),
\end{split}
\end{equation}
where $P=\Delta_h+V_0(y)+(n-2)^2/4$.

\end{proposition}

\begin{proof} We construct the Schr\"odinger propagator by using Cheeger's functional calculus for which we refer to \cite{CT1, Taylor}. 
We write the Schr\"odinger operator on $X$ 
\begin{equation*}
H=\Delta_g+V_0(y)r^{-2}=-\partial_r^2-\frac{n-1}r\partial_r+\frac{\Delta_{h}+V_0(y)}{r^2},
\end{equation*}
where $\Delta_{h}$ is the Laplacian operator on $Y$. Since $Y$ is a compact Riemannian manifold, by the spectral theory,
there exist discrete eigenvalues $\mu_k$ and eigenfunctions $\varphi_k(y)$ of the operator $\Delta_{h}+V_0(y)$ such that
\begin{equation}\label{eig-eq}
\big(\Delta_{h}+V_0(y)\big)\varphi_{k}(y)=\mu_k\varphi_{k}(y),\quad k\in\mathbb{N}:=\{0,1,2,\ldots\},
\end{equation}
where we repeat each eigenvalue as many times as its multiplicity.
Define 
\begin{equation}\label{nu-k}
\nu_k=\sqrt{\mu_k+(n-2)^2/4},
\end{equation}
then 
\begin{equation}
P\varphi_{k}(y)=\nu_k^2\varphi_{k}(y),\quad P=\Delta_h+V_0(y)+\frac{(n-2)^2}4,\quad k\in\mathbb{N}=\{0,1,2,\ldots\}.
\end{equation}
By Cheeger's separation of variables functional calculus (e.g. \cite[Chapter~8 (8.47)]{Taylor}),  we obtain the kernel $K(t,z_1,z_2)$ of the operator  $e^{itH}$
\begin{equation}\label{ker:S}
\begin{split}
K(t,z_1,z_2)&=K(t,r_1,y_1,r_2, y_2)\\
&=\big(r_1 r_2\big)^{-\frac{n-2}2}\sum_{k\in\mathbb{N}}\varphi_{k}(y_1)\overline{\varphi_{k}(y_2)}K_{\nu_k}(t,r_1,r_2),
\end{split}
\end{equation}
where $\overline{\varphi}_k$ denotes the complex conjugate of the eigenfunction $\varphi_k$ and
\begin{equation} \label{eq:def-Knuk}
\begin{split}
  K_{\nu_k}(t,r_1,r_2)&=\int_0^\infty e^{it\rho^2}J_{\nu_k}(r_1\rho)J_{\nu_k}(r_2\rho) \,\rho d\rho.
  \end{split}
\end{equation}
By using spectral theory, if $F$ is a Borel measurable function, we identify the operator with its kernel as in \cite{Taylor} to obtain 
\begin{equation}\label{FA}
F(\sqrt{P})=\sum_{k\in\mathbb{N}}F(\nu_k) \varphi_{k}(y_1)\overline{\varphi_{k}(y_2)},
\end{equation}
which gives an operator on $Y$.
If we view $K_{\nu}(t,r_1,r_2)$ defined by (in the same way as \eqref{eq:def-Knuk} above)
\begin{equation}\label{eq:Knu-def}
\begin{split}
  K_{\nu}(t,r_1,r_2):
  &=\big(r_1 r_2\big)^{-\frac{n-2}2}\int_0^\infty e^{it\rho^2}J_{\nu}(r_1\rho)J_{\nu}(r_2\rho) \,\rho d\rho
  \\&=\big(r_1 r_2\big)^{-\frac{n-2}2}\lim_{\epsilon \searrow 0} \int_0^\infty e^{(-\epsilon+it)\rho^2}J_{\nu}(r_1\rho)J_{\nu}(r_2\rho) \,\rho d\rho
  \end{split}
\end{equation}
as a function of $\nu$ and take $\nu=\sqrt{P}=\sqrt{\Delta_h+V_0(y)+\frac{(n-2)^2}4}$ in the sense of \eqref{FA} (i.e. restrict to each $\nu_k$, multiply by $\varphi_k(y_1)\overline{\varphi}_k(y_2)$ and take the sum), then we obtain the kernel $K(t,z_1,z_2)$ in \eqref{ker:S}.

By using Weber's second exponential integral \cite[Section 13.31 (1)]{Watson}, we have, for $\epsilon>0$ 
\begin{equation}\label{Weber}
\begin{split}
 \int_0^\infty  e^{(-\epsilon+it)\rho^2}  J_{\nu}(r_1\rho) J_{\nu}(r_2\rho)\rho d\rho =\frac{e^{-\frac{r_1^2+r_2^2}{4(\epsilon-it)}}}{2(\epsilon-it)} I_\nu\big(\frac{r_1r_2}{2(\epsilon-it)}\big),
\end{split}
\end{equation}
where $I_\nu(x)$ is the modified Bessel function of the first kind
\begin{equation*}
\begin{split}
I_\nu(x)=\sum_{j=0}^\infty \frac{1}{j!\Gamma(\nu+j+1)}\big(x/2\big)^{\nu+2j}.
\end{split}
\end{equation*}
We have two ways to see the Schr\"odinger kernel \eqref{ker:S}. On one hand, from \eqref{ker:S} and \eqref{Weber}, we have 
\begin{equation}\label{ker:S1}
\begin{split}
K(t,z_1,z_2)&=\big(r_1 r_2\big)^{-\frac{n-2}2}\sum_{k\in\mathbb{N}}\varphi_{k}(y_1)\overline{\varphi_{k}(y_2)}\lim_{\epsilon\searrow 0} 
\frac{e^{-\frac{r_1^2+r_2^2}{4(\epsilon-it)}}}{2(\epsilon-it)} I_{\nu_k}\big(\frac{r_1r_2}{2(\epsilon-it)}\big)\\
&=\big(r_1 r_2\big)^{-\frac{n-2}2}\frac{e^{-\frac{r_1^2+r_2^2}{4it}}}{2it} \sum_{k\in\mathbb{N}}\varphi_{k}(y_1)\overline{\varphi_{k}(y_2)}
(-i)^{\nu_k} J_{\nu_k}\big(\frac{r_1r_2}{2t}\big),
\end{split}
\end{equation}
where we use the formula $I_\nu(-ix)=(-i)^\nu J_\nu(x)$. Even though \eqref{ker:S1} is not needed for the proof of \eqref{S-kernel}, we record it here for the purposes of the subsequent sections. 

Define $$z_\epsilon=\frac{r_1r_2}{2(\epsilon-it)}, \quad \epsilon>0, $$
and recall the integral representation (see \cite[p. 181]{Watson} or \cite[III, p. 186]{MUH}) of the modified Bessel function 
\begin{equation*}
I_\nu(z)=\frac1{\pi}\int_0^\pi e^{z\cos(s)} \cos(\nu s) ds-\frac{\sin(\nu\pi)}{\pi}\int_0^\infty e^{-z\cosh s} e^{-s\nu} ds,
\end{equation*}
then
\begin{equation}
\begin{split}
 & K_{\nu}(t,r_1,r_2)=\big(r_1 r_2\big)^{-\frac{n-2}2}\lim_{\epsilon\searrow 0}\frac{e^{-\frac{r_1^2+r_2^2}{4(\epsilon-it)}}}{2(\epsilon-it)} I_\nu\big(\frac{r_1r_2}{2(\epsilon-it)}\big)\\
  &=\big(r_1 r_2\big)^{-\frac{n-2}2}\frac{e^{-\frac{r_1^2+r_2^2}{4it}}}{2it}
  \Big(\frac1{\pi}\int_0^\pi e^{ \frac{r_1r_2}{2it} \cos(s)} \cos(\nu s) ds\\
  &\qquad\qquad\qquad\qquad\qquad\qquad-\frac{\sin(\nu\pi)}{\pi}\int_0^\infty e^{-\frac{r_1r_2}{2it}  \cosh s} e^{-s\nu} ds\Big).
  \end{split}
\end{equation}
This implies \eqref{S-kernel} by the discussion about \eqref{eq:Knu-def}, after taking $\nu=\sqrt{P}$
in the sense of \eqref{FA}. 
\end{proof}

\section{The parametrix construction} \label{sec: parametrix}
In the previous section, we have shown the propagator on $X$ in terms of $\sqrt{P}=\sqrt{\Delta_h+V_0(y)+\frac{(n-2)^2}4}$, which is an operator on $\CS$.
In this section, we will construct the parametrices for the even wave propagator $\cos(s\sqrt{P})$ and the Poisson wave propagator $e^{(-s\pm i\pi)\sqrt{P}}$ in the terminology of Zelditch in \cite{Zelditch}. 
The construction is essentially the Hadamard parametrix construction, but the main point proven is that they can be represented as oscillatory integrals with a specific phase function and symbolic amplitudes, so that the singularities of $e^{\pm is\sqrt{P}}$ at $s=\pi$ and $e^{(-s\pm i\pi)\sqrt{P}}$ at $s=0$ are cancelled, and such a representation is needed in the proof of our main theorem. 

We observe that while the usual Hadamard parametrix for  $\cos(s\sqrt{P})$ is only valid for $|s|<\mathrm{inj}(Y)$, our analysis requires this approximation to remain valid up to timescales of order $\pi$, where the influence of closed geodesics (or loops) in $Y$ must be considered under our assumption that the conjugate radius $\conR$ of  $Y$ satisfies $\conR>\pi$. To this end, we will recall some geometric facts in Section~\ref{subsec:parametrix-geometry} and then give the oscillatory integral representation of parametrices in Section~\ref{subsec:parametrix-main}.

As one can see, the complication in Section~\ref{subsec:parametrix-geometry} is caused by the (topological) obstruction for the exponential map to be a global diffeomorphism within a $\pi$-geodesic ball, so readers who only wish to apply the result to $Y$ with $\mathrm{inj}(Y)> \pi$ (recalling the comparison after \eqref{eq: cojR vs injR}) can skip this part.

\subsection{Geometric preliminaries} \label{subsec:parametrix-geometry}

We recall some geometric facts that we need in our parametrix construction.
We take the symplectic (instead of Riemannian) perspective to view the geodesic flow of $Y$ as a flow on $T^*Y$, which is the Hamilton flow associated with (the symbol of) $\Delta_h$. Also, we consider $h$ (in fact the inverse of the original metric, when realized as matrices) as a function on $T^*Y$ that is quadratic in the fiber and the exponential map as a flow defined on $T^*Y$.

We first introduce a notion that characterizes our exponential map restricted to the region on which it is non-degenerate. 
\begin{definition} \label{def: local covering}
Suppose $(\tilde{N},\tilde{h}),\; (N,h)$ are Riemannian manifolds and
\begin{equation}
f: \tilde{N} \to N,
\end{equation} 
we say that $f$ is a local covering map if it has the following properties:
\begin{itemize}
\item $\tilde{h} = f^*h$, where $f^*$ is the pullback of the map $f$.

\item For each $y \in N$, there is a neighborhood $U_y$ such that 
\begin{equation}
f^{-1}(U_y) = \cup_{\gamma \in \Gamma_y} \tilde{U}_{\gamma}, 
\end{equation}
where $\Gamma_y$ is an index set (allowed to be empty when $y \notin f(\tilde{N})$) such that $|\Gamma_y| \leq C$ for a constant independent of $y$, and $f$ restricted to each $\tilde{U}_{\gamma}, \gamma \in \Gamma_y$ is a diffeomorphism onto $U_y$.
\end{itemize}

\end{definition}

Now we state a modified version of the well-known Cartan-Hadamard theorem. This is an observation inspired by \cite[Chapter~4]{Keeler}.

\begin{proposition}[Modified Cartan-Hadamard Theorem]
\label{prop:modified-Cartan-Hadamard}
Let $Y$ be as above (in particular it satisfies $\conR > \pi$), and take $\epsilon>0$ such that $\pi+2\epsilon<\conR$. 
Then for any $y_0 \in Y$, the exponential map, identified as a map
\begin{equation}
\exp_{y_0}|_{\mathsf{B}} : \; \mathsf{B} \to Y
\end{equation}
is a local covering map, where $\mathsf{B}$ is the open ball in $T^*_{y_0}Y$ centered at the origin with radius $\pi+\epsilon$ measured using $\exp_{y_0}^*h$. (Notice that this is the same ball as using $h_{y_0}$ since they coincide in the radial direction.)
\end{proposition}

\begin{remark}
Recalling the definition of $\Gamma_y$ in Definition~\ref{def: local covering}, we should emphasize that in the current setting we won't have the even cover property (that is, $|\Gamma_y|$ being the same for all $y$) in general because we only considered a `truncated' covering space. For example, one can consider a two-dimensional torus with very different radii of its two circles (the vertical one and the horizontal one).
\end{remark}

\begin{proof}

By our choice of $\mathsf{B}$, $\exp_{y_0}|_{\mathsf{B}}$ has non-degenerate differential everywhere, hence the claim that it is a local diffeomorphism follows from the inverse function theorem.

Next we show the uniform boundedness of $|\Gamma_y|$.
For any $\tilde{y} \in \overline{\mathsf{B}}$, there is a neighborhood $U_{\tilde{y}}$ of it such that the exponential map is injective on $U_{\tilde{y}}$. In addition, by compactness we can use finitely many (say, $N$ of) $U_{\tilde{y}}$ to cover $\mathsf{B}$. This means that each point $y$ has at most $N$ preimages under $(\exp_{y_0}|_{\mathsf{B}})^{-1}$, which gives the uniform boundedness of $|\Gamma_y|$.
\end{proof}

\subsection{The parametrices} \label{subsec:parametrix-main}

Now we turn to the oscillatory integral representation of parametrices.
We first recall some basic facts about the propagator $e^{\mp is\sqrt{P}}$ and Fourier integral operators.
It is well-known (see \cite{HorVol4}) that $e^{\mp is\sqrt{P}}$ are Fourier integral operators associated with the 
propagating Lagrangian submanifolds $\mathscr{L}_\pm$ given by
\begin{align}  \label{eq: propagating lagrangians}
\begin{split}
\mathscr{L}_{\pm}:= & \{ (s,y_1,y_2,\tau,\mu_1,-\mu_2) \in T^*(\mathbb{R} \times Y \times Y): 
\\& \tau = \mp |\mu_1|_h, (y_1,\mu_1) = \exp(\pm s\mathsf{H}_{p})(y_2,\mu_2) \}.
\end{split}
\end{align}
Here we use $p = |\mu|^2_h$ to denote the homogeneous principal symbol of $P$, and
\begin{equation}
\mathsf{H}_{p} = (2|\mu|_h)^{-1}H_{p}
\end{equation}
is the rescaled Hamilton vector field.

If we use $I^{m}(\mathbb{R} \times Y \times Y, \mathscr{L}_\pm)$ to denote the $m$-th order Fourier integral operators associated with $\mathscr{L}_\pm$ respectively, which are operators with kernels that can be written as 
 an oscillatory integral of the form
\begin{equation}
\int_{\R^N} e^{i \phi(s,y_1,y_2,\theta)}a(s,y_1,y_2,\theta) d\theta, \quad \theta \in \R^N,
\end{equation}
with $\phi$ parametrizing $\mathscr{L}_\pm$ (in the sense of \cite[Definition~21.2.15]{HorVol3}), then $a \in S^{m+\frac{1+(n-1)+(n-1)}{4}-\frac{N}{2}}(\R \times Y \times Y \times \R^N)$. We refer to \cite[Proposition~25.1.5]{HorVol4} for details of this numerology. 

Then we have
\begin{equation}\label{claim:FIO}
e^{\mp is \sqrt{P}} \in I^{-\frac{1}{4}}(\mathbb{R} \times Y \times Y, \mathscr{L}_\pm).
\end{equation}
See \cite[Section~4.1]{sogge} for details. The order $-\frac{1}{4}$ of \eqref{claim:FIO} means that we can write it in terms of momenta $\xi \in \R^{n-1}$ as `$\theta$' with $N=n-1$, so we should have amplitude $a \in S^0( (\mathbb{R} \times Y \times Y) \times \mathbb{R}^{n-1})$ since $0 =-\frac{1}{4}+\frac{1+(n-1)+(n-1)}{4}-\frac{n-1}{2}$. 

For $(y_1,y_2) \in Y \times Y$ and $\hat{\mu}=\mu|\mu|^{-1}$, we define the forward/backward distance spectrum associated with $(y_1,y_2) \in Y \times Y$ to be
\begin{align} \label{eq:defn-distance-spectrum}
\begin{split}
\mathfrak{D}_\pm (y_1,y_2) = & \{ \mathfrak{d} \in [0,\pi+\epsilon): 
\exists \hat{\mu}_2 \in S^*_{y_2}Y, \;\hat{\mu}_1 \in S^*_{y_1}Y \text{ such that }
\\ & \exp(\pm \mathfrak{d}\mathsf{H}_{p})(y_2,\hat{\mu}_2) = (y_1,\hat{\mu}_1) \},
\end{split}
\end{align}
which is a collection of smooth functions $\mathfrak{d} (y_1,y_2)$ of $y_1,y_2$. Here we count $\mathfrak{d}$ with multiplicity for different $\hat{\mu}_2$ (and corresponding $\hat{\mu}_1$). 
Since we allow the momentum to run over the entire $S^*_{y_2}Y$, and the forward $\mathsf{H}_{p}$-flow starting at $(y_2,\hat{\mu}_2)$ is the same as the backward flow starting at $(y_2,-\hat{\mu}_2)$, $\mathfrak{D}_+(y_1,y_2)$ is actually the same as $\mathfrak{D}_-(y_1,y_2)$. We only keep the $\pm$ sign to emphasize which of $e^{\pm is\sqrt{P}}$ we are considering 
and we will denote it by $\mathfrak{D}(y_1,y_2)$ when we consider the cosine propagator.

Equivalently, these are all $\mathfrak{d} \in [0,\pi+\epsilon)$ such that there is a (unit speed) geodesic $\gamma$ (with loops counted with multiplicity) starting at $y_2$ with $\gamma(\mathfrak{d}) = y_1$.
In particular, when $d_h(y_1,y_2)<\mathrm{inj}(Y)$ (hence $d_h(y_1,y_2)$ is smooth and realized by the unique distance minimizing geodesic), $d_h(y_1,y_2) \in \mathfrak{D}_\pm(y_1,y_2)$.

Let $\mathsf{B}$ be the $(\pi+\epsilon)$-ball in $T_{y_2}^*Y$ with $\epsilon$ as in Proposition~\ref{prop:modified-Cartan-Hadamard}. Next we show that the injectivity radius of 
\begin{equation} \label{eq:tilde-h-def}
\tilde{h}:=\exp_{y_2}^*h
\end{equation}
at $0$ is at least $\conR>\pi+\epsilon$. Here $\tilde{h}$ remains non-degenerate on $\mathsf{B}$ since the differential of $\exp_{y_2}$ is non-degenerate by our assumption $\conR>\pi+\epsilon$.
Firstly, a tangent vector $v \in T_0T^*_{y_2}Y$ can be identified as a curve (which we choose to be a straight line, exploiting the fact that $T^*_{y_2}Y$ is a vector space) passing through $0$: $s \to s\xi$, $\xi \in T^*_{y_2}Y$ with tangent vector $v$ at $0$.
Then this gives a canonical identification $v \to \xi$ between $T_0T^*_{y_2}Y$ and $T^*_{y_2}Y$. So from now on we will view an element $v \in T_0T^*_{y_2}Y$ as an element in $T^*_{y_2}Y$ under this identification. 
Consider a curve in $T^*_{y_2}Y$ defined by 
\begin{align} \label{eq:tilde-gamma-v}
\tilde{\gamma}(s) = sv.
\end{align}
Then this is a geodesic with respect to $\tilde{h}$ because $\exp_{y_2}(\tilde{\gamma}(s)) = \exp_{y_2}(sv)$ is a geodesic with respect to $h$ on $Y$ and $\exp_{y_2}$ is a local isometry between $(\mathsf{B},\tilde{h})$ and $(Y,h)$ by our choice of $\tilde{h}$.
Evaluating \eqref{eq:tilde-gamma-v} at $s=1$, we know that the exponential map of $\tilde{h}$ at $0$ is precisely the identification $v \to \xi$ above, which is a diffeomorphism from the $(\pi+\epsilon)$-ball (centered at $0$ in $T_0T_{y_2}^*Y$) to $\mathsf{B}$, and we know that the injectivity radius of $\tilde{h}$ at $0$ is at least $\pi+\epsilon$.
By the definition of $\exp$, we know that those points in $(\exp_{y_2}|_{\mathsf{B}})^{-1}(y_1)$ have a one-to-one correspondence with geodesics connecting $y_1$ and $y_2$ with length less than $\pi+\epsilon$, part of which might be geodesic loops with multiplicity counted.
So by the argument about the boundedness of $|\Gamma_{y_1}|$ in Proposition~\ref{prop:modified-Cartan-Hadamard}, we know $|\mathfrak{D}(y_1,y_2)|$ is finite. 

Using $\mathfrak{d}_\pm(y_1,y_2)$ instead of the distance function $d_h$ avoids the issue of non-smoothness of the distance function when more than one geodesic meets at the same point. The singularity is formed because the distance function takes the minimum within $\mathfrak{D}_\pm (y_1,y_2)$.

Let
\begin{equation}
\mathscr{P}_\pm \; : \mathscr{L}_\pm \to  Y \times Y
\end{equation}
be the projection. Since we know $\conR>\pi+2\epsilon$, this projection is a local diffeomorphism (when restricted to each level set of $|\mu|$, since it is conic). And we can take those neighborhoods on which it is a diffeomorphism to have a lower bound on their size by compactness. 
So, for $(y_1,y_2) \in Y \times Y$, let $U$ be a small neighborhood of it. Then $\mathscr{P}_\pm^{-1}(U)$ is a union of a finite number of disjoint sheets in $\mathscr{L}_\pm$ over $U$ and by definition they have a one-to-one correspondence with geodesics connecting $y_1,y_2$ with length (counted in terms of the $\mathsf{H}_p$-flow). And they have a one-to-one correspondence with $\mathfrak{d} \in \mathfrak{D}_\pm(y_1,y_2)$.

Now we prove that, over the sheet in $\mathscr{P}_\pm^{-1}(U)$ corresponding to $\mathfrak{d}$, we can use 
\begin{equation}\label{phi-d}
\phi_{\mathfrak{d}}(s,y_1,y_2,\xi) = \mathfrak{d}(y_1,y_2) \hat{\mu}_2 \cdot \xi \mp s|\xi|, \quad \xi \in \R^{n-1}, \, \mathfrak{d} \in \mathfrak{D}_\pm(y_1,y_2), 
\end{equation}
to parametrize $\mathscr{L}_\pm$ in the sense of \cite[Definition~21.2.15]{HorVol3}. 
Here we fix a coordinate system on the $y_2$-component and then choose the coordinate system on $y_1$ so that $y_1-y_2$ is the geodesic normal coordinate centered at $y_2$. (It is not hard to verify that this $(y_1,y_2)$ has full-rank differential everywhere by writing it as $((y_1-y_2)+y_2,y_2)$).
Here $\hat{\mu}_2$ is the covector at $y_2$ determining the geodesic associated with $\mk{d}$ from $y_2$ to $y_1$.
For definiteness, we only consider the case with the $+$ sign below since the other case can be proven using the same argument.
We recall the definition of parametrizing a Lagrangian by a phase function here in our setting. 

Fixing $\mathfrak{d} \in \mathfrak{D}_+(y_1,y_2)$, this parametrization means that over the sheet corresponding to $\mathfrak{d}$, $\mathscr{L}_+$ is the image of the map
\begin{equation} \label{eq:Lambdapm-1}
\begin{split}
\Lambda_{+} &\rightarrow T^*(\mathbb{R} \times Y \times Y)\\
(s,y_1,y_2;\xi) &\rightarrow (s,y_1,y_2; d_{s,y_1,y_2}\phi_{\mathfrak{d}}),
\end{split}
\end{equation} 
on the critical set 
\begin{equation*}
\Lambda_{+}:=\{(s,y_1,y_2; \xi)\in (\mathbb{R} \times Y \times Y)\times((\R^{n-1})\setminus\{0\}): d_{\xi}\phi_{\mathfrak{d}}=0\}.
\end{equation*}
More concretely, we need to show that in suitable coordinates, we have
\begin{equation} \label{eq: L+-, defn2}
\begin{split}
\mathscr{L}_+= \big\{ &(s,y_1,y_2,\tau,\mu_1,-\mu_2) \in T^*(\mathbb{R} \times Y \times Y): \\
&\tau=d_s \phi_{\mathfrak{d}},\, \mu_2=-d_{y_2}\phi_{\mathfrak{d}},\, \mu_1=d_{y_1}\phi_{\mathfrak{d}}, \, d_{\xi}\phi_{\mathfrak{d}}=0 \big\}.
\end{split}
\end{equation}
Recalling \eqref{phi-d}, the condition $d_{\xi}\phi_{\mathfrak{d}}=0$ is equivalent to:
\begin{equation}
\mk{d}\hat{\mu}_2 - s \frac{\xi}{|\xi|} = 0.
\end{equation}
And this is in turn equivalent to
\begin{equation} \label{eq: critical set, +}
\xi = |\xi| \hat{\mu}_2, \; s= \mk{d}(y_1,y_2).
\end{equation}

Recalling the discussion about $\exp_{y_2}$ in Section~\ref{subsec:parametrix-geometry}, those sheets also correspond one-to-one to open neighborhoods of $\exp_{y_2}^{-1}(y_1)$ and we can use the coordinates in $T_{y_2}^*Y$ near the part that corresponds to this geodesic plus (just apply the addition in $\R^{n-1}$) as coordinates for the $y_1$-component, and this has smooth dependence on $y_2$ as well. 
Recalling our choice of coordinates after \eqref{phi-d}, we have
\begin{equation}\label{eq: y1 y2 relations}
y_1-y_2 = \mk{d}(y_1,y_2) \hat{\mu}_2 = s\hat{\mu}_2,\quad \hat{\mu}_2=\mu_2/ |\mu_2|\in \mathbb{S}^{n-2}.
\end{equation}


Now we verify that the characterization of momentum variables in \eqref{eq: L+-, defn2} coincides with that in \eqref{eq: propagating lagrangians}. Using the definition of $\mk{d}(y_1,y_2)$, which is locally a distance function, we have
\begin{equation} \label{eq: dy1,dy2}
d_{y_1}(\mk{d}(y_1,y_2)) = \hat{\mu}_1 =  \hat{\mu}_2, \quad d_{y_2} (\mk{d}(y_1,y_2)) = -  \hat{\mu}_2.
\end{equation}
One way to see this more explicitly is that, by looking at the $\partial_{y_1}$-component of $\mathsf{H}_p$, which is unit speed on $Y$, we know
\begin{align*}
\nabla_{y_1} \mk{d}(y_1,y_2) = (\sum_{j} h^{ij}\hat{\mu}_{2,j})_{1 \leq i \leq n-1},
\end{align*}
where the gradient is taken with respect to $h$. This implies the first equation in \eqref{eq: dy1,dy2} by the definition of the gradient and the second one follows from a similar argument.

Thus on the critical set we further have
\begin{equation}
\begin{split}\mu_1& = |\xi| \hat{\mu}_1 = d_{y_1}\phi_{\mathfrak{d}}, \\
\mu_2& = |\xi|\hat{\mu}_2 = - d_{y_2}\phi_{\mathfrak{d}},\\
\tau & = d_s\phi_{\mathfrak{d}}  = -|\xi| = -|\mu_2|= -|\mu_1|,
\end{split}
\end{equation}
which implies that the requirements in \eqref{eq: L+-, defn2} and \eqref{eq: propagating lagrangians} are equivalent and $\phi_{\mathfrak{d}}$ parametrizes $\mathscr{L}_+$.

In addition, we can make a change of coordinates on $\xi$ 
so that the oscillatory integral 
\begin{equation}
\int_{\R^{n-1}} e^{ i(\mk{d}(y_1,y_2) \hat{\mu}_2 \cdot \xi - s|\xi|) } a(s,y_1,y_2;\xi) d\xi
\end{equation}
becomes (abusing the notation to still use $a(s,y_1,y_2;\xi)$ to denote the amplitude)
\begin{equation} \label{eq: parametrix-FIO-0}
\int_{\R^{n-1}} e^{ i( \mk{d}(y_1,y_2) \textbf{1} \cdot \xi - s|\xi|) } a(s,y_1,y_2;\xi) d\xi,
\end{equation}
where $\textbf{1}=(1,0,...,0) \in \R^{n-1}$.

Recalling the proof of the vanishing principal symbol composition in \cite[Theorem~5.3.1]{FIO2}, it encodes the following fact (of course, this holds for general Fourier integral operators):
for an expression like \eqref{eq: parametrix-FIO-0} representing $I^{m}(\R \times Y \times Y, \mathscr{L}_+)$, if the amplitude vanishes to order $k$ on the critical set $\Lambda_+$, then the operator is in $I^{m-k}(\R \times Y \times Y, \mathscr{L}_+)$.
We sketch the reason here: the given condition allows us to Taylor expand the amplitude in terms of the defining functions of $\Lambda_+$ starting from order $k$, which are derivatives of the phase function. Then one can integrate by parts to obtain an expression that is a sum of oscillatory integrals with the same phase function but with amplitudes
\begin{equation*}
\partial^{\alpha}_{\xi} a(s,y_1,y_2;\xi), \; |\alpha| = k,
\end{equation*} 
which shows the operator is in $I^{m-k}(\R \times Y \times Y, \mathscr{L}_+)$.
Now if we replace $a(s,y_1,y_2;|\xi|,\hat{\xi})$ by its Taylor expansion with respect to $\hat{\xi}$ at $\hat{\xi} = \textbf{1}$, which does not depend on $\hat{\xi}$, then the difference vanishes to infinite order at the critical set $\Lambda_+$, hence only causes an error in $I^{-\infty}(\R \times Y \times Y, \mathscr{L}_+)$. So we can in fact write \eqref{eq: parametrix-FIO-0} (modulo a smoothing error, and still use $a$ to denote this new amplitude)
\begin{equation} \label{eq: parametrix-FIO-1}
\int_{\R^{n-1}} e^{ i(\mk{d}(y_1,y_2) \textbf{1} \cdot \xi - s|\xi|) } a(s,y_1,y_2;|\xi|) d\xi.
\end{equation}

This reduction can also be achieved by the proof of the equivalence of the phase function via the stationary phase expansion, which also says only the jet of the amplitude at the critical set has a non-residual contribution.
And this reduction is possible exactly because our assumption that $\exp$ is non-degenerate within time $\pi$ implies that the projection $\mathscr{P}_+$ is non-degenerate and the Fourier integral operator can locally be written as an oscillatory integral of only one variable (the length of the frequency). See \cite[Theorem~3.1.4]{FIO1} for the general relation between the number of parameters needed and the rank of this projection.

In summary, we have the parametrix for $e^{\pm is\sqrt{P}}$ when $s\in [0,\pi]$:
\begin{proposition} \label{prop:exp-half-wave}
Let $\mathfrak{D}_\pm(y_1,y_2)$ be as in \eqref{eq:defn-distance-spectrum}. The kernel of $e^{\pm is\sqrt{P}}$ can be written as 
\begin{equation}\label{eq:half-wave-kernel}
e^{\pm is\sqrt{P}}(y_1,y_2) = K_{\pm,N}(s; y_1,y_2)+R_{\pm,N}(s; y_1,y_2),
\end{equation}
where $R_{\pm,N}(s; y_1,y_2) \in C^{N-n-2}([0,\pi] \times Y \times Y)$ and
\begin{equation} \label{K-pm-N}
\begin{split}
K_{\pm,N}(s; y_1,y_2)&=(2\pi)^{n-1}
\sum_{\mk{d} \in \mk{D}_\pm (y_1,y_2)}
\int_{\R^{n-1}} e^{i \mk{d}(y_1,y_2){\bf 1}\cdot\xi} a_{\pm,\mk{d}}(s, y_1,y_2; |\xi|) e^{\pm  is|\xi|} d\xi\\
&=\sum_{\mk{d} \in\mk{D}_\pm (y_1,y_2)}\sum_{\varsigma = \pm}\int_0^\infty b_{\varsigma }(\rho \mk{d}) e^{\varsigma i \rho \mk{d} } a_{\pm,\mk{d}}(s, y_1,y_2; \rho) e^{\pm  is\rho} \rho^{n-2} d\rho,
\end{split}
\end{equation}
 with ${\bf 1}=(1,0,\ldots,0)$ and $a_{\pm,\mk{d}} \in S^0$: 
 \begin{equation}\label{a}
 |\partial^\alpha_{s,y_1,y_2}\partial_\rho^k a_{\pm,\mk{d}} (s,y_1,y_2;\rho)|\leq C_{\alpha,k}(1+\rho)^{-k},
 \end{equation}
 and
 \begin{equation}\label{b+-}
\begin{split}
| \partial_r^k b_\pm(r)|\leq C_k(1+r)^{-\frac{n-2}2-k},\quad k\geq 0.
\end{split}
\end{equation}
In addition, we may choose $a_{+,\mk{d}}=a_{-,\mk{d}}$ and choose them to be supported in $\rho \geq 1$.

\end{proposition}

\begin{remark}Recalling the definition of $\mk{D}_\pm (y_1,y_2)$ in \eqref{eq:defn-distance-spectrum} and the discussion of the uniform boundedness of $|\mk{D}_\pm (y_1,y_2)|$ (number of elements in it) there, the number of terms in the summation over $\mk{d} \in \mk{D}_\pm (y_1,y_2)$ is uniformly bounded.
\end{remark}

For the rest of this paper, we keep the $\pm$ sub-indices in $a_{\pm,\mk{d}}$ to indicate which operators those amplitudes are associated with, but one should keep in mind that they are actually the same function.

\begin{proof}
The form of the first oscillatory integral follows from the discussion above. 

The property of $b_\pm$ in the second representation of \eqref{K-pm-N} follows from the following identity:
\begin{equation} \label{eq:sphere-integral}
\begin{split}
 \int_{\mathbb{S}^{n-2}} e^{i \mk{d}(y_1, y_2) \rho{\bf 1}\cdot\omega} d\omega =\sum_{\pm}  b_\pm(\rho \mk{d}(y_1, y_2)) e^{\pm i \rho \mk{d}(y_1, y_2)},
\end{split}
\end{equation}
with $b_\pm \in S^{-\frac{n-2}{2}}$, which follows from \cite[Theorem 1.2.1]{sogge}.

Recall that the parametrix (for example, of $e^{is\sqrt{P}}$) is constructed via an asymptotic sum
\begin{equation}
K_+ = \sum_{j=0}^\infty K_j,
\end{equation}
with $K_j \in I^{-\frac{1}{4}-j}(\R \times Y \times Y , \mathscr{L}_-)$.
The index $N$ here is to indicate that we take the parametrix after only $N+1$ iterations in the parametrix construction. That is, we take 
\begin{equation}
K_{+,N} = \sum_{j=0}^{N} K_j.
\end{equation}
This is to avoid some further discussions about the Borel summation when we want to take $N=\infty$.
The $C^{N-n-2}$-regularity of the $R_N$-term follows from noticing that this term has a symbol that is $N$ orders lower compared with the leading term and then applying Sobolev regularity. See the discussion before \cite[Theorem~3.1.5]{Sogge-H} for the details of this numerology (notice that our $n-1$ is $n$ there).

We can take $a_{-,\mk{d}}=a_{+,\mk{d}}$ by the following observation: $\mathscr{L}_\pm$ are actually the same Lagrangian (ignoring time related components) just with the flow with initial condition $(y_2,\mu_2)$ in $\mathscr{L}_+$ replaced by $(y_2,-\mu_2)$ in $\mathscr{L}_-$. So the corresponding amplitude $a_{-,\mk{d}}(s,y_1,y_2;\xi)$ should satisfy
\begin{equation} \label{eq: a+,- relationship}
a_{-,\mk{d}}(s,y_1,y_2;-\xi) = a_{+,\mk{d}}(s,y_1,y_2;\xi),
\end{equation}
since they are obtained through solving the same transport equation along the same (lifted) geodesic. But in the step of the reduction to a function of $|\xi|$, the critical point for $a_{-,\mk{d}}(s,y_1,y_2;-\xi)$ is now at $\xi = |\xi|(-1,0,...,0)$. So after the reduction which makes the symbol depend only on $|\xi|$, (abusing notation as above, still using $a_{\pm,\mk{d}}$ to denote the amplitude) we have 
\begin{equation}  
a_{-,\mk{d}}(s,y_1,y_2;|\xi|) = a_{+,\mk{d}}(s,y_1,y_2;|\xi|).
\end{equation}


Finally, the last claim that we can take those amplitudes to be supported in $\rho \geq 1$ is because we can insert a cutoff $\chi(\rho)$ that is supported on $[1,\infty)$ and is identically $1$ on $[2,\infty)$. Then the part $(1-\chi(\rho))a_\pm(s,y_1,y_2;\rho)$ is only supported over finite $\rho$ and is a smooth function, hence can be collected into the $R_N$-term.
\end{proof}

Using $\cos(s\sqrt{P}) = \frac{1}{2}(e^{is\sqrt{P}}+e^{-is\sqrt{P}})$, and $a_{+,\mk{d}}(s,y_1,y_2;\rho) = a_{-,\mk{d}}(s,y_1,y_2;\rho)$ in the proposition above, we have:

\begin{corollary}[Hadamard parametrix I] 
\label{lemma: parametrix 1}
Let $Y$ and  $\mathfrak{D}(y_1,y_2)$ be as above, then for $|s|\leq \pi$ and  $\forall N>n+2$, the kernel of  $ \cos(s \sqrt{P})$ can be written as
\begin{equation}\label{KR}
\cos(s \sqrt{P})(y_1,y_2)=K_N(s; y_1,y_2)+R_N(s; y_1,y_2),
\end{equation}
where $R_N(s; y_1,y_2) \in C^{N-n-2}([0,\pi] \times Y \times Y)$ and
\begin{equation}\label{KN1}
\begin{split}
K_N(s; y_1,y_2)&=(2\pi)^{n-1} \sum_{\mk{d} \in \mk{D}(y_1,y_2) } \int_{\R^{n-1}} e^{i \mk{d}(y_1,y_2){\bf 1}\cdot\xi} a_{\mk{d}}(s, y_1,y_2; |\xi|) \cos(s |\xi|) d\xi\\
&=\sum_{\mk{d} \in \mk{D}(y_1,y_2) }\sum_{\pm} \int_0^\infty b_{\pm}(\rho \mk{d}) e^{\pm i \rho \mk{d}} a_{\mk{d}}(s, y_1,y_2; \rho) \cos(s \rho) \rho^{n-2} d\rho
\end{split}
 \end{equation}
 with ${\bf 1}=(1,0,\ldots,0)$ and $a_{\mk{d}}\in S^0$: 
 \begin{equation}
 |\partial^\alpha_{s,y_1,y_2}\partial_\rho^k a_{\mk{d}}(s,y_1,y_2;\rho)|\leq C_{\alpha,k}(1+\rho)^{-k},
 \end{equation}
 and
 \begin{equation} 
\begin{split}
| \partial_r^k b_\pm(r)|\leq C_k(1+r)^{-\frac{n-2}2-k},\quad k\geq 0.
\end{split}
\end{equation}
In addition, we can choose $a_{\mk{d}}(s,y_1,y_2;\rho)$ to be supported in $\rho \geq 1$.
 \end{corollary}

Next we give the oscillatory integral representation of the Poisson-wave operators $e^{(- \tilde{s} \pm i\pi)\sqrt{P}}$.
The jet of the half-wave propagator at the `endpoint' and that of the Poisson wave propagator at the `initial point' match in the following sense:
\begin{lemma}[Hadamard parametrix: Poisson-wave operator]   \label{lemma:parametrix-poisson} 
For $\tilde{s} \geq 0$ and $\forall N>n+2$, the kernel of the Poisson-wave operator $ e^{(-\tilde{s} \pm  i\pi)\sqrt{P}}$ can be written as
 \begin{equation}\label{KR'}
e^{(- \tilde{s} \pm i\pi)\sqrt{P}}(y_1,y_2) =\tilde{K}_{\pm,N}( \tilde{s} ; y_1,y_2)+\tilde{R}_{\pm,N}(\tilde{s} ; y_1,y_2),
 \end{equation}
 where $\tilde{R}_{\pm,N}( \tilde{s} ; y_1,y_2)\in C^{N-n-2} ([0,+\infty)\times \CS\times \CS)$ and  
  \begin{equation}\label{eq: poisson-0}
  \begin{split}
\tilde{K}_{\pm,N}( \tilde{s} ; y_1,y_2)&=(2\pi)^{n-1}
\sum_{\mk{d} \in \mk{D}(y_1,y_2) }
\int_{\R^{n-1}} e^{i \mk{d}(y_1,y_2){\bf 1}\cdot\xi} 
\tilde{a}_{\pm,\mk{d}}( \tilde{s} , y_1,y_2; |\xi|) e^{( - \tilde{s} \pm  i\pi)|\xi| } d\xi
\\&=\sum_{\mk{d} \in \mk{D}(y_1,y_2)} \sum_{ \varsigma = \pm} \int_{0}^\infty b_{\varsigma}(\rho \mk{d}) e^{\varsigma i \rho \mk{d}} \tilde{a}_{\pm,\mk{d}}(\tilde{s} , y_1,y_2; \rho) e^{(- \tilde{s} \pm  i\pi)\rho } \rho^{n-2} d\rho
\end{split}
 \end{equation}
 with ${\bf 1}=(1,0,\ldots,0)$ and $\tilde{a}_{\pm,\mk{d}}\in S^{0}$ satisfying
 \begin{equation}\label{a'}
 |\partial^\alpha_{\tilde{s} ,y_1,y_2}\partial_\rho^k \tilde{a}_{\pm,\mk{d}}(\tilde{s},y_1,y_2;\rho)|\leq C_{\alpha,k}(1+\rho)^{-k}.
 \end{equation}
 In addition, we can choose $\tilde{a}_{\pm,\mk{d}}$ such that the jets of $\tilde{a}_{\pm,\mk{d}}$ at $\tilde{s} = 0$ coincide with those of $a_{\pm,\mk{d}}$ (given in \eqref{K-pm-N}) at $s=\pi$ in the sense that
\begin{equation} \label{eq:wave-poisson-jet-match}
(\partial_{\tilde{s}}^k \tilde{a}_{\pm,\mk{d}})(0,y_1,y_2;\rho)
= i^k(\partial_s^k a_{\pm,\mk{d}})(\pi,y_1,y_2;\rho).
\end{equation}
\end{lemma}

\begin{remark}
The matching condition \eqref{eq:wave-poisson-jet-match} can be considered as stating that $\tilde{a}_{\pm}$ is an almost analytic extension of $a_{\pm}$ since it means that they satisfy the Cauchy-Riemann equations with the pair $(s,\tilde{s})$ up to infinite order at $\tilde{s}=0$.
This conclusion is actually not so surprising as $e^{(-\tilde{s}\pm i\pi)\sqrt{P}}$ is the analytic (in $s$) extension of $e^{\pm is\sqrt{P}}$ into the upper half plane at $s=\pi$. It is the shift to the parametrix that makes the extension only almost analytic.
\end{remark}

\begin{proof} 
We prove the result for $e^{(-s+i\pi)\sqrt{P}}$ and the proof with the other sign is similar.
We consider the case $m=0$ first. We define $I_{\rm dp}^{m}(\R \times Y \times Y; \mathscr{L}_\pi)$, where $\mathscr{L}_\pi$ stands for the part of $\mathscr{L}_+$ with $s=\pi$, to be the class of operators that have Schwartz kernels that are sums of oscillatory integrals locally of the form
\begin{equation}
\mathscr{A} :=  \int_{\R^{n-1}} e^{i \mk{d}(y_1,y_2){\bf 1}\cdot\xi} a(s,y_1,y_2; |\xi|) e^{(-s + i\pi)|\xi| } d\xi,
\end{equation}
modulo a smooth function, where $ \mk{d} \in \mk{D}(y_1,y_2), \; a \in S^{m}(\R \times Y \times Y \times \R^{n-1})$ that is symbolic in $\xi$. And we call
\begin{equation}
\sigma_m(\mathscr{A}) = [a] \in S^{m}(\R \times Y \times Y \times \R^{n-1})/S^{m-1}
\end{equation}
its principal symbol. And by definition 
\begin{align}
\sigma_m(\mathscr{A}) = 0 \text{ if and only if } \mathscr{A} \in I_{\rm dp}^{m-1}(\R \times Y \times Y; \mathscr{L}_\pi).
\end{align}

We will write $a$ for the equivalence class $[a]$ when there is no confusion. 
We call this class of operators like $\mathscr{A}$ \emph{the damped Fourier integral operators} associated with $\mathcal{L}_\pi$. We are not developing a geometrically invariant calculus for this type of integral operator, but only conducting the parametrix construction for $(\partial_s+\sqrt{P})$ in a fixed coordinate chart:
\begin{equation} \label{eq: heat parametrix, IVP}
(\partial_s+\sqrt{P})\tilde{K}(s;y_1,y_2) \in I_{\rm dp}^{-\infty}(\R \times Y \times Y; \mathscr{L}_\pi), \; \tilde{K}(0,y_1,y_2) = e^{i\pi\sqrt{P}}, 
\end{equation}
where $I_{\rm dp}^{-\infty}(\R \times Y \times Y; \mathscr{L}_\pi) = 
\bigcap_{m \in \Z} I_{\rm dp}^{m}(\R \times Y \times Y; \mathscr{L}_\pi)$. 
By Proposition~\ref{prop:exp-half-wave}, $e^{i\pi\sqrt{P}}$ has the representation
\begin{equation}  \label{eq: e i*pi*sqrt P}
e^{i\pi\sqrt{P}} = (2\pi)^{n-1} \sum_{\mk{d} \in \mk{D}(y_1,y_2)} \int_{\R^{n-1}} e^{i \mk{d}(y_1,y_2){\bf 1}\cdot\xi} a_{+,\mk{d}}(\pi,y_1,y_2; |\xi|) e^{ i\pi|\xi| } d\xi + R_N
\end{equation}
with $R_N \in C^{N-n-2}(Y \times Y)$. 
In addition we have $I_{\rm dp}^{-\infty}(\R \times Y \times Y; \mathscr{L}_\pi) \subset C^\infty([0,\infty) \times Y \times Y)$ since differentiation only introduces $|\xi|$-factors, which can be absorbed by the amplitude, which has arbitrarily high polynomial decay. 

Then we construct the solution to \eqref{eq: heat parametrix, IVP} by a similar argument to the H\"ormander type parametrix construction through an asymptotic sum:
\begin{equation}
\tilde{K} = \sum_{j=0}^\infty \tilde{K}_j, 
\end{equation}
where 
\begin{equation}
\tilde{K}_j = (2\pi)^{n-1} \sum_{\mk{d} \in \mk{D}(y_1,y_2)} \int_{\R^{n-1}} e^{i \mk{d}(y_1,y_2){\bf 1}\cdot\xi} \tilde{a}_{j,\mk{d}}(s, y_1,y_2; |\xi|) e^{(-s+ i\pi)|\xi| } d\xi,
\end{equation}
where $\tilde{a}_{j,\mk{d}} \in S^{-j}$, $\tilde{a}_{0,\mk{d}}(0,y_1,y_2;|\xi|)=a_{+,\mk{d}}(\pi,y_1,y_2;|\xi|)$, $\tilde{a}_{j,\mk{d}}(0,y_1,y_2;|\xi|)=0$ for $j \geq 1$,
and most importantly
\begin{align}  \label{eq: parametrix, step N}
(\partial_s + \sqrt{P})\big(\sum_{j=0}^N \tilde{K}_j\big) \in I_{\rm dp}^{-1-N}(\R \times Y \times Y; \mathscr{L}_\pi).
\end{align}

Now we consider the part in $\tilde{K}_N$ associated with $\mk{d}$ individually. This is sufficient to give the global parametrix since we only need to construct the parametrix for very short time due to the exponentially decaying factor $e^{-\tilde{s}|\xi|}$, which means there is no propagation between different pieces.
We denote
\begin{equation}
\tilde{K}_{j,\mk{d}} = (2\pi)^{n-1} \int_{\R^{n-1}} e^{i \mk{d}(y_1,y_2){\bf 1}\cdot\xi} \tilde{a}_{j,\mk{d}}(s, y_1,y_2; |\xi|) e^{(-s+ i\pi)|\xi| } d\xi,
\end{equation}
and construct $\tilde{K}_{j,\mk{d}}$ inductively. Since \eqref{eq: parametrix, step N} is equivalent to 
\begin{equation} \label{eq:transport-dp}
\sigma_{-N}\Big((\partial_s + \sqrt{P})\tilde{K}_{N,\mk{d}}\Big)
= - \sigma_{-N}\Big((\partial_s + \sqrt{P})\big(\sum_{j=0}^{N-1}\tilde{K}_{j,\mk{d}}\big)\Big),
\end{equation}
and this becomes a transport equation of $\sigma_{-N}(\tilde{K}_{N,\mk{d}})$ by the same argument as in the real phase case because our phase function satisfies the conditions in \cite{MelinSj1976complexFIO}, we can apply \cite[Theorem~2.3]{MelinSj1976complexFIO}, which is the stationary phase lemma with complex phase to the composition $\sqrt{P}\tilde{K}_j$. More concretely, by the result of \cite{Seeley1967}, $\sqrt{P} \in \Psi^{1}(Y)$. Denoting the variables of the Schwartz kernel of $\sqrt{P}$ by $(y_1,y_1')$ and those of $\tilde{K}_{j,\mk{d}}$ by $(y_1',y_2)$ as above, we apply the stationary phase lemma to the $y_1'$-integral. 
Thus $\tilde{a}_{N,\mk{d}}$ exists for $s \in [0,\delta_1]$ with $\delta_1$ independent of $N$.

Notice that the contribution of the oscillatory integral outside any neighborhood of $0$ is a smooth function, thus one can extend $\tilde{a}_{j,\mk{d}}$ above smoothly while remaining in the same symbol class and the parametrix property continues to hold.

For the parametrix construction, it only remains to show that $\tilde{K}$ differs from $e^{(-s+i\pi)\sqrt{P}}$ only by a smooth term. Setting
\begin{equation*}
\mathscr{R}(s):=\tilde{K}-e^{(-s+i\pi)\sqrt{P}},
\end{equation*}
then it solves
\begin{equation} \label{eq: PDE, scr R 1}
(\partial_s+\sqrt{P})\mathscr{R} = f, \quad \mathscr{R}(0) = 0,
\end{equation}
where $f \in C^\infty([0,\infty) \times Y \times Y)$. Applying $(\partial_s-\sqrt{P})$ to both sides, we have
\begin{equation} \label{eq: PDE, scr R 2}
(\partial_s^2-P)\mathscr{R} = \tilde{f}, \quad \mathscr{R}(0) = 0,
\end{equation}
where $\tilde{f}=(\partial_s-\sqrt{P})f$.

Next we show that $|\xi_1|$ is comparable to $|\xi_2|$ near $\mathrm{WF}(e^{(-s+i\pi)\sqrt{P}})$, where $\xi_i$ are dual variables to $y_i$.
For $s>0$, $e^{(-s+i\pi)\sqrt{P}}$ has smooth kernel, and for $s=0$, it follows from the oscillatory integral representation \eqref{eq: e i*pi*sqrt P} and a non-stationary phase argument (see, for example, the proof of \cite[Proposition~2.5.7]{FIO1}) with respect to the $y_1,y_2$-regularity. In addition, the regularity in $s$ can be transferred to the regularity in $y_1$ since $\partial_se^{(-s+i\pi)\sqrt{P}}=-\sqrt{P}e^{(-s+i\pi)\sqrt{P}}$.
The same argument applies to $\mathrm{WF}(\tilde{K})$, showing that $|\xi_1|$ is comparable to $|\xi_2|$ near it.
Consequently,  $|\xi_1|$ is comparable to $|\xi_2|$ near $\mathrm{WF}(\mathscr{R})$ and $\partial_s^2-P$ is elliptic near it. 
Thus one can select a pseudodifferential operator that is fully elliptic (not only when $|\xi_1|$ is comparable to $|\xi_2|$), but coincides with $\partial_s^2-P$ near $\mathrm{WF}(\mathscr{R})$, and apply \cite[Theorem~17.3.2]{HorVol3} to it. 
Though the theorem there is local in $(y_1,y_2)$, that is sufficient for us since we are only concerned with smoothness, and in fact the control can be upgraded to a global one by the compactness of $Y \times Y$. In addition, the cited theorem only concerns up to the second-order derivatives, but one can apply $\partial_s,\sqrt{P}$ iteratively to both sides of \eqref{eq: PDE, scr R 2} to obtain the same form of equation for $P^k\mathscr{R},\partial_s^k\mathscr{R}$, and conclude arbitrary order of smoothness of $\mathscr{R}$.

The only thing that remains to justify is \eqref{eq:wave-poisson-jet-match}. We fix a $\mk{d} \in \mk{D}(y_1,y_2)$ and will omit this index in amplitudes below.
We only consider the $+$ sign (i.e., for $e^{(-\tilde{s}+i\pi)\sqrt{P}}$) case and the other case can be obtained in the same way.
This is because $\tilde{a}_{+}$ is obtained by solving a transport equation using $a_{+}(\pi,y_1,y_2;\rho)$ as the initial value.
Let $\delta>0$ be such that $\pi+\delta<\conR$. 
Next we give the transport equation that is satisfied by $\tilde{a}_{+,j}$ for $\tilde{s}<\delta$.
Recall that $a_{+}$ is constructed as an asymptotic sum $a_{+} = \sum_{j=0}^\infty a_{+,j}$ such that $a_{+,j} \in S^{-j}$ and $a_{+,j}$ satisfy transport equations of the form 
\begin{equation} \label{eq:transport-a0j}
(\partial_s + b_j)a_{+,j} = f_j,
\end{equation}
for $s<\pi+\delta$, $f_0=0$. 
On the other hand, the transport equation is the same as that for $s$ except that we use `imaginary time' instead (and still continue to use the real geodesic flow on $\CS$), so the equation for $\tilde{a}$ takes the form
\begin{equation} \label{eq:transport-tilde-a0j}
(i^{-1}\partial_{\tilde{s}} + \tilde{b}_j)\tilde{a}_{+,j} = \tilde{f}_j,
\end{equation}
$\tilde{f}_0=0$ and $\tilde{b}_j$ is obtained from the sub-principal symbol of $\sqrt{P}$, hence $\tilde{b}_j$ at `time' $i\tilde{s}$ equals $b_j$ at $s=\tilde{s}+\pi$ when they have the same starting point along the flow when $\tilde{s}=0$ and $s=\pi$ respectively. (See \cite[Equation~(1.18)]{MSjcomplexFIO-CPDE} for the complex Hamilton vector field associated with this transport equation). This observation gives
\begin{equation} \label{eq:bj-tildebj-CR}
\partial^k_{\tilde{s}} \tilde{b}_j = i^k \partial^k_s b_j
\end{equation}
when $\tilde{s}<\delta$.
Now comparing \eqref{eq:transport-a0j} and \eqref{eq:transport-tilde-a0j} gives \eqref{eq:wave-poisson-jet-match} with $k=1$ for the $j=0$ part.

Now we run induction on $k$ with fixed $j=0$ first and then run induction on $j$. We first show
\begin{equation} \label{eq:wave-poisson-jet-match-j}
(\partial_{\tilde{s}}^k \tilde{a}_{+,0})(0,y_1,y_2;\rho)
= i^k(\partial_s^k a_{+,0})(\pi,y_1,y_2;\rho),
\end{equation}
for $j=0$.

We differentiate in $s,\tilde{s}$ ($k-1$ times) respectively to see that
\begin{equation} \label{eq:k-times-s-derivative-a00}
\partial_s^ka_{+,0} = -\sum_{\ell=0}^{k-1} \binom{k-1}{\ell} \partial^{\ell}_s b_0 \partial^{k-1-\ell}_sa_{+,0},
\end{equation}
and
\begin{equation} \label{eq:k-times-s-derivative-tildea00}
\partial_{\tilde{s}}^k\tilde{a}_{+,0} = -i \sum_{\ell=0}^{k-1} \binom{k-1}{\ell} \partial^{\ell}_{\tilde{s}} \tilde{b}_0 \partial_{\tilde{s}}^{k-1-\ell}\tilde{a}_{+,0}.
\end{equation}
This proves \eqref{eq:wave-poisson-jet-match-j} by \eqref{eq:bj-tildebj-CR} and the induction hypothesis.

Now we proceed to induction on $j$. The case $k=0$ for all $j$ holds since we took $a_{+}(\pi,y_1,y_2;|\xi|)$ as the initial value for $\tilde{a}_{+}$ at $\tilde{s}=0$.
For $j \geq 1$, recalling \eqref{eq:transport-dp}, $f_j$ is obtained by $- \sigma_{-j}\Big((\partial_s + \sqrt{P})\big(\sum_{j=0}^{N-1}\tilde{K}_{j,\mk{d}}\big)\Big)$, 
which is a linear differential operator with coefficients satisfying the same property as \eqref{eq:bj-tildebj-CR} (for the same reason as above for $b_j$ and $\tilde{b}_j$, those coefficients are from the same flow with different parametrization) applied to $\sum_{\ell = 0}^{j-1} a_{+,\ell}$ and $\sum_{\ell=0}^{j-1} \tilde{a}_{+,\ell}$, which satisfy \eqref{eq:wave-poisson-jet-match-j}.
Then by the same strategy as above, differentiating in $s$ and $\tilde{s}$ repeatedly in \eqref{eq:transport-a0j} and \eqref{eq:transport-tilde-a0j} will give two equations of the same form as \eqref{eq:k-times-s-derivative-a00} and \eqref{eq:k-times-s-derivative-tildea00} except that now we have the contribution from $f_j$ as well. But as aforementioned, this term also has the property that: the coefficients satisfy the analogue of \eqref{eq:bj-tildebj-CR} and those $a_{+,\ell}$ have indices less than $j$. So the conclusion follows from the induction hypothesis. 

\end{proof}

\section{The proof of Theorem \ref{thm1:dispersiveS} }   \label{sec: proof of dispersive S}

In this section, we prove Theorem \ref{thm1:dispersiveS} by using
Proposition \ref{prop:Sch-pro} and the properties of the Hadamard parametrix on $Y$ shown in Section \ref{sec: parametrix}. To this end, we divide this section into two parts. The first part is devoted to dealing with the case $\frac{r_1r_2}{|t|}\lesssim 1$ by establishing Proposition \ref{prop<}, and in the second part,
 we prove Proposition \ref{prop>} in the case that $\frac{r_1r_2}{|t|}\gg 1$ under the assumption that  the {conjugate radius} $\conR$ of $Y$ satisfies $\conR>\pi$.
 
 \subsection{Part I: The case that $\frac{r_1r_2}{|t|}\lesssim 1$} By \eqref{ker:S1}, Theorem \ref{thm1:dispersiveS} is a consequence of the following proposition.
 
 \begin{proposition}\label{prop<} Let $P$ be the operator in Proposition \ref{prop:Sch-pro} and let $z_1=(r_1,y_1)$ and $z_2=(r_2,y_2)$ be in $X=C(Y)$.
 Suppose that $z:=\frac{r_1r_2}{2|t|}\lesssim 1$. Then there exists a constant $C$ such that
\begin{equation}\label{est:<1}
\begin{split}
z^{-\frac{n-2}2} \Big|\sum_{k\in\mathbb{N}}\varphi_{k}(y_1)\overline{\varphi_{k}(y_2)}
(-i)^{\nu_k} J_{\nu_k}\big(\frac{r_1r_2}{2t}\big)\Big|\leq C z^{-\frac{n-2}2+\nu_0},
\end{split}
\end{equation}
where $\varphi_{k}(y)$ is the eigenfunction of the operator $P$ corresponding to the eigenvalue $\nu^2_k$ and $\nu_0$ is the positive square root of the smallest eigenvalue of the operator $P$.
\end{proposition}

\begin{proof}[The proof of Proposition \ref{prop<}]
We mainly use the asymptotic estimates of the eigenfunctions and the Bessel functions to prove \eqref{est:<1}.
Recall \eqref{eig-eq} and the eigenfunction estimate (see \cite[(3.2.2), (3.2.5)-(3.2.6)]{Sogge-H}) 
 \begin{equation} \label{eq:eigen-L-infinity}
\|\varphi_{k}(y)\|_{L^\infty(Y)}\leq  C (1+\nu^2_k)^{\frac{n-2}4},
 \end{equation}
and Weyl's asymptotic formula (e.g. see \cite{Yau}, \cite[Eq.(0.6)]{Garding53})
 \begin{equation}\label{est:eig}
\nu^2_k\sim (1+k)^{\frac{2}{n-1}},\quad k\geq 1,\implies\|\varphi_{k}(y)\|^2_{L^\infty(Y)}\leq  C (1+k)^{\frac{n-2}{n-1}}.
 \end{equation}
 For our purpose, we recall that the Bessel function $J_\nu(z)$ of order $\nu>-1/2$ satisfies 
\begin{equation}\label{est:r}
|J_\nu(z)|\leq\frac{Cz^\nu}{2^\nu\Gamma(\nu+\frac{1}{2})\Gamma(\frac12)}\Big(1+\frac{1}{\nu+\frac12}\Big),
\end{equation}
where $C$ is an absolute constant independent of $z$ and $\nu$.
Therefore, from \eqref{est:eig} and the facts that $z\leq C$ and $\nu_k\geq \nu_0$, we have 
\begin{equation*}
\begin{split}
\text{LHS of}\, \eqref{est:<1}
&\leq C z^{-\frac{n-2}2} \sum_{k\in\mathbb{N}} (1+k)^{\frac{n-2}{n-1}} \frac{z^{\nu_0}C^{\nu_k}}{2^{\nu_k}\Gamma(\nu_k+\frac12)}
\\&\leq C z^{-\frac{n-2}2+\nu_0}\sum_{k\in\mathbb{N}} \frac{(1+k)^{\frac{n-2}{n-1}}(C/2)^{\nu_k} }{\Gamma(\nu_k+\frac12)}.
  \end{split}
 \end{equation*}
 Recalling that $\nu_k\sim (1+k)^{\frac 1{n-1}}$, then the summation in $k\in\mathbb{N}$ converges.  Hence we have completed the proof of \eqref{est:<1}.
\end{proof}

\subsection{Part II: The case that $\frac{r_1r_2}{|t|}\gg 1$}
In this subsection, we mainly use \eqref{S-kernel} to prove Theorem \ref{thm1:dispersiveS} in the case  $\frac{r_1r_2}{|t|}\gg 1$. We want to prove 
\begin{proposition}\label{prop>} Let $P$ be the operator in Proposition \ref{prop:Sch-pro} and let $z_1=(r_1,y_1)$ and $z_2=(r_2,y_2)$ in $X=C(Y)$.
 Suppose that $z:=\frac{r_1r_2}{2|t|}\gg 1$. If the {conjugate radius} $\conR$ of $Y$ satisfies $\conR>\pi$. Then there exists a constant $C$ such that
\begin{equation}\label{est:>1}
\begin{split}
z^{-\frac{n-2}2}
 & \Big|\frac1{\pi}\int_0^\pi e^{-iz\cos(s)} \cos(s\sqrt{P})(y_1,y_2) ds\\
  &-\frac{\sin(\pi\sqrt{P})}{\pi}\int_0^\infty e^{iz \cosh s} e^{-s\sqrt{P}}(y_1,y_2) ds\Big|\leq C.
\end{split}
\end{equation}

\end{proposition}

The proof is more delicate than the above case that $z\lesssim 1$. To this end, we introduce a smooth cutoff function $\chi_\delta\in C^\infty([0,\pi])$ with small $0<\delta\ll 1$ such that
 \begin{equation}
 \chi_\delta(s)=
 \begin{cases} 1, \quad s\in[0, \delta];\\
 0, \quad s\in[2\delta,\pi],
 \end{cases}\qquad  \chi^c_\delta(s)=1- \chi_\delta(s).
 \end{equation}
We aim to consider three terms:
\begin{equation}\label{S1}
\begin{split}
I_G(z;y_1,y_2):=\frac{z^{-\frac{n-2}2}}{\pi}\int_0^\pi  e^{-iz\cos(s)}  \chi^c_\delta(\pi-s) \cos(s\sqrt{P}) ds,
\end{split}
\end{equation}
\begin{equation}\label{S2}
\begin{split}
I_{GD}(z;y_1,y_2):=&\frac{z^{-\frac{n-2}2}}{\pi}\Big(\int_0^\pi  e^{-iz\cos(s)} \chi_\delta(\pi-s) \cos(s\sqrt{P}) ds\\
&-\sin(\pi\sqrt{P})\int_0^\infty  e^{iz\cosh(s)} \chi_\delta(s) e^{-s\sqrt{P}} ds\Big),
\end{split}
\end{equation}
and
\begin{equation}\label{S3}
\begin{split}
I_{D}(z;y_1,y_2):=-\frac{z^{-\frac{n-2}2}\sin(\pi\sqrt{P})}{\pi}\int_0^\infty  e^{iz\cosh(s)}  \chi^c_\delta(s) e^{-s\sqrt{P}} ds.
\end{split}
\end{equation}

Therefore, Proposition \ref{prop>} is proved if we can prove that the three terms $I_G(z;y_1,y_2)$, $I_{GD}(z;y_1,y_2)$ and $I_D(z;y_1,y_2)$ are uniformly bounded when $z\gg1$, and this is the goal of the rest of this section.


\begin{proof}[The contribution of \eqref{S1}]
By using the Hadamard parametrix \eqref{KR}, we need to consider two terms associated with $K_N(s; y_1,y_2)$ and $R_N(s; y_1,y_2)$ respectively. 
It is easy to see the contribution of the term associated with $R_N$  is
  \begin{equation}
\begin{split}
z^{-\frac{n-2}2}\Big| \int_0^\pi e^{-iz\cos(s)}  \chi^c_\delta(\pi-s) R_N(s; y_1,y_2) ds\Big| \lesssim 1
  \end{split}
 \end{equation}
due to the fact that one can choose $N$ large enough such that $$|R_N(s,y_1,y_2)| \lesssim 1,\quad 0\leq s\leq \pi.$$

Now we consider terms associated with $K_N(s; y_1,y_2)$.  Recall \eqref{KN1}, we want to show
\begin{equation*}
\begin{split}
&z^{-\frac{n-2}2}  \Big|\int_0^\pi e^{-iz\cos(s)}  \chi^c_\delta(\pi-s)\\
&\quad \times\int_{0}^\infty b_\pm(\rho \mk{d}) e^{\pm i \rho \mk{d} } a(s, y_1,y_2; \rho) \cos(s \rho) \rho^{n-2}d\rho \, ds\Big|\leq C ,
\end{split}
\end{equation*}
with $\mk{d} \in \mk{D}(y_1,y_2)$.
We summarize this as the lemma below, which will finish the proof of this part.
\end{proof}

\begin{lemma}\label{lem:key1} Let $z\gg 1$, $\mk{d}(y_1,y_2) \in \mk{D}(y_1,y_2)$, and suppose that
\begin{equation}\label{bpm}
\begin{split}
| \partial_r^k b_\pm(r)|\leq C_k(1+r)^{-\frac{n-2}2-k}, \forall k \in \mathbb{N},
\end{split}
\end{equation}
and let $a\in S^0$:
\begin{equation}\label{symb-a}
|\partial^\alpha_{s,y_1,y_2}\partial_\rho^k a(s,y_1,y_2;\rho)|\leq C_{\alpha,k}(1+\rho)^{-k},
\end{equation}
then there exists a constant $C$ independent of  $z, y_1, y_2$ such that
\begin{equation}\label{osi-1}
\begin{split}
&\Big|\int_0^\pi e^{-iz\cos(s)}  \chi^c_\delta(\pi-s)\\
&\quad \times\int_{0}^\infty b_\pm(\rho \mk{d}) e^{\pm i \rho \mk{d}} a(s, y_1,y_2; \rho) \cos(s \rho) \rho^{n-2}d\rho \, ds\Big|\leq C z^{\frac{n-2}2} .
\end{split}
\end{equation}
\end{lemma}

For the rest of this section, for concrete estimates we only consider the case $\mk{d}(y_1,y_2)=d_h(y_1,y_2)$. The proofs in all other cases are the same, in fact simpler, since $\mk{d}$ is lower bounded by $\mathrm{inj}(Y)>0$ when it is not $d_h(y_1,y_2)$. 
We don't need to consider the case $\mk{d} \ll 1$ in that setting and the proof in the other case in which $\mk{d}$ is lower bounded proceeds in the same manner.

\begin{proof}[The proof of Lemma \ref{lem:key1}] 
Let us fix a bump function
$\beta\in C^\infty_0((-2,-1/2)\cup(1/2,2))$ satisfying
\begin{equation}\label{beta-d}
\sum_{\ell=-\infty}^\infty \beta(2^{-\ell} s)=1, \quad s \in \R\backslash\{0\},
\end{equation}
and we set
$$\beta_{J}(s)=\sum_{\ell\le J} \beta(2^{-\ell}s)
\in C^\infty_0((-2^{J+1},2^{J+1})),$$
for $J \in \mathbb{N}_+$ to be determined. To prove \eqref{osi-1}, we consider two cases.

 \textbf{Case 1}. $d_h(y_1,y_2)\leq C_1 z^{-\frac12}$.   In this case, we take $J$ large enough so that $2^{J-1}\geq 2C_1$ and we want to show that 
 \begin{equation}\label{osi-1<b}
 \begin{split}
 z^{-\frac{n-2}2}
\Big|\int_0^\pi &e^{-iz\cos(s)}  \chi^c_\delta(\pi-s)\Big(\beta_{J}(z^{1/2} s)+\sum_{j\ge J+1}\beta(2^{-j}z^{1/2}s)\Big) \\
&\times \int_{0}^\infty  b_\pm(\rho d_h) e^{\pm i \rho d_h} a(s, y_1, y_2; \rho) \cos(s \rho) \rho^{n-2}d\rho\, ds \Big|\lesssim 1.
\end{split}
 \end{equation}
 For the term associated with $\beta_J$, we have $|s|\lesssim z^{-\frac12}\ll 1$ due to the compact support of $\beta_J$. If we also have $\rho\le 4z^{1/2}$,  then the integral in \eqref{osi-1<b} with $\beta_J$ is always bounded by 
  \begin{equation}\label{rho-l}
 \begin{split}
 z^{-\frac{n-2}2}
\int_{|s|\lesssim z^{-\frac12}} ds
\int_{\rho\leq 4z^{\frac12}} \rho^{n-2}d\rho \lesssim z^{-\frac{n-2}2} z^{-1/2} z^{\frac{n-1}{2}}\lesssim 1.
\end{split}
 \end{equation}
On the other hand, if we have $\rho\ge 4z^{1/2}$, we integrate by parts in $s$ in \eqref{osi-1<b} $N$ times by writing $\cos(\rho s)ds = \rho^{-1}d\sin(\rho s)$ or $\sin(\rho s)ds = -\rho^{-1}d\cos(\rho s)$ and view all other factors together as another part. 
Then all boundary terms at either $s=0$ or  $s=\pi$ vanish. 
For $s=\pi$, this is because the factor $\chi^c_\delta(\pi-s)$ and all of its derivatives vanish at $s=\pi$.
For $s=0$, an induction argument shows that when we integrate by parts an odd number of times, the boundary term is a multiple of
\begin{equation} \label{eq:boundary-0-odd}
\big(\frac{\partial}{\partial s}\big)^{k}\Big(e^{-iz\cos(s)}  \chi^c_\delta(\pi-s) \beta_J(z^{1/2}s)\Big) \sin(\rho s) \Big|_{s=0}
\end{equation}
with even $k$; while when we integrate by parts an even number of times, then the boundary term is a multiple of  
\begin{equation} \label{eq:boundary-0-even}
\big(\frac{\partial}{\partial s}\big)^{k}\Big(e^{-iz\cos(s)}  \chi^c_\delta(\pi-s) \beta_J(z^{1/2}s)\Big) \cos(\rho s) \Big|_{s=0}
\end{equation}
with odd $k$. For \eqref{eq:boundary-0-odd}, it vanishes due to the $\sin(\rho s)$-factor. For \eqref{eq:boundary-0-even}, since the derivatives of $\chi^c_\delta(\pi-s)$ and $\beta_J(z^{1/2}s)$ vanish at $s=0$, only the term with all derivatives hitting $e^{-iz\cos(s)}$ gives contribution and it is equal to 
\begin{equation} \label{eq:boundary-0-even-2}
\big(\frac{\partial}{\partial s}\big)^{k}\Big(e^{-iz\cos(s)} \Big)\chi^c_\delta(\pi-s) \beta_J(z^{1/2}s) \cos(\rho s) \Big|_{s=0}
\end{equation}
with odd $k$. Since $e^{-iz\cos(s)}$ is even in $s$,  its odd number of derivatives vanishes at $0$ and this boundary term vanishes as well.
In sum, all boundary terms vanish and each time we gain a factor of $\rho^{-1}$. 
 In addition, the factor introduced by differentiating other factors is a sum of terms of the form (modulo uniformly bounded smooth factors)
 \begin{equation} \label{eq:typical-term-1}
 z^{ \frac{k_1}{2} } (z \sin s)^{k_2} z^{k_3}P(\cos s,\sin s),
 \end{equation} 
 where $P$ is a polynomial. Here $k_1$ is the number of derivatives falling on the $\beta_J$-factor and $k_2+k_3$ is the number of derivatives falling on $e^{-iz\cos s}$. Notice that a $z$ factor without $\sin s$ paired with it can only arise by differentiating $\sin s$ (or its power), so $z^{k_3}$ has also cost $k_3$ derivatives on $\sin s$ and we have
 \begin{equation}
k_1 + k_2 + 2k_3 \leq N.
 \end{equation}
Using $|z\sin s| \lesssim z^{1/2}$ on the current region, we have
\begin{equation} \label{eq:amplitude-N-derivative-bound}
\Big|\Big(\frac{d}{ds}\Big)^N\Big(e^{-iz\cos(s)}  \chi^c_\delta(\pi-s)\beta_{J}(z^{1/2} s)\Big)\Big|\leq C_N z^{\frac N2}.
\end{equation}
So, after integration by parts $N$ times for $N\ge n$, 
  the integral in \eqref{osi-1<b} is bounded by 
 $$z^{-\frac{n-2}2} z^{-1/2} z^{N/2}  \int_{z^{1/2}}^\infty \rho^{n-2-N}d\rho\lesssim 1.
 $$
In sum, we have proved
 \begin{equation}
 \begin{split}
 z^{-\frac{n-2}2}
\Big|\int_0^\pi &e^{-iz\cos(s)}  \chi^c_\delta(\pi-s) \beta_{J}(z^{1/2} s)\\
&\times \int_{0}^\infty b_\pm(\rho d_h) e^{\pm i \rho d_h} a(s, y_1,y_2; \rho) \cos(s \rho) \rho^{n-2}d\rho ds\Big| \lesssim 1.
\end{split}
 \end{equation}
 For the terms with $\beta(2^{-j}z^{1/2}s), j \geq J+1$, we have $ 2^{j-1}z^{-1/2} \leq s \leq 2^{j+1}z^{-1/2}$ and $2^j\lesssim  z^{1/2}$ on the support of this $\beta-$factor.  In this case, we will show that 
   \begin{equation}\label{beta-j}
 \begin{split}
 z^{-\frac{n-2}2}
&\Big|\int_0^\pi e^{-iz\cos(s)}  \chi^c_\delta(\pi-s) \beta(2^{-j}z^{1/2}s)\\
&\times \int_{0}^\infty b_\pm(\rho d_h) e^{\pm i \rho d_h} a(s, y_1,y_2; \rho) \cos(s \rho) \rho^{n-2}d\rho ds\Big| \lesssim 2^{-j(n-2)},
\end{split}
 \end{equation}
which would give us the desired bound after summing over $j$ when $n\geq3$.
 Now we repeat the previous argument, if in this case we have $\rho\le 2^{-j}z^{1/2}$, then we do not do any integration by parts, the integral in \eqref{beta-j} is always bounded by 
 \begin{equation}\nonumber z^{-\frac{n-2}2} (z^{-\frac12}2^j) (2^{-j}z^{\frac{1}{2}})^{n-1} \lesssim 2^{-j(n-2)}.
 \end{equation}
 
On the other hand, if we have $\rho\ge 2^{-j}z^{1/2}$, we write $\cos(s \rho)=\frac12\big(e^{is\rho}+e^{-is\rho}\big)$, then we integrate by parts in $\rho$ instead \footnote{To rigorously justify the argument near the boundary at $\rho=+\infty$, one may further introduce a dyadic decomposition in $\rho$ to localize the analysis, on which we omit the details. The boundary term at $\rho=+\infty$ can be dropped since this equality is interpreted as for oscillatory integrals and one only needs to pair with functions with sufficient decay in $\rho$. Also, boundary terms at $\rho=0$ are absent since we truncated to a region with $\rho$ lower bounded by a constant like $2^{-j}z^{1/2}$. }, then each time we gain a factor of $\rho^{-1}$, and we at most lose a factor of $(s \pm d_h)^{-1}$. 
All boundary terms vanish by the same argument as after \eqref{rho-l} except that with $\beta_J(z^{1/2}s)$ replaced by $\beta(2^{-j}z^{1/2}s)$. The argument there still applies because all derivatives of $\beta(2^{-j}z^{1/2}s)$ still vanish at $0$ and other steps remain the same. 
Recalling that $J$ is large enough so that $2^{J-2}$ is larger than $C_1$, we have
$$|s\pm d_h|^{-1}\lesssim \Big((2^{j-1}-C_1)z^{-\frac12}\Big)^{-1}\lesssim \Big((2^{j-2}+2^{J-2}-C_1)z^{-\frac12}\Big)^{-1}\sim 2^{-j}z^{\frac12}.$$ 
So after integration by parts $N$ times for $N\ge n$, the integral in \eqref{beta-j} is bounded by 
 \begin{equation}\nonumber
z^{-\frac{n-2}2} (z^{-\frac12}2^j)  \big(2^{-j}z^{\frac12}\big)^{N}  \int_{2^{-j}z^{1/2}}^\infty \rho^{n-2-N}d\rho\lesssim 2^{-j(n-2)},
  \end{equation}
  where the first $(z^{-\frac12}2^j)$-factor is due to the length of the $s$-interval.

 \bigskip

  \textbf{Case 2}. $d_h(y_1,y_2)\geq C_1 z^{-\frac12}$.  In this case, taking $J=0$,
  we will show that 
 \begin{equation}\label{osi-1>b}
 \begin{split}
 z^{-\frac{n-2}2}
\Big|\int_0^\pi &e^{-iz\cos(s)} \chi^c_\delta(\pi-s)\Big(\beta_{0}(z d_h |s-d_h|)+\sum_{j\ge 1}\beta(2^{-j}z d_h |s-d_h|)\Big) \\
&\times \int_{0}^\infty b_\pm(\rho d_h) e^{\pm i \rho d_h} a(s, y_1,y_2; \rho) \cos(s \rho) \rho^{n-2}d\rho\, ds \Big|\lesssim 1,
\end{split}
 \end{equation}
where $\beta_0$ and $\beta$ are the same as the above ones in \eqref{beta-d}.

 For the term associated with $\beta_0$, we have $|s-d_h|\leq (z d_h)^{-1}\lesssim z^{-\frac12}$ due to the compact support of $\beta_0$. If we also have $\rho\le z d_h$,  then the integral in \eqref{osi-1>b} with $\beta_0$ is always bounded by 
  \begin{equation}
 \begin{split}
 z^{-\frac{n-2}2}
&\int_{|s-d_h|\lesssim (zd_h)^{-1}} ds
\int_{\rho\leq z d_h} (1+\rho d_h)^{-\frac{n-2}2} \rho^{n-2}d\rho 
\\&\lesssim z^{-\frac{n-2}2} (zd_h)^{-1} (zd_h)^{\frac{n-2}{2}+1} d_h^{-\frac{n-2}2}\lesssim 1.
\end{split}
 \end{equation}
On the other hand, if we have $\rho\ge z d_h$, we integrate by parts in $s$. Due to the support of $\chi_\delta^c(\pi-s)$, the term at the boundary $s=\pi$ still vanishes. While at $s=0$, the boundary term also vanishes.
Indeed, due to the support of $\beta_0$,  one has $|s-d_h|\leq 2(z d_h)^{-1}\leq  2C_1^{-1}z^{-1/2}$ which implies $s\geq C_1\big(1- 2C_1^{-2}\big)z^{-1/2}>0$ if $C_1$ is large enough.
So each time we gain a factor of $\rho^{-1}$ from the function $\cos(s \rho)$.
Next we consider the loss introduced by differentiating other factors in the integrand, which is a sum of terms of the form (modulo uniformly bounded smooth factors)
 \begin{equation} \label{eq:derivative-bound-3}
 (zd_h)^{k_1} (z \sin s)^{k_2} z^{k_3}P(\cos s,\sin s),
 \end{equation} 
 where $P(\cdot)$ is a polynomial. Here $k_1$ is the number of derivatives falling on the $\beta_0$-factor and $k_2+k_3$ is the number of times that the derivative falls on $e^{-iz\cos s}$. Notice that a $z$ factor without $\sin s$ paired to it can only arise by differentiating $\sin s$ (or its power), so $z^{k_3}$ has also cost $k_3$ derivatives on $\sin s$ and we have
 \begin{equation}
k_1 + k_2 + 2k_3 \leq N.
 \end{equation}
By the discussion above, and the assumption $d_h(y_1,y_2)\geq C_1 z^{-\frac12}$ in the current case, we have
$$z\sin s\lesssim z (d_h+z^{-\frac12})\lesssim z d_h, \quad z^{1/2} \lesssim zd_h. $$
So summing terms in \eqref{eq:derivative-bound-3} gives
\begin{equation} \label{eq:N-D-amp-zdh}
\Big|\Big(\frac{d}{ds}\Big)^N\Big(e^{-iz\cos(s)}\chi^c_\delta(\pi-s)\beta_{0}(z d_h |s-d_h|)\Big)\Big|\leq C_N (z d_h)^N.
\end{equation}

So after integration by parts $N$ times for $N\ge n$, 
  the integral in \eqref{osi-1>b} is bounded by 
 $$z^{-\frac{n-2}2} (z d_h)^{-1} (z d_h)^{N}  d_h^{-\frac{n-2}2}\int_{z d_h}^\infty \rho^{\frac{n-2}2-N}d\rho\lesssim 1.
 $$
In sum, we have proved
 \begin{equation}
 \begin{split}
 z^{-\frac{n-2}2}
\Big|\int_0^\pi &e^{-iz\cos(s)} \chi^c_\delta(\pi-s)\beta_{0}(z d_h |s-d_h|)\\
&\times \int_{0}^\infty b_\pm(\rho d_h) e^{\pm i \rho d_h} a(s, y_1,y_2; \rho) \cos(s \rho) \rho^{n-2}d\rho ds\Big| \lesssim 1.
\end{split}
 \end{equation}
 
For the terms associated with $\beta(2^{-j}z d_h |s-d_h|), j \geq 1$, we have $|s-d_h| \approx 2^{j}(z d_h)^{-1}$, due to the support condition of $\beta$, and $2^j\lesssim z d_h$ since $s,d_h$ are bounded. In this case, we will show that
\begin{equation}\label{beta-j'}
\begin{split}
 z^{-\frac{n-2}2}
&\Big|\int_0^\pi e^{-iz\cos(s)}\chi^c_\delta(\pi-s) \beta(2^{-j}z d_h |s-d_h|)\\
&\times \int_{0}^\infty b_\pm(\rho d_h) e^{\pm i \rho d_h} a(s, y_1,y_2; \rho) \cos(s \rho) \rho^{n-2}d\rho ds\Big| \lesssim 2^{-j\frac{n-2}2},
\end{split}
 \end{equation}
which would give us the desired bounds \eqref{osi-1>b} after summing over $j\geq1$.
 Now we repeat the previous argument, if in this case we have $\rho\le 2^{-j}z d_h$, then we do not do any integration by parts, the integral in \eqref{beta-j'} is always bounded by 
 \begin{equation}\nonumber 
 \begin{split}
& z^{-\frac{n-2}2} \int_{|s-d_h|\sim 2^j(zd_h)^{-1}}\int_{\rho\leq 2^{-j}z d_h}(1+\rho d_h)^{-\frac{n-2}2} \rho^{n-2}\, d\rho\\
 &\lesssim z^{-\frac{n-2}2} ((zd_h)^{-1}2^j) (2^{-j}z d_h)^{\frac{n-2}2+1} d_h^{-\frac{n-2}2} \lesssim 2^{-j\frac{n-2}2}.
 \end{split}
 \end{equation}

On the other hand, if we have $\rho\ge 2^{-j}z d_h$, we write $\cos(s \rho)=\frac12\big(e^{is\rho}+e^{-is\rho}\big)$, then we integrate by parts in $\rho$ again, then each time we gain a factor of $\rho^{-1}$, and we at most lose a factor of 
$$(s \pm d_h)^{-1}\lesssim  2^{-j}z d_h,$$ so after integration by parts $N$ times for $N\ge n$, 
 the integral in \eqref{beta-j'} is bounded by
 \begin{equation}\nonumber
 \begin{split}
&z^{-\frac{n-2}2} 2^j (z d_h)^{-1}   \big(2^{-j}z d_h\big)^{N}  \int_{2^{-j}z d_h}^\infty \rho^{\frac{n-2}2-N} d_h^{-\frac{n-2}2}d\rho\\
&\lesssim (z d_h)^{-\frac{n-2}2-1} 2^j  \big(2^{-j}z d_h\big)^{N} \big(2^{-j}z d_h\big)^{\frac{n-2}2+1-N} \lesssim 2^{-j\frac{n-2}2}.
\end{split}
\end{equation}
Therefore we have proved \eqref{osi-1} and this proves the uniform boundedness of the contribution of \eqref{S1}. 
\end{proof}

\vspace{0.2cm}

\begin{proof}[The contribution of \eqref{S2}] 
Since this term contains the boundary terms from the first part at $s=\pi$ and from the second part at $s=0$ which do not vanish in contrast to \eqref{S1}, the proof needs to deal with those boundary terms. 
The fortunate fact is that the boundary term of the first term at $s=\pi$ is the same as the boundary term of the second term at $s=0$, which leads to the cancellation of the singularity at the boundary.

Recall \begin{equation}
\begin{split}
I_{GD}(z;y_1,y_2):=&\frac{z^{-\frac{n-2}{2}}}{\pi}\int_0^\pi  e^{-iz\cos s}\chi_\delta(\pi-s) \cos(s\sqrt{P}) ds\\
&-\frac{z^{-\frac{n-2}{2}}\sin(\pi\sqrt{P})}{\pi}\int_0^\infty  e^{iz\cosh s}\chi_\delta(s) e^{-s\sqrt{P}} ds. 
\end{split}
\end{equation}
\begin{equation} 
\cos(s \sqrt{P})(y_1,y_2)=K_N(s; y_1,y_2)+R_N(s; y_1,y_2),
\end{equation}
where $K_N(s; y_1,y_2), R_N(s; y_1,y_2)$ are as in Corollary~\ref{lemma: parametrix 1}.

Before estimating it, we use integration by parts to obtain the following property of $I_{GD}(z;y_1,y_2)$ on the amplitude level.

\begin{lemma}\label{lem:in-parts} For any $m\in \mathbb{N}$, we have the following identity
\begin{equation}\label{in-parts}
\begin{split}
&\frac1{\pi}\int_0^\pi e^{-iz\cos s} \chi_\delta(\pi-s)\cos(\nu s) ds-\frac{\sin(\nu\pi)}{\pi}\int_0^\infty e^{iz\cosh s} \chi_\delta(s) e^{-s\nu} ds\\
&=\frac{(-1)^{m}}{\pi}\int_0^\pi \Big(\frac{\partial}{\partial s}\Big)^{2m}\big( e^{-iz\cos s}\chi_\delta(\pi-s)\big) \frac{ \cos(\nu s)}{\nu^{2m}} ds
\\&\qquad-\frac{\sin(\nu\pi)}{\pi}\int_0^\infty \Big(\frac{\partial}{\partial s}\Big)^{2m}\big( e^{iz\cosh s}\chi_\delta(s)\big) \frac{e^{-s\nu}}{\nu^{2m}} ds.
\end{split}
\end{equation}
\end{lemma}
\begin{proof}This lemma, a variant of \cite[(5.30)]{Na}, can be proved by using integration by parts and an induction argument. 
We first verify $m=1$. By integration by parts, we have
\begin{equation*}
\begin{split}
 &\frac1{\pi}\int_0^\pi e^{-iz\cos s} \chi_\delta(\pi-s) \cos(\nu s) ds-\frac{\sin(\nu\pi)}{\pi}\int_0^\infty e^{iz\cosh s} \chi_\delta(s)  e^{-s\nu} ds\\
& =\frac1{\pi}\big( e^{-iz\cos s}  \chi_\delta(\pi-s)\big) \frac{ \sin(\nu s)}{\nu}\Big|_{s=0}^{s=\pi} \\
&\quad+ \frac{(-1)}{\pi}\int_0^\pi \Big(\frac{\partial}{\partial s}\Big)\big( e^{-iz\cos s} \chi_\delta(\pi-s) \big) \frac{ \sin(\nu s)}{\nu} ds
\\&+\frac{\sin(\nu\pi)}{\pi}\big(  e^{iz\cosh s}  \chi_\delta(s) \big) \frac{e^{-s\nu}}{\nu}\Big|_{s=0}^\infty
-\frac{\sin(\nu\pi)}{\pi}\int_0^\infty \Big(\frac{\partial}{\partial s}\Big)\big(  e^{iz\cosh s}  \chi_\delta(s) \big) \frac{e^{-s\nu}}{\nu} ds.
\end{split}
\end{equation*}
We note that the boundary term 
\begin{equation*}
\begin{split}
&\frac1{\pi}\big( e^{-iz\cos s} \chi_\delta(\pi-s)   \big) \frac{ \sin(\nu s)}{\nu}\Big|_{s=0}^{s=\pi}
+\frac{\sin(\nu\pi)}{\pi}\big(  e^{iz\cosh s} \chi_\delta(s)  \big) \frac{e^{-s\nu}}{\nu}\Big|_{s=0}^\infty\\
&=\frac1{\pi}\big( e^{-iz\cos s} \chi_\delta(\pi-s)  \big) \frac{ \sin(\nu s)}{\nu}\Big|_{s=\pi}-
\frac{\sin(\nu\pi)}{\pi}\big(  e^{iz\cosh s} \chi_\delta(s)  \big) \frac{e^{-s\nu}}{\nu}\Big|_{s=0}=0.
\end{split}
\end{equation*}
By integration by parts again, we have
\begin{equation*}
\begin{split}
&\frac1{\pi}\int_0^\pi e^{-iz\cos s} \chi_\delta(\pi-s) \cos(\nu s) ds-\frac{\sin(\nu\pi)}{\pi}\int_0^\infty e^{iz\cosh s} \chi_\delta(s)  e^{-s\nu} ds\\
& =\frac1{\pi}\Big(\frac{\partial}{\partial s}\Big)\big( e^{-iz\cos s} \chi_\delta(\pi-s) \big)\frac{ \cos(\nu s)}{\nu^2}\Big|_{s=0}^{s=\pi} \\
&\quad+ \frac{(-1)}{\pi}\int_0^\pi \Big(\frac{\partial}{\partial s}\Big)^2\big( e^{-iz\cos s} \chi_\delta(\pi-s) \big)\frac{ \cos(\nu s)}{\nu^2} ds
\\&+\frac{\sin(\nu\pi)}{\pi}\Big(\frac{\partial}{\partial s}\Big)\big( e^{iz\cosh s}  \chi_\delta(s) \big) \frac{e^{-s\nu}}{\nu^2}\Big|_{s=0}^\infty
-\frac{\sin(\nu\pi)}{\pi}\int_0^\infty \Big(\frac{\partial}{\partial s}\Big)^2\big(  e^{iz\cosh s}  \chi_\delta(s)  \big) \frac{e^{-s\nu}}{\nu^2} ds.
\end{split}
\end{equation*}
If the derivative hits $e^{-iz\cos s}$ and $e^{iz\cosh s} $, it brings $\sin s$ and $\sinh s$ respectively, the boundary term vanishes due to the fact $\sin\pi=\sinh 0=0$ and $\nu\geq\nu_0>0$.
More precisely we observe that the boundary term
\begin{equation*}
\begin{split}
&\frac1{\pi}\Big(\frac{\partial}{\partial s}\Big)\big( e^{-iz\cos s} \chi_\delta(\pi-s) \big)\frac{ \cos(\nu s)}{\nu^2}\Big|_{s=0}^{s=\pi} 
+\frac{\sin(\nu\pi)}{\pi}\Big(\frac{\partial}{\partial s}\Big)\big( e^{iz\cosh s}  \chi_\delta(s) \big) \frac{e^{-s\nu}}{\nu^2}\Big|_{s=0}^\infty\\
&=\frac1{\pi}\Big(\frac{\partial}{\partial s}\Big)\big( e^{-iz\cos s} \chi_\delta(\pi-s) \big)\frac{ \cos(\nu s)}{\nu^2}\Big|_{s=\pi} 
-\frac{\sin(\nu\pi)}{\pi}\Big(\frac{\partial}{\partial s}\Big)\big( e^{iz\cosh s}  \chi_\delta(s) \big) \frac{e^{-s\nu}}{\nu^2}\Big|_{s=0}
\end{split}
\end{equation*}
vanishes due to the fact $\sin\pi=\sinh 0=\sinh s\, e^{-\nu s}\big|_{s=\infty}=0$. Therefore, we have proved \eqref{in-parts} with $m=1$. Now we assume  \eqref{in-parts} holds for $m=k$, that is,
\begin{equation*}
\begin{split}
&\frac1{\pi} \int_0^\pi e^{-iz\cos s} \chi_\delta(\pi-s) \cos(\nu s) ds-\frac{\sin(\nu\pi)}{\pi}\int_0^\infty e^{iz\cosh s} \chi_\delta(s)  e^{-s\nu} ds\\&=\frac{(-1)^{k}}{\pi}\int_0^\pi \Big(\frac{\partial}{\partial s}\Big)^{2k}\big( e^{-iz\cos s}\chi_\delta(\pi-s) \big) \frac{ \cos(\nu s)}{\nu^{2k}} ds
\\&\qquad-\frac{\sin(\nu\pi)}{\pi}\int_0^\infty \Big(\frac{\partial}{\partial s}\Big)^{2k}\big( e^{iz\cosh s}  \chi_\delta(s)  \big) \frac{e^{-s\nu}}{\nu^{2k}} ds,
\end{split}
\end{equation*}
we aim to prove  \eqref{in-parts} when $m=k+1$. To this end, it suffices to check the boundary terms vanish. Indeed,
\begin{equation*}
\begin{split}
&\frac{(-1)^k}{\pi}\Big(\frac{\partial}{\partial s}\Big)^{2k}\big( e^{-iz\cos s}  \chi_\delta(\pi-s)  \big) \frac{ \sin(\nu s)}{\nu^{2k+1}}\Big|_{s=0}^{s=\pi} 
\\ \qquad &+\frac{\sin(\nu\pi)}{\pi}\Big(\frac{\partial}{\partial s}\Big)^{2k}\big( e^{iz\cosh s}  \chi_\delta(s) \big) \frac{e^{-s\nu}}{\nu^{2k+1}}\Big|_{s=0}^\infty\\
&=\frac{(-1)^k}{\pi}\Big(\frac{\partial}{\partial s}\Big)^{2k}\big( e^{-iz\cos s} \chi_\delta(\pi-s)   \big) \frac{ \sin(\nu s)}{\nu^{2k+1}}\Big|_{s=\pi}\\
&\qquad-
\frac{\sin(\nu\pi)}{\pi}\Big(\frac{\partial}{\partial s}\Big)^{2k}\big( e^{iz\cosh s}  \chi_\delta(s)  \big)\frac{e^{-s\nu}}{\nu^{2k+1}}\Big|_{s=0}=0,
\end{split}
\end{equation*}
and
\begin{equation*}
\begin{split}
&\frac{(-1)^{k+1}}{\pi}\Big(\frac{\partial}{\partial s}\Big)^{2k+1}\big( e^{-iz\cos s}  \chi_\delta(\pi-s) \big) \frac{ \cos(\nu s)}{\nu^{2k+2}}\Big|_{s=0}^{s=\pi} 
\\ \qquad &+\frac{\sin(\nu\pi)}{\pi}\Big(\frac{\partial}{\partial s}\Big)^{2k+1}\big( e^{iz\cosh s} \chi_\delta(s)  \big) \frac{e^{-s\nu}}{\nu^{2k+2}}\Big|_{s=0}^\infty\\
&=\frac{(-1)^{k+1}}{\pi}\Big(\frac{\partial}{\partial s}\Big)^{2k+1}\big( e^{-iz\cos s}  \chi_\delta(\pi-s)  \big) \frac{ \cos(\nu s)}{\nu^{2k+2}}\Big|_{s=\pi}\\
&\qquad-
\frac{\sin(\nu\pi)}{\pi}\Big(\frac{\partial}{\partial s}\Big)^{2k+1}\big( e^{iz\cosh s} \chi_\delta(s) \big)\frac{e^{-s\nu}}{\nu^{2k+2}}\Big|_{s=0}=0,
\end{split}
\end{equation*}
where we used the following facts similar to equations in \cite[p. 420]{Na}:
\begin{equation} \label{eq: end point jet, even}
\begin{split}
&(-1)^k\Big(\frac{\partial}{\partial s}\Big)^{2k}\big( e^{-iz\cos s} \chi_\delta(\pi-s)  \big) \Big|_{s=\pi} =
\Big(\frac{\partial}{\partial s}\Big)^{2k}\big( e^{iz\cosh s} \chi_\delta(s)  \big)\Big|_{s=0},
\end{split}
\end{equation}
and
\begin{equation} \label{eq: end point jet, odd}
\begin{split}
(-1)^{k+1}\Big(\frac{\partial}{\partial s}\Big)^{2k+1}\big( e^{-iz\cos s}  \chi_\delta(\pi-s)  \big) \Big|_{s=\pi} =
\Big(\frac{\partial}{\partial s}\Big)^{2k+1}\big( e^{iz\cosh s} \chi_\delta(s) \big) \Big|_{s=0}.\end{split}
\end{equation}
Since both equations concern only the jet structure of these functions on the left and right hand sides at $\pi$ and $0$ respectively, near which the $\chi_\delta$-factors are identically 1, the $\chi_\delta$-factors have no effect. Thus we only need to show identities without $\chi_\delta$. 
Now setting
\begin{equation}
E_z(s):= e^{-i z \cos s},
\end{equation}
then we have
\begin{equation}
 e^{iz\cosh s} = E_z(is+\pi).
\end{equation}
By the even property of $E_z$ at $\pi$, and correspondingly the even property of $E_z(is+\pi)$ at 0, we know that the odd order terms vanish, hence \eqref{eq: end point jet, odd} holds. And \eqref{eq: end point jet, even} holds by the fact that the $2k-$th term in the Taylor expansion of $E_z(s)$ at $\pi$ and $E_z(is+\pi)$ at $0$ differs by a factor $i^{2k}=(-1)^k$.
\end{proof}

A direct consequence of the lemma above and the functional calculus is the following result on the operator level: 
\begin{corollary}\label{prop:in-parts} For any $m\in \mathbb{N}$,  it holds that
\begin{equation}\label{IGD:in-parts}
\begin{split}
I_{GD}(z;y_1,y_2) = &
\frac{z^{-\frac{n-2}{2}}(-1)^{m}}{\pi}\int_0^\pi \Big(\frac{\partial}{\partial s}\Big)^{2m}\big( e^{-iz\cos s} \chi_\delta(\pi-s)\big) \frac{ \cos(s\sqrt{P})}{P^{m}} ds
\\&\qquad-\frac{z^{-\frac{n-2}{2}}\sin(\pi\sqrt{P})}{\pi}\int_0^\infty \Big(\frac{\partial}{\partial s}\Big)^{2m}\big( e^{iz\cosh s}\chi_\delta(s)\big) \frac{e^{-s\sqrt{P}}}{P^{m}} ds,
\end{split}
\end{equation}
where $P=\Delta_h+V_0(y)+\frac{(n-2)^2}4$.
\end{corollary}

In fact, exploiting \eqref{eq:wave-poisson-jet-match}, we can have a more refined microlocalized version of this. 
We first define the frequency localized version of $K_{\pm}$ and $\tilde{K}_\pm$ (we fix the index $N$ in \eqref{KN1}\eqref{eq: poisson-0} and abbreviate it from now on). For $-1 \leq A \leq B \leq \infty$ and $\mk{d} \in \mk{D}(y_1,y_2)$, we set
\begin{equation}
\begin{split}
&K_{\pm,\mk{d}, [A,B]}(s,y_1,y_2) \\
&=  \sum_{\varsigma = \pm} \int_0^\infty \chi_{[A,B]}(\rho) b_{\varsigma}(\rho \mk{d}) e^{\varsigma i \rho \mk{d}} a_{\pm,\mk{d}}(s, y_1,y_2; \rho) e^{\pm  is\rho} \rho^{n-2} d\rho,
\end{split}
\end{equation}
where $\chi_{[A,B]}$ is a smooth cut-off function supported in  $[A,B]$ (we will be more specific in applications) and 
\begin{equation}
\begin{split}
&\tilde{K}_{\pm, \mk{d}, [A,B]}(s,y_1,y_2) \\
&=  \sum_{\varsigma = \pm}\int_0^\infty \chi_{[A,B]}(\rho) b_{\varsigma}(\rho \mk{d}) e^{\varsigma i \rho \mk{d}} \tilde{a}_{\pm,\mk{d}}(s, y_1,y_2; \rho) e^{(-s \pm  i\pi)\rho } \rho^{n-2} d\rho.
\end{split}
\end{equation}
Then we define the corresponding frequency localized version of $I_{GD}$ to be
\begin{equation}\label{I-GD-L}
\begin{split}
&I_{GD,[A,B]}(z;y_1,y_2)\\
&:= \sum_{\mk{d} \in \mk{D}(y_1,y_2)} \Big( \frac{z^{-\frac{n-2}{2}}}{\pi}\int_0^\pi  e^{-iz\cos s}\chi_\delta(\pi-s)  \frac{1}{2}\Big(K_{+, \mk{d}, [A,B]}(s,y_1,y_2)+K_{-, \mk{d}, [A,B]}(s,y_1,y_2)\Big)  ds\\
&-\frac{z^{-\frac{n-2}{2}}}{\pi}\int_0^\infty  e^{iz\cosh s} \chi_\delta(s)
\frac{1}{2i}\Big(\tilde{K}_{+, \mk{d}, [A,B]}(s,y_1,y_2)-\tilde{K}_{-, \mk{d}, [A,B]}(s,y_1,y_2)\Big) ds \Big).
\end{split}
\end{equation}
\begin{proposition} \label{prop:localized-GD-IBP} For any $m\geq0$, $I_{GD,[A,B]}$ defined in \eqref{I-GD-L} can be rewritten as
\begin{equation} \label{eq:localized-GD-IBP}
\begin{split}
&I_{GD,[A,B]}(z;y_1,y_2)\\
&=\sum_{\mk{d} \in \mk{D}(y_1,y_2)} \Big[\frac{1}{\pi}\int_0^\pi  
P_m(z,s) \frac{1}{2}\Big(K_{+,\mk{d}, [A,B],m}(s,y_1,y_2)+K_{-,\mk{d}, [A,B],m}(s,y_1,y_2)\Big)  ds\\
&-\frac{1}{\pi}\int_0^\infty  Q_m(z,s) \frac{1}{2i}\Big(\tilde{K}_{+, \mk{d}, [A,B],m}(s,y_1,y_2)-\tilde{K}_{-,\mk{d}, [A,B],m}(s,y_1,y_2)\Big) ds\Big],
\end{split}
\end{equation}
where $P_m(z,s)$ is a (linear) combination of derivatives of $e^{-iz\cos s}\chi_\delta(\pi-s)$ with respect to $s$ up to $m-$th order and 
$Q_m(z,s)$ is a (linear) combination of derivatives of $e^{iz\cosh s} \chi_\delta(s)$ with respect to $s$ up to $m-$th order.
And most importantly, for $\mk{d} \in \mk{D}(y_1,y_2)$,
\begin{align*}
\begin{split}
K_{\pm,\mk{d}, [A,B],m}(s,y_1,y_2) = &  \sum_{\varsigma = \pm} \int_0^\infty \chi_{[A,B]}(\rho) b_{\varsigma}(\rho \mk{d}) e^{\varsigma i \rho \mk{d}} a_{\pm,m,\mk{d}}(s, y_1,y_2; \rho) e^{\pm  is\rho} \rho^{n-2} d\rho,\\
\tilde{K}_{\pm, \mk{d}, [A,B],m}(s,y_1,y_2) = &\sum_{\varsigma = \pm}\int_0^\infty \chi_{[A,B]}(\rho) b_{\varsigma}(\rho \mk{d}) e^{\varsigma i \rho \mk{d}} \tilde{a}_{\pm,m,\mk{d}}(s, y_1,y_2; \rho) e^{(-s \pm  i\pi)\rho } \rho^{n-2} d\rho.
\end{split}
\end{align*}
where $a_{\pm,m,\mk{d}}, \tilde{a}_{\pm,m,\mk{d}} \in S^{-m}$ and the symbol order is in terms of $\rho$.
\end{proposition}

\begin{proof}
This follows from integrating by parts in $s$ by writing 
\begin{equation*}
e^{is\rho} = i^{-1}\rho^{-1}\partial_s(e^{is\rho}), 
\end{equation*}
and our amplitudes are supported in $\rho \geq 1$.
The boundary terms from the two parts cancel each other as in the proof of Lemma~\ref{lem:in-parts}, in combination with \eqref{eq:wave-poisson-jet-match}, which deals with terms having derivatives falling on $a$ and $\tilde{a}$.
\end{proof}

Now we split the kernel $I_{GD}(z;y_1,y_2)$ into two parts
\begin{equation}
\begin{split}
I_{GD}(z;y_1,y_2)=I^{<\kappa}_{GD}(z;y_1,y_2)+I^{>\kappa}_{GD}(z;y_1,y_2)
\end{split}
\end{equation}
where 
\begin{equation}
I^{<\kappa}_{GD}(z;y_1,y_2) = I_{GD,[-1,2\kappa]}(z;y_1,y_2),
\quad I^{>\kappa}_{GD}(z;y_1,y_2) = I_{GD,[\kappa,\infty]}(z;y_1,y_2),
\end{equation}
and we choose cut-off functions such that $\chi_{[-1,2\kappa]}+\chi_{[\kappa,\infty]}$ is identically $1$ on $[0,\infty)$. The part of $\chi_{[-1,2\kappa]}$ on $[-1,0)$ is unimportant, as long as it is smooth.
Explicitly, using \eqref{eq:localized-GD-IBP}, for any $m\geq0$, we have 
\begin{equation} \label{I<}
\begin{split}
&I^{<\kappa}_{GD}(z;y_1,y_2) \\
=&\sum_{\mk{d} \in \mk{D}(y_1,y_2)} \Big[\frac{1}{\pi}\int_0^\pi  
P_m(z,s) \frac{1}{2}\Big(K_{+,\mk{d}, [-1,2\kappa],m}(s,y_1,y_2)+K_{-,\mk{d}, [-1,2\kappa],m}(s,y_1,y_2)\Big)  ds\\
&-\frac{1}{\pi}\int_0^\infty  Q_m(z,s) \frac{1}{2i}\Big(\tilde{K}_{+, \mk{d}, [-1,2\kappa],m}(s,y_1,y_2)-\tilde{K}_{-,\mk{d}, [-1,2\kappa],m}(s,y_1,y_2)\Big) ds\Big],
\end{split}
\end{equation}
and
\begin{equation} \label{I>}
\begin{split}
&I^{>\kappa}_{GD}(z;y_1,y_2)\\
=&\sum_{\mk{d} \in \mk{D}(y_1,y_2)} \Big[\frac{1}{\pi}\int_0^\pi  
P_m(z,s) \frac{1}{2}\Big(K_{+,\mk{d}, [\kappa,+\infty],m}(s,y_1,y_2)+K_{-,\mk{d}, [\kappa,+\infty],m}(s,y_1,y_2)\Big)  ds\\
&-\frac{1}{\pi}\int_0^\infty  Q_m(z,s) \frac{1}{2i}\Big(\tilde{K}_{+, \mk{d}, [\kappa,+\infty],m}(s,y_1,y_2)-\tilde{K}_{-,\mk{d}, [\kappa,+\infty],m}(s,y_1,y_2)\Big) ds\Big].
\end{split}
\end{equation}

To control the contribution of \eqref{S2}, as in the argument for \eqref{S1}, we need the Hadamard parametrix. Noticing that
\begin{equation*}
\sin(\pi\sqrt{P})e^{-s\sqrt{P}}=\Im \big(e^{(-s+i\pi)\sqrt{P}}\big),
\end{equation*}
we can use the parametrix for Poisson-wave operators in Lemma~\ref{lemma:parametrix-poisson}. 

Now we return to the proof of the uniform boundedness of \eqref{S2}.
The contribution of the term associated with $\tilde{R}_N$ can be estimated by
\begin{equation}
\begin{split}
|z^{-\frac{n-2}2}\int_0^\infty e^{iz\cosh s}  \chi_\delta(s) \tilde{R}_N(s; y_1,y_2) ds| \lesssim 1.
\end{split}
\end{equation}
We will consider the contribution from $\tilde{K}_N$ below.

Similar to the proof of the uniform boundedness of \eqref{S1}, we only consider the case $\mk{d}=d_h(y_1,y_2)$ and divide it into two cases that $d_h(y_1,y_2)\leq C_1 z^{-\frac12}$ and $d_h(y_1,y_2) \geq C_1 z^{-\frac12}$ where $C_1\gg1$. 
When $\mk{d} \neq d_h(y_1,y_2)$, it could only have the second case $\mk{d}(y_1,y_2) \geq C_1 z^{-\frac{1}{2}}$ and the proof proceeds in the same manner as here.
In each case, we choose a different $\kappa$ in the argument. 
\vspace{0.2cm}

\textbf{Case 1}. $d_h(y_1,y_2)\leq C_1 z^{-\frac12}$. In this case, we take $\kappa=4z^{\frac12}$. We first consider $I^{<\kappa}_{GD}(z;y_1,y_2)$. 
For this low frequency term, since we do not integrate by parts in $s$ (so the boundary issue mentioned above will not be involved), we use \eqref{I<} with $m=0$. 
For the first term of \eqref{I<}, we can use the same argument as in the proof of \eqref{osi-1<b}.

For the term associated with $\tilde{K}_{\pm,N}$, we want to show that 
 \begin{equation}\label{osi-1<d}
 \begin{split}
 z^{-\frac{n-2}2}
&\Big|\int_0^\infty e^{iz\cosh s}  \chi_\delta(s) \Big(\beta_{J}(z^{1/2} s)+\sum_{j\ge J+1}\beta(2^{-j}z^{1/2}s)\Big) \\
&\times \int_{0}^{\infty} \chi_{[-1,2\kappa]}(\rho)  b_\pm(\rho d_h) e^{\pm i \rho d_h} \tilde{a}_{\pm}(s, y_1, y_2; \rho) e^{(-s\pm i\pi)\rho} \rho^{n-2}d\rho\, ds \Big|\lesssim 1,
\end{split}
 \end{equation}
 where $\beta$ and $\beta_J$ are in \eqref{beta-d} with $2^{J-2}\geq C_1$.
For the term associated with $\beta_J$, we have $|s|\lesssim z^{-\frac12}\ll 1$ due to the compact support of $\beta_J$. Since $\rho \le 2\kappa=8z^{1/2}$ in this part, the integral in \eqref{osi-1<d} with $\beta_J$ is always bounded by 
  \begin{equation}\label{d-rho-l}
 \begin{split}
 z^{-\frac{n-2}2}
\int_{|s|\lesssim z^{-\frac12}} ds
\int_{\rho\leq 8z^{\frac12}} \rho^{n-2}d\rho \lesssim z^{-\frac{n-2}2} z^{-1/2} z^{\frac{n-1}{2}}\lesssim 1.
\end{split}
 \end{equation}
 For the terms with $\beta(2^{-j}z^{1/2}s), j \geq J+1$, we have $s \approx 2^{j}z^{-1/2}$, and $2^j\lesssim  z^{1/2}$, due to the compact support of $\beta$.  In this case, we will show that 
\begin{equation}\label{beta-d-j}
\begin{split}
z^{-\frac{n-2}2}
&\Big|\int_0^\infty e^{iz\cosh s}  \chi_\delta(s)\beta(2^{-j}z^{1/2}s)\\
&\times \int_{0}^{\infty} \chi_{[-1,2\kappa]}(\rho) b_\pm(\rho d_h) e^{\pm i \rho d_h} 
\tilde{a}_{\pm}(s, y_1, y_2; \rho) e^{(-s\pm i\pi)\rho} \rho^{n-2}d\rho\, ds \Big| \lesssim 2^{-j(n-2)},
\end{split}
 \end{equation}
which would give us the desired bounds after summing over $j$.
 Now we repeat the previous argument, if in this case we have $\rho\le 2^{-j}z^{1/2}$, then we do not do any integration by parts, the integral in \eqref{beta-d-j} is always bounded by 
 \begin{equation}\nonumber z^{-\frac{n-2}2} (z^{-\frac12}2^j) (2^{-j}z^{\frac{1}{2}})^{n-1} \lesssim 2^{-j(n-2)}.
 \end{equation}

On the other hand, if we have $\rho\ge 2^{-j}z^{1/2}$, we use the factor $e^{(-s\pm i\pi)\rho}$ to integrate by parts in $\rho$ instead, then each time we gain a factor of $\rho^{-1}$, and we at most lose factors of 
\begin{equation*}
|s\pm i\pi|^{-1} \lesssim 1, \text{ or} \quad d_h\lesssim 1 .
\end{equation*}
So after integration by parts $N$ times for $N\ge n$, the integral in \eqref{beta-d-j} is bounded by 

\begin{equation}\nonumber
z^{-\frac{n-2}2} (z^{-\frac12}2^j)   \int_{2^{-j}z^{1/2}}^\infty \rho^{n-2-N}d\rho\lesssim z^{-\frac{n-2}2}   (2^{-j}z^{1/2})^{n-2-N} \lesssim 2^{-j(n-2)}
\end{equation}
because $2^j\lesssim  z^{1/2}$.

Next we consider $I^{>\kappa}_{GD}(z;y_1,y_2)$. For this high frequency part, we use Proposition~\ref{prop:localized-GD-IBP} (or more directly, use \eqref{I>}) with $m$ large enough. We need to show 
\begin{equation}\label{osi-1<gd}
\begin{split}
z^{-\frac{n-2}2}
\Big|\int_0^\pi & P_m(z,s) \Big(\beta_{J}(z^{1/2} s)+\sum_{j\ge J+1}\beta(2^{-j}z^{1/2}s)\Big) \\
&\times \int_{0}^\infty  \chi_{[\kappa,\infty]}(\rho) b_\pm(\rho d_h) e^{\pm i \rho d_h} a_{\pm,m}(s, y_1, y_2; \rho) \cos(s \rho) \rho^{n-2}d\rho\, ds \Big|\lesssim 1.
\end{split}
 \end{equation}
 For the term associated with $\beta_J$, we have $|s|\lesssim z^{-\frac12}\ll 1$ due to the compact support of $\beta_J$. Using the same proof of \eqref{eq:amplitude-N-derivative-bound}, we have
\begin{equation*}
\Big|\Big(\frac{d}{ds}\Big)^{2m}\Big(e^{-iz\cos(s)}  \chi_\delta(\pi-s)\Big)\Big|  \lesssim z^{m},
\end{equation*}
and this bounds the entire $P_m(z,s)$-factor.
So for $2m\ge n$, the integral in \eqref{osi-1<gd} is bounded by 
\begin{equation*}
z^{-\frac{n-2}2} z^{-1/2} z^{m}  \int_{z^{1/2}}^\infty \rho^{n-2-2m}d\rho\lesssim 1.
\end{equation*}
So we have proved \eqref{osi-1<gd} with $\beta_J$.

For the terms with $\beta(2^{-j}z^{1/2}s)$, we have $2^{j-1}z^{-1/2} \leq s \leq 2^{j+1}z^{-1/2}$ and $2^j \lesssim  z^{1/2}$ on its support. And by our construction we have $j \geq J+1$, hence $2^{j-2} > C_1$. 
Again by an argument similar to that before \eqref{eq:amplitude-N-derivative-bound}, we have
\begin{equation*}
\Big|\Big(\frac{d}{ds}\Big)^{2m}\Big(e^{-iz\cos(s)}  \chi_\delta(\pi-s)\Big)\Big|\lesssim C_m 2^{2mj} z^m.
\end{equation*}
  In this case, we will show that 
   \begin{equation}\label{beta-j-gd}
 \begin{split}
 z^{-\frac{n-2}2}
&\Big|\int_0^\pi P_m(z,s) \beta(2^{-j}z^{1/2}s) \\
&\times \int_{0}^\infty  \chi_{[\kappa,\infty]}(\rho) b_\pm(\rho d_h) e^{\pm i \rho d_h} a_{\pm,m}(s, y_1, y_2; \rho) \cos(s \rho) \rho^{n-2}d\rho\, ds\Big| \lesssim 2^{-j(n-2)},
\end{split}
 \end{equation}
which would give us the desired bounds after summing over $j$.
We have $\rho\ge\kappa =4z^{1/2}$ in this part. Writing $\cos(s \rho)=\frac12\big(e^{is\rho}+e^{-is\rho}\big)$, then we integrate by parts in $\rho$ instead, then each time we gain a factor of $\rho^{-1}$, and we at most lose a factor of (by our choice of $J$, $s$ will dominate $d_h$)
\begin{equation*}
|s\pm d_h|^{-1}\lesssim 2^{-j}z^{\frac12}.
\end{equation*}
So after integration by parts $N$ times for $N\ge n+2m$,  the integral in \eqref{beta-j-gd} is bounded by 
 \begin{equation}\nonumber
z^{-\frac{n-2}2} (z^{-\frac12}2^j) \big(2^{2mj} z^m\big)  \big(2^{-j}z^{\frac12}\big)^{N}  \int_{4z^{1/2}}^\infty \rho^{n-2-2m-N}d\rho\lesssim 2^{-j(n-2)},
  \end{equation}
which proves \eqref{osi-1<gd} for $n\geq 3$.

To treat the second term of \eqref{I>}, we closely follow the argument above but with minor modifications.
The desired estimate follows if we can show that 
 \begin{equation}\label{osi-1<gd-d}
 \begin{split}
 z^{-\frac{n-2}2}
\Big|\int_0^\pi &Q_m(z,s)\Big(\beta_{J}(z^{1/2} s)+\sum_{j\ge J+1}\beta(2^{-j}z^{1/2}s)\Big) \\
&\times \int_{0}^\infty  \chi_{[\kappa,\infty]}(\rho) b_\pm(\rho d_h) e^{\pm i \rho d_h} \tilde{a}_{\pm,m}(s, y_1, y_2; \rho) e^{(-s \pm i\pi)\rho} \rho^{n-2}d\rho\, ds \Big|\lesssim 1.
\end{split}
 \end{equation}
 For the term associated with $\beta_J$, we have $|s|\lesssim z^{-\frac12}\ll 1$ due to the compact support of $\beta_J$. Using the same discussion as before \eqref{eq:amplitude-N-derivative-bound}, we have
\begin{equation*}
\Big|\Big(\frac{d}{ds}\Big)^{2m}\Big( e^{iz\cosh s}  \chi_\delta(s)\Big)\Big|\leq C_m z^{m}.
\end{equation*}
So for $2m \ge n$, the integral in \eqref{osi-1<gd-d} is bounded by 
\begin{equation}
z^{-\frac{n-2}2} z^{-1/2} z^{m}  \int_{z^{1/2}}^\infty \rho^{n-2-2m}d\rho\lesssim 1,
\end{equation}
and the $Q_m$-factor is bounded by this. For the term with $\beta(2^{-j}z^{1/2}s)$, we have $s \approx 2^{j}z^{-1/2}$, and $2^j\lesssim  z^{1/2}$, due to the compact support of $\beta$. By the same discussion as before \eqref{eq:amplitude-N-derivative-bound} again, we have
\begin{equation*}
\Big|\Big(\frac{d}{ds}\Big)^{2m}\Big(e^{iz\cosh s}  \chi_\delta(s)\Big)\Big|\leq  C_m 2^{2mj} z^m.
 \end{equation*}
Since $\rho\ge\kappa =4z^{1/2}$, we use the factor $e^{(-s \pm i\pi)\rho}$ to integrate by parts in $\rho$ instead, then each time we gain a factor of $\rho^{-1}$, and we at most lose factors of 
$$|s\pm i\pi|^{-1},\, d_h \lesssim 1\lesssim 2^{-j}z^{\frac12},$$ so after integration by parts $N$ times for $N\ge n+2m$, the integral in \eqref{osi-1<gd-d} with $\beta(2^{-j}z^{1/2}s)$ is bounded by 
   \begin{equation*}
 \begin{split}
  z^{-\frac{n-2}2}
&\Big|\int_0^\pi Q_m(z,s) \beta(2^{-j}z^{1/2}s) \\
&\times \int_{0}^\infty  \chi_{[\kappa,\infty]}(\rho) b_\pm(\rho d_h) e^{\pm i \rho d_h} \tilde{a}_{\pm,m}(s, y_1, y_2; \rho) e^{(-s \pm i\pi)\rho} \rho^{n-2}d\rho\, ds\Big| \\&\lesssim 
z^{-\frac{n-2}2} (z^{-\frac12}2^j) \big(2^{2mj} z^m\big)  \big(2^{-j}z^{\frac12}\big)^{N}  \int_{4z^{1/2}}^\infty \rho^{n-2-2m-N}d\rho\lesssim 2^{-j(n-2)},
\end{split}
 \end{equation*}
which would give us the desired bounds \eqref{osi-1<gd-d} after summing over $j$ provided $n\geq3$.\vspace{0.2cm}

\textbf{Case 2}. $d_h(y_1,y_2)\geq C_1 z^{-\frac12}$.  In this case, we take $\kappa=zd_h$ and $J=0$ in $\beta_J$ \eqref{beta-d}. We first consider $I^{<\kappa}_{GD}(z;y_1,y_2)$. One can control the first term of \eqref{I<} as in the proof of \eqref{osi-1>b}, since we do not use integration by parts in $s$. 
For the second term of \eqref{I<}, 
we want to show that 
 \begin{equation}\label{osi-1<d2}
 \begin{split}
 z^{-\frac{n-2}2}
&\Big|\int_0^\infty e^{iz\cosh s}  \chi_\delta(s) \Big(\beta_{0}(zd_h |s-d_h|)+\sum_{j\ge 1}\beta(2^{-j}zd_h |s-d_h|)\Big) \\
&\times \int_{0}^{\infty} \chi_{[-1,2\kappa]}(\rho) b_\pm(\rho d_h) e^{\pm i \rho d_h} \tilde{a}_\pm(s, y_1, y_2; \rho) e^{(-s\pm i\pi)\rho} \rho^{n-2}d\rho\, ds \Big|\lesssim 1.
\end{split}
 \end{equation}
For the term associated with $\beta_0$, we have $|s-d_h|\leq (zd_h)^{-1}\lesssim z^{-\frac12}\ll 1$ due to the compact support of $\beta_0$. Since $\rho\le 2\kappa=2zd_h$ in the current part, the integral in \eqref{osi-1<d2} with $\beta_0$ is always bounded by 
  \begin{equation}\label{d-rho-l2}
 \begin{split}
 z^{-\frac{n-2}2}
&\int_{|s-d_h|\lesssim (zd_h)^{-1}} ds
\int_{\rho\leq zd_h} (1+\rho d_h)^{-\frac{n-2}2}\rho^{n-2}d\rho 
\\&\lesssim z^{-\frac{n-2}2} (zd_h)^{-1} (zd_h)^{\frac{n-2}{2}+1} d_h^{-\frac{n-2}2}\lesssim 1.
\end{split}
 \end{equation}
 For the terms associated with $\beta(2^{-j}z d_h |s-d_h|)$, we have $|s-d_h| \approx 2^{j}(z d_h)^{-1}$, $j\ge 1$ and $2^j\lesssim z d_h$, due to the compact support of $\beta$.  In this case, we want to show that 
   \begin{equation}\label{beta-j'2}
 \begin{split}
 z^{-\frac{n-2}2}
&\Big|\int_0^\infty e^{iz\cosh s}  \chi_\delta(s) \beta(2^{-j}z d_h |s-d_h|)\\
&\times \int_{0}^{\infty} \chi_{[-1,2\kappa]}(\rho) b_\pm(\rho d_h) e^{\pm i \rho d_h} \tilde{a}_\pm(s, y_1, y_2; \rho) e^{(-s\pm i\pi)\rho} \rho^{n-2}d\rho\, ds\Big| \lesssim 2^{-j\frac{n-2}2},
\end{split}
 \end{equation}
which would give us the desired bounds \eqref{osi-1<d2} after summing over $j$.
For the part $\rho\le 2^{-j}z d_h$, we do not integrate by parts. The integral in \eqref{beta-j'2} is always bounded by 
 \begin{equation}\nonumber 
 \begin{split}
& z^{-\frac{n-2}2} \int_{|s-d_h|\sim 2^j(zd_h)^{-1}}\int_{\rho\leq 2^{-j}z d_h}(1+\rho d_h)^{-\frac{n-2}2} \rho^{n-2}\, d\rho\\
 &\lesssim z^{-\frac{n-2}2} ((zd_h)^{-1}2^j) (2^{-j}z d_h)^{\frac{n-2}2+1} d_h^{-\frac{n-2}2} \lesssim 2^{-j\frac{n-2}2}.
 \end{split}
 \end{equation}

On the other hand, for the part with $\rho\ge 2^{-j}z d_h$, we use the factor $e^{(-s \pm i\pi)\rho}$ to integrate by parts in $\rho$ instead, then each time we gain a factor of $\rho^{-1}$, and we at most lose factors of 
$$|s\pm i\pi|^{-1},\, d_h \lesssim 1\lesssim 2^{-j}z d_h,$$ so after integration by parts $N$ times for $N\ge n$, 
 the integral in \eqref{beta-j'2} is bounded by 
 \begin{equation}\nonumber
 \begin{split}
&z^{-\frac{n-2}2} 2^j (z d_h)^{-1}   \big(2^{-j}z d_h\big)^{N}  \int_{2^{-j}z d_h}^\infty \rho^{\frac{n-2}2-N} d_h^{-\frac{n-2}2}d\rho\\
&\lesssim (z d_h)^{-\frac{n-2}2-1} 2^j  \big(2^{-j}z d_h\big)^{N} \big(2^{-j}z d_h\big)^{\frac{n-2}2+1-N} \lesssim 2^{-j\frac{n-2}2}.
\end{split}
  \end{equation}
   \vspace{0.2cm}

Next we consider $I^{>\kappa}_{GD}(z;y_1,y_2)$. Using \eqref{I>}, we will show that 
 \begin{equation}\label{osi-1>gd}
 \begin{split}
 z^{-\frac{n-2}2}
\Big|\int_0^\pi & P_m(z,s)
\Big(\beta_{0}(zd_h |s-d_h|)+\sum_{j\ge 1}\beta(2^{-j}zd_h |s-d_h|)\Big) \\
&\times \int_{0}^\infty  \chi_{[\kappa,\infty]}(\rho) b_\pm(\rho d_h) e^{\pm i \rho d_h} a_{\pm,m}(s, y_1, y_2; \rho) \cos(s \rho) \rho^{n-2}d\rho\, ds \Big|\lesssim 1.
\end{split}
 \end{equation}
 For the term associated with $\beta_0$, we have $|s-d_h|\leq (zd_h)^{-1}\lesssim z^{-\frac12}\ll 1$ due to the compact support of $\beta_0$. Therefore, $s \lesssim d_h+(zd_h)^{-1}$ and $z^{1/2} \lesssim zd_h$ on this region and in the same manner as the discussion before \eqref{eq:N-D-amp-zdh}, we have
 $$\Big|\Big(\frac{d}{ds}\Big)^{2m}\Big(e^{-iz\cos(s)}  \chi_\delta(\pi-s)\Big)\Big| \lesssim (zd_h)^{2m},$$
and the $P_m$-factor is controlled by this.
So for $2m\ge n$, the integral in \eqref{osi-1>gd} is bounded by 
 $$z^{-\frac{n-2}2} (zd_h)^{-1} (zd_h)^{2m}  \int_{zd_h}^\infty (1+\rho d_h)^{-\frac{n-2}2} \rho^{n-2-2m}d\rho\lesssim 1,
 $$
and we have proved \eqref{osi-1>gd} with $\beta_0$. 
 For the terms with $\beta(2^{-j}zd_h |s-d_h|)$, we have $|s-d_h| \approx 2^{j}(zd_h)^{-1}$, and $2^j\lesssim  zd_h$, due to the compact support of $\beta$. Therefore, $s \lesssim d_h+2^j(zd_h)^{-1}$, and by the same discussion before \eqref{eq:N-D-amp-zdh}, we have
\begin{equation*}
\Big|\Big(\frac{d}{ds}\Big)^{2m}\Big(e^{-iz\cos(s)}  \chi_\delta(\pi-s)\Big)\Big| \lesssim C_m (zd_h+2^j z(zd_h)^{-1})^{2m},
\end{equation*}
 which is controlled by $(zd_h)^{2m}+(2^jz^{\frac12})^{2m}$. In this case, we will show that 
\begin{equation}  \label{eq:Pm-osc}
\begin{split}
 z^{-\frac{n-2}2}
&\Big|\int_0^\pi P_m(z,s)\beta(2^{-j}zd_h |s-d_h|) \\
&\times \int_{0}^\infty  \chi_{[\kappa,\infty]}(\rho) b_\pm(\rho d_h) e^{\pm i \rho d_h} a_{\pm,m}(s, y_1, y_2; \rho) \cos(s \rho) \rho^{n-2}d\rho\, ds\Big| \lesssim 2^{-j(n-2)},
\end{split}
 \end{equation}
which would give us the desired bounds after summing over $j$.

Since $\rho\ge\kappa =zd_h$ in this part, we write $\cos(s \rho)=\frac12\big(e^{is\rho}+e^{-is\rho}\big)$, then we integrate by parts in $\rho$ instead, then each time we gain a factor of $\rho^{-1}$, and we at most lose a factor of $$|s\pm d_h|^{-1}\lesssim 2^{-j}z d_h,$$ so after integration by parts $N$ times for $N\ge n+2m$, 
 the integral in \eqref{eq:Pm-osc} is bounded by 
 \begin{equation}\nonumber
 \begin{split}
&z^{-\frac{n-2}2} (2^j(zd_h)^{-1}) \big[(zd_h)^{2m}+(2^jz^{\frac12})^{2m}\big]  \big(2^{-j}z d_h\big)^{N}  \int_{zd_h}^\infty (1+\rho d_h)^{-\frac{n-2}2}\rho^{n-2-2m-N}d\rho\\
&\lesssim 2^{-j(N-2m-1)}.
\end{split}
  \end{equation}
Therefore, we have proved \eqref{osi-1>gd} for $n\geq 3$. 

For the second term of \eqref{I>}, 
we need to show that 
 \begin{equation} \label{osi-1<gd-d-2}
 \begin{split}
 z^{-\frac{n-2}2}
\Big|\int_0^\pi &Q_m(z,s) \Big(\beta_{0}(zd_h |s-d_h|)+\sum_{j\ge 1}\beta(2^{-j}zd_h |s-d_h|)\Big) \\
&\times \int_{0}^\infty  \chi_{[\kappa,\infty]}(\rho) b_\pm(\rho d_h) e^{\pm i \rho d_h} \tilde{a}_{\pm,m}(s, y_1, y_2; \rho) e^{(-s \pm i\pi)\rho} \rho^{n-2}d\rho\, ds \Big|\lesssim 1.
\end{split}
 \end{equation}
 For the term associated with $\beta_0$, we have $|s-d_h|\leq (zd_h)^{-1}\lesssim z^{-\frac12}\ll 1$ due to the compact support of $\beta_0$. Therefore, we know $s\lesssim d_h+(zd_h)^{-1}\lesssim d_h$ and $z^{1/2} \lesssim z d_h$. Again using the discussion before \eqref{eq:N-D-amp-zdh}, we have
 $$\Big|\Big(\frac{d}{ds}\Big)^{2m}\Big( e^{iz\cosh s}  \chi_\delta(s)\Big)\Big|\leq C_m (z d_h)^{2m},$$
 and the entire $Q_m$-factor is bounded by this.
So for $2m\ge n$, the integral in \eqref{osi-1<gd-d-2} is bounded by 
 $$z^{-\frac{n-2}2} (zd_h)^{-1} (z d_h)^{2m} \int_{z d_h}^\infty (1+\rho d_h)^{-\frac{n-2}{2}}\rho^{n-2-2m}d\rho\lesssim 1.
 $$
 For the terms with $\beta(2^{-j}zd_h |s-d_h|)$, we have $|s-d_h| \approx 2^{j}(zd_h)^{-1}$, and $2^j\lesssim  zd_h$, due to the compact support of $\beta$. 
 Therefore, $s \lesssim d_h+2^j(zd_h)^{-1}\in [0,\delta]$, again by the aforementioned discussion for \eqref{eq:N-D-amp-zdh}, we have
 $$\Big|\Big(\frac{d}{ds}\Big)^{2m}\Big(e^{iz\cosh(s)}  \chi_\delta(s)\Big)\Big| \leq C_m (zd_h+2^j z(zd_h)^{-1})^{2m},$$
 which is less than $(zd_h)^{2m}+(2^jz^{\frac12})^{2m}$.
In this case, we will show that 
 \begin{equation}\label{beta-j-gd-d2}
 \begin{split}
 z^{-\frac{n-2}2}
&\Big|\int_0^\pi Q_m(z,s)\beta(2^{-j}zd_h |s-d_h|) \\
&\times \int_{0}^\infty  \chi_{[\kappa,\infty]}(\rho) b_\pm(\rho d_h) e^{\pm i \rho d_h} \tilde{a}_{\pm,m}(s, y_1, y_2; \rho) e^{(-s\pm i\pi)\rho} \rho^{n-2}d\rho\, ds\Big| \lesssim 2^{-j(n-2)},
\end{split}
 \end{equation}
which would give us the desired bounds after summing over $j$.
Since $\rho\ge\kappa =zd_h$ due to the $\chi_{[\kappa,\infty]}$-factor, we use the factor $e^{(-s \pm i\pi)\rho}$ to integrate by parts in $\rho$ instead, then each time we gain a factor of $\rho^{-1}$, and we at most lose factors of 
$$|s\pm i\pi|^{-1},\, d_h \lesssim 1\lesssim 2^{-j}zd_h.$$ 
So after integration by parts $N$ times for $N\ge n+2m$, 
 the integral in \eqref{beta-j-gd-d2} is bounded by 
  \begin{equation}\nonumber
 \begin{split}
&z^{-\frac{n-2}2} (2^j(zd_h)^{-1}) \big[(zd_h)^{2m}+(2^jz^{\frac12})^{2m}\big]  \big(2^{-j}z d_h\big)^{N}  \int_{zd_h}^\infty (1+\rho d_h)^{-\frac{n-2}2}\rho^{n-2-2m-N}d\rho\\
&\lesssim 2^{-j(N-2m-1)},
\end{split}
  \end{equation}
which would give us the desired bounds \eqref{osi-1<gd-d-2} after summing over $j$ provided $n\geq3$.
\end{proof}

\begin{proof}[The contribution of \eqref{S3}] This term is easier than the two terms above.
By the definition of $I_{D}(z;y_1,y_2)$ in \eqref{S3}, it is a direct consequence of the following lemma. Indeed, for $n\geq3$, we have
$$|I_{D}(z;y_1,y_2)|\lesssim z^{-\frac{n-2}2}\int_\delta^\infty  |e^{(-s\pm i\pi)\sqrt{P}}| ds\lesssim_\delta 1,$$
by the following lemma:
 \begin{lemma} \label{lemma:Poisson-bound}
 Let $d_h=d_h(y_1,y_2)$ be the distance on $Y$. If $s\geq \delta$ where $0<\delta\ll 1$, then there is a constant $c>0$ such that the Poisson-wave operator satisfies 
 \begin{equation}\label{est:Poisson}
 |e^{(-s\pm i\pi)\sqrt{P}}| \lesssim e^{-cs}.
 \end{equation}
 \end{lemma}

\begin{proof}
By \eqref{est:eig} together with $\nu_0>0$ (by the strict positivity of $P$), there is a constant $c_0>0$ such that 
\begin{align}
\nu_k \geq c_0(1+k)^{\frac{1}{n-1}}.
\end{align}
Combining with \eqref{eq:eigen-L-infinity} and \eqref{FA}, we know 
\begin{align}
\begin{split}
|e^{(-s\pm i\pi)\sqrt{P}}|= & |\sum_{k\in\mathbb{N}}\varphi_k(y_1)\overline{\varphi_k(y_2)}e^{(-s\pm i\pi)\nu_k}|
\\ \lesssim & \sum_{k\in\mathbb{N}} (1+k)^{\frac{n-2}{n-1}}
e^{-sc_0(1+k)^{\frac{1}{n-1}}}
\\ \lesssim & \sum_{k\in\mathbb{N}}
e^{-\frac{sc_0}{2}(1+k)^{\frac{1}{n-1}}},
\end{split}
\end{align}
where we absorbed the polynomial factor into the exponential one by choosing a smaller coefficient in the exponent. 
Then we have 
\begin{align}
\begin{split}
|e^{-(s \pm i \pi)\sqrt{P}} | \lesssim & e^{-sc_0/2} \sum_{k\in\mathbb{N}}
e^{-\frac{sc_0}{2}\big((1+k)^{\frac{1}{n-1}}-1\big)}
\\ \lesssim & e^{-sc_0/2},
\end{split}
\end{align}
which completes the proof after denoting $c_0/2$ by $c$.
\end{proof}

In summary, we have shown that \eqref{S1}\eqref{S2}\eqref{S3} are uniformly (in terms of large $z$) bounded, concluding the proof.
\end{proof}

\section{The Littlewood-Paley theory }\label{sec:LP}

In this section,  we study the Bernstein inequalities and the square function inequalities associated with the Schr\"odinger operator $H$ for our purposes. 
As in the work of Killip, Miao, Visan, Zheng and the last author \cite{KMVZZ}, in which the Schr\"odinger operator on Euclidean space with an inverse-square potential was studied,
the Littlewood-Paley theory has its own independent interest. Here we provide a slightly different method based on the heat kernel estimates 
\begin{equation}\label{up-est:heat}
\begin{split}
\big|e^{-tH}(r_1,y_1;r_2,y_2)\big|\leq C \Big[\min\Big\{1,\Big(\frac{r_1r_2}{2t}\Big)\Big\}\Big]^{\alpha} t^{-\frac{n}2}e^{-\frac{d^2((r_1,y_1),(r_2,y_2))}{c t}},
  \end{split}
  \end{equation}
proved in \cite[Theorem 1.1]{HZ23}.
\vspace{0.2cm}

Now we study the Littlewood-Paley theory, including the Bernstein inequalities and the square function inequalities, associated with the Schr\"odinger operator $H$. More precisely, we prove the following propositions.

\begin{proposition}[Bernstein inequalities]\label{prop:Bern}
Let $\varphi(\lambda)$ be a $C^\infty_c$ bump function on $\R$  with support in $[\frac{1}{2},2]$ and let $\alpha$ and $q(\alpha)$ be given in \eqref{def:alpha} and \eqref{def:q-alpha} respectively, then the following holds for any $f\in L^q(X)$ and $j\in\mathbb{Z}$:
\begin{equation}\label{est:Bern}
\|\varphi(2^{-j}\sqrt{H})f\|_{L^p(X)}\lesssim2^{nj\big(\frac{1}{q}-\frac{1}{p}\big)}\|\varphi(2^{-j}\sqrt{H}) f\|_{L^q(X)}, \,q'(\alpha)<q\leq p<q(\alpha).
\end{equation}
In addition, if $\alpha\geq0$, the range can be extended to $1\leq q \leq p\leq +\infty$ including the endpoints.
\end{proposition}

\begin{proposition}[The square function inequality]\label{prop:squarefun} Let $\{\varphi_j\}_{j\in\mathbb Z}$ be a Littlewood-Paley sequence given by \eqref{LP-dp} and let $\alpha$ and $q(\alpha)$ be given in \eqref{def:alpha} and \eqref{def:q-alpha} respectively.
Then for $q'(\alpha)<p<q(\alpha)$,
there exist constants $c_p$ and $C_p$ depending on $p$ such that
\begin{equation}\label{square}
c_p\|f\|_{L^p(X)}\leq
\Big\|\Big(\sum_{j\in\Z}|\varphi_j(\sqrt{H})f|^2\Big)^{\frac12}\Big\|_{L^p(X)}\leq
C_p\|f\|_{L^p(X)}.
\end{equation}
\end{proposition}

\begin{proof}[The proof of Proposition \ref{prop:Bern}] If $\alpha\geq0$, from \eqref{up-est:heat}, the operator $H$ obeys the Gaussian heat kernel upper bounds and so the result
follows from general results covering this class of operators; see, for example \cite{Alex}.

In the spirit of \cite[Proposition 4.1]{WZZ}, we provide a simple argument which can be generalized to the case that the heat kernel of $H$
satisfies \eqref{up-est:heat}. Let  $\psi(x)=\varphi(\sqrt{x})$ and $\psi_e(x):=\psi(x)e^{2x}$. Then $\psi_e$  is  a $C^\infty_c$-function on $\R$ with support in $[\frac{1}{4},4]$ and
then its Fourier
transform $\hat{\psi}_e$ belongs to the Schwartz class. We write
\begin{align*}
 &\varphi(\sqrt{x})=\psi(x)= e^{-2x}\psi_{e}(x)=e^{-2x}\int_{\R} e^{i x \cdot\xi} \hat{\psi}_e(\xi)\, d\xi\\
 &=e^{-x}\int_{\R} e^{-x(1-i\xi)} \hat{\psi}_e(\xi)\, d\xi.
\end{align*}
Therefore, by the functional calculus, we obtain
\begin{align*}
 &\varphi(\sqrt{H})=\psi(H)=e^{-H}\int_{\R} e^{-(1-i\xi)H} \hat{\psi}_e(\xi)\, d\xi,
\end{align*}
furthermore,
\begin{align*}
 &\varphi(2^{-j}\sqrt{H})=\psi(2^{-2j}H)=e^{-2^{-2j}H}\int_{\R} e^{-(1-i\xi)2^{-2j}H} \hat{\psi}_e(\xi)\, d\xi.
\end{align*}
By using \eqref{up-est:heat} with $t=2^{-2j}$ and letting $z_{i,j}=(2^j r_i, y_i)$ with $i=0,1,2$ and writing $z_{i}=z_{i,0}$, we have
\begin{align} \label{eq:varphi-sqrtH-bound}
\begin{split}
& \Big|\varphi(2^{-j}\sqrt{H})(z_1,z_2)\Big|\\
&\lesssim 2^{2nj}\int_{X}\Big[\min\Big\{1,\Big(\frac{r_1r_0}{2 \times 2^{-2j}}\Big)\Big\}\min\Big\{1,\Big(\frac{r_0r_2}{2 \times 2^{-2j}}\Big)\Big\}\Big]^{\alpha} e^{-\frac{d^2(z_1,z_0)+d^2(z_0,z_2)}{c 2^{-2j}}}\, 
\\ & r_0^{n-1}dr_0 dy_0 \int_{\R} \hat{\psi}_e(\xi)\, d\xi\\
&\lesssim 2^{nj}\int_{X}\Big[\min\Big\{1,\Big(\frac{2^jr_1r_0}{2}\Big)\Big\}\min\Big\{1,\Big(\frac{r_0 2^jr_2}{2}\Big)\Big\}\Big]^{\alpha}  e^{-\frac{d^2(z_{1,j},z_0)+d^2(z_0,z_{2,j})}{c}}\, r_0^{n-1}dr_0 dy_0\\
& \lesssim 2^{nj}e^{-\frac{2^{2j}d^2(z_1,z_2)}{4c}}K(2^jr_1, y_1; 2^jr_2, y_2) \\
& \lesssim 2^{nj}(1+2^jd(z_1,z_2))^{-N}K(2^jr_1, y_1; 2^jr_2, y_2),\quad \forall N\geq 0
\end{split}
\end{align}
where we use the fact that 
$$d^2(z_{1,j}, z_0)+d^2(z_0, z_{2,j}) \geq 
\frac{1}{2} (d(z_{1,j}, z_0)+d(z_0, z_{2,j}))^2 \geq \frac12 d^2(z_{1,j},z_{2,j})$$ and the notation that
\begin{align*}
&K(2^jr_1, y_1; 2^jr_2, y_2)\\
=\int_{X}&\Big[\min\Big\{1,\Big(\frac{2^jr_1r_0}{2}\Big)\Big\}\min\Big\{1,\Big(\frac{ 2^jr_2r_0}{2}\Big)\Big\}\Big]^{\alpha} e^{-\frac{d^2(z_{1,j},z_0)+d^2(z_0,z_{2,j})}{4c}}\,  r_0^{n-1}dr_0 dy_0.
\end{align*}
To prove \eqref{est:Bern}, we only need to prove \eqref{est:Bern} with $j=0$ by the scaling argument. 
If $\alpha\geq 0$, then $\big|K(r_1, y_1; r_2, y_2)\big|\lesssim 1$. Therefore, in combination with \eqref{eq:varphi-sqrtH-bound}, we take $s \geq 1$ such that 
\begin{equation} \label{eq:s-condition}
    \frac{1}{p}=\frac{1}{q}+\frac{1}{s}-1,
\end{equation}
which exists since $q \leq p$, and apply Young's inequality (in the form \cite[Theorem~0.3.1]{sogge}) to obtain
\begin{align} \label{eq:Lp-Lq-Young}
\begin{split}
\|\varphi(\sqrt{H})f\|_{L^p(X)}&\lesssim \big\|\int_{X} (1+d(z_1,z_2))^{-N} |f(z_2)| dg(z_2)\big\|_{L^p(X)}\lesssim \|f\|_{L^q(X)},
\end{split}
\end{align}
since for $N$ such that $Ns>n$, $$\sup_{z_1} \|(1+d(z_1,\cdot))^{-N}\|_{L^s(X)} \lesssim 1,\quad \sup_{z_2} \|(1+d(\cdot,z_2))^{-N}\|_{L^s(X)} \lesssim 1.$$
Therefore, \eqref{eq:Lp-Lq-Young} implies \eqref{est:Bern} when $\alpha\geq 0$ by the scaling argument mentioned above. If $-(n-2)/2<\alpha<0$, then
\begin{align*}
&\big|K(r_1, y_1; r_2, y_2)\big|\\
&\lesssim \int_{X}\Big[\min\Big\{1, r_1r_0, r_2r_0, r_1r_2r^2_0\Big\}\Big]^{\alpha} e^{-\frac{d^2(z_{1},z_0)+d^2(z_0,z_{2})}{4c}}\,  r_0^{n-1}dr_0 dy_0\\
&\lesssim \max\Big\{1, r^{\alpha}_1, r^{\alpha}_2, (r_1r_2)^{\alpha}\Big\}.
\end{align*}
Therefore, we obtain
\begin{align*}
\|\varphi(\sqrt{H})f\|_{L^p(X)}&\lesssim \big\|\int_{X} (1+d(z_1,z_2))^{-N} \max\Big\{1, r^{\alpha}_1, r^{\alpha}_2, (r_1r_2)^{\alpha}\Big\} |f(z_2)| dg(z_2)\big\|_{L^p(X)}.
\end{align*}
Let $\chi\in \mathcal{C}_c^\infty ([0,+\infty))$ be defined as 
\begin{equation}\label{def:chi}
\chi(r)=
\begin{cases}1,\quad r\in [0, \frac12],\\
0, \quad r\in [1,+\infty)
\end{cases}
\end{equation}
and let us set $\chi^c=1-\chi$. Hence, when $q'(\alpha)<q\leq p<q(\alpha)$, we have
\begin{align} \label{eq:varphi-sqrtH-f}
\begin{split}    
&\|\varphi(\sqrt{H})f\|_{L^p(X)}\\
&\lesssim 
\| T_{11} f\|_{L^p(X)}+\| T_{01} f\|_{L^p(X)}
+ \| T_{10} f\|_{L^p(X)}+\| T_{00} f\|_{L^p(X)},
\end{split}
\end{align}
where
\begin{equation*}
T_{11}f(z_1)
=
\int_X (1+d(z_1,z_2))^{-N}
\chi^c(r_1)\chi^c(r_2)f(z_2)\, r_2^{n-1}dr_2dy_2,
\end{equation*}
\begin{equation*}
T_{01}f(z_1)
=
\int_X (1+d(z_1,z_2))^{-N}
r_1^\alpha\chi(r_1)\chi^c(r_2)f(z_2)\, r_2^{n-1}dr_2dy_2,
\end{equation*}
\begin{equation*}
T_{10}f(z_1)
=
\int_X (1+d(z_1,z_2))^{-N}
r_2^\alpha\chi^c(r_1)\chi(r_2)f(z_2)\,r_2^{n-1}dr_2dy_2,
\end{equation*}
and
\begin{equation*}
T_{00}f(z_1)
=
\int_X (1+d(z_1,z_2))^{-N}
(r_1r_2)^\alpha\chi(r_1)\chi(r_2)f(z_2)\, r_2^{n-1}dr_2dy_2.
\end{equation*}
Next we prove that each of these operators is bounded from $L^q(X)$ to $L^p(X)$ and the right hand side of \eqref{eq:varphi-sqrtH-f} is bounded by $\|f\|_{L^q(X)}$ up to a constant.
Since $\alpha<0$ and $p<-\frac{n}{\alpha}$, we know
\begin{equation*}
r^\alpha\chi(r)\in L^p(X),
\end{equation*}
since near $r=0$ we have $(r^\alpha)^p r^{n-1}dr=r^{n+\alpha p -1} dr$ with $n+\alpha p -1>-1$ and $\chi$ is supported in the region $r \lesssim 1$.
Similarly, let $q'$ be such that $\frac{1}{q}+\frac{1}{q'}=1$, then the assumption $q>q'(\alpha)=\frac{n}{n+\alpha}$ in \eqref{est:Bern} 
is equivalent to $q'<-\frac{n}{\alpha}$ and gives 
\begin{equation} \label{eq:rchi-in-Lq'}
r^\alpha\chi(r)\in L^{q'}(X)
\end{equation}
for the same reason as above.
We now estimate $\|T_{ij}f\|_{L^p(X)}$ for $i,j=0,1$.

First, for $T_{11}$, choosing $s\ge 1$ satisfying \eqref{eq:s-condition} as before, the same argument as in \eqref{eq:Lp-Lq-Young} gives

\begin{equation*}
\|T_{11}f\|_{L^p(X)}
\le C\|f\|_{L^q(X)}.
\end{equation*}

Next, for $T_{01}$, we write
\begin{equation*}
T_{01}f(z_1)
= r_1^\alpha\chi(r_1)
\int_X (1+d(z_1,z_2))^{-N}
\chi^c(r_2)f(z_2)\,r_2^{n-1}dr_2dy_2.
\end{equation*}
By H\"older's inequality,
\begin{equation*}
\left|
\int_X (1+d(z_1,z_2))^{-N}
\chi^c(r_2)f(z_2)\,r_2^{n-1}dr_2dy_2
\right|
\le
\left\|(1+d(z_1,\cdot))^{-N}\right\|_{L^{q'}(X)}
\|f\|_{L^q(X)}.
\end{equation*}
If $Nq'>n$, $\left\|(1+d(z_1,\cdot))^{-N}\right\|_{L^{q'}(X)} \lesssim 1$ uniformly in $z_1$. Hence
\begin{equation*}
|T_{01}f(z_1)|
\le
C r_1^\alpha\chi(r_1)\|f\|_{L^q(X)}.
\end{equation*}
Taking the $L^p$-norm in $z_1$, we obtain
\begin{equation*}
\|T_{01}f\|_{L^p(X)}
\le
C\|r^\alpha\chi\|_{L^p(X)}
\|f\|_{L^q(X)}
\le
C\|f\|_{L^q(X)}.
\end{equation*}
For $T_{10}$, we set
\begin{equation*}
\tilde{f}(z_2)=r_2^\alpha\chi(r_2)f(z_2).
\end{equation*}
By H\"older's inequality and the fact that $r_2^\alpha\chi(r_2) \in L^{q'}(X)$ discussed in \eqref{eq:rchi-in-Lq'},
\begin{equation*}
\|\tilde{f}\|_{L^1(X)}
\le
\|r_2^\alpha\chi(r_2)\|_{L^{q'}(X)}
\|f\|_{L^q(X)}
\le
C\|f\|_{L^q(X)}.
\end{equation*}
Then
\begin{equation*}
T_{10}f(z_1)
=
\chi^c(r_1)
\int_X (1+d(z_1,z_2))^{-N}\tilde{f}(z_2)\,r_2^{n-1}dr_2dy_2.
\end{equation*}
Using Young's inequality for the integral kernel $(1+d(z_1,z_2))^{-N}$ as a map from $L^1(X) \to L^p(X)$, and taking $Np>n$, we know the operator norm is bounded by $\sup_{z_1} \|(1+d(z_1,\cdot))^{-N}\|_{L^p(X)},\quad \sup_{z_2} \|(1+d(\cdot,z_2))^{-N}\|_{L^p(X)}$, both of which are $\lesssim 1$.
So we obtain
\begin{equation*}
\|T_{10}f\|_{L^p(X)}
\le
C\|\tilde{f}\|_{L^1(X)}
\le
C\|f\|_{L^q(X)}.
\end{equation*}

Finally, for $T_{00}$, 
the factor $(1+d(z_1,z_2))^{-N}$ is bounded by $1$, and we have
\begin{equation*}
|T_{00}f(z_1)|
\le
r_1^\alpha\chi(r_1)
\int_X r_2^\alpha\chi(r_2)|f(z_2)|\,r_2^{n-1}dr_2dy_2.
\end{equation*}
By H\"older's inequality,
\begin{equation*}
\int_X r_2^\alpha\chi(r_2)|f(z_2)|\,r_2^{n-1}dr_2dy_2
\le
\|r^\alpha\chi\|_{L^{q'}(X)}
\|f\|_{L^q(X)}
\le
C\|f\|_{L^q(X)}.
\end{equation*}
Hence
\begin{equation*}
|T_{00}f(z_1)|
\le
C r_1^\alpha\chi(r_1)\|f\|_{L^q(X)}.
\end{equation*}
Taking the $L^p$-norm gives
\begin{equation*}
\|T_{00}f\|_{L^p(X)}
\le
C\|r_1^\alpha\chi(r_1)\|_{L^p(X)}
\|f\|_{L^q(X)}
\le
C\|f\|_{L^q(X)}.
\end{equation*}
Combining the estimates for the four parts, we obtain
\begin{equation*}
\|T_{11}f\|_{L^p(X)}
+
\|T_{01}f\|_{L^p(X)}
+
\|T_{10}f\|_{L^p(X)}
+
\|T_{00}f\|_{L^p(X)}
\le
C\|f\|_{L^q(X)}.
\end{equation*}
Thus we have the desired bound that gives \eqref{est:Bern} when $-(n-2)/2<\alpha< 0$.
\end{proof}

\begin{proof}[The proof of Proposition \ref{prop:squarefun}]
In order to prove the square function estimates \eqref{square}, by using the Rademacher functions and 
the argument of Stein \cite[Appendix D]{Stein}, it suffices to show 
that the Littlewood-Paley operator satisfies 
$$\|\varphi(\sqrt{H})f\|_{L^p(X)}\lesssim \|f\|_{L^p(X)}, \quad q'(\alpha)<p<q(\alpha), $$
which can be done by repeating the above argument of Proposition \ref{prop:Bern}.
 We also refer the reader to \cite{Alex} for the result that the square function inequality \eqref{square} can be derived from
the heat kernel with Gaussian upper bounds.\end{proof}

\section{The decay estimates for the Schr\"odinger propagator}\label{sec:decayS}
In this section, we prove the decay estimates in Corollary \ref{cor:L1Linfty} and  Theorem \ref{thm:LqLq'} by using the main Theorem \ref{thm1:dispersiveS}. In the concrete proof we consider $t \geq 0$ and the case $t<0$ follows by symmetry.

\begin{proof}[The proof of Corollary \ref{cor:L1Linfty}] Since $\alpha\geq 0$, \eqref{est:dis-cl} and \eqref{est:dis-weight} follow from \eqref{est:dispersive} directly.
If $-\frac{n-2}2<\alpha< 0$, we obtain \eqref{est:dis-weight'}  from  \eqref{est:dispersive} and the fact that 
$$(1+r_1^\alpha)^{-1}(1+r_2^\alpha)^{-1}(1+|t|^{-\alpha})^{-1}\leq \min\Big\{1, \Big(\frac {|t|}{r_1r_2}\Big)^{\alpha}\Big\}.$$
\end{proof}

\begin{proof}[The proof of Theorem \ref{thm:LqLq'}] By the spectral theorem, one has the $L^2$-estimate
\begin{equation}\label{est:L2}
\|e^{itH}\|_{L^2(X)\to L^2(X)}\leq C.
\end{equation}
To prove this, we need a property of the Hankel transform. 
For $f\in L^2(X)$, as in \cite[p. 523]{BPSS}, we define the Hankel transform of order $\mu$
\begin{equation}\label{hankel}
(\mathcal{H}_{\mu}f)(\rho)=\int_0^\infty (r\rho)^{-\frac{n-2}2}J_{\mu}(r\rho)f(r) \,r^{n-1}dr.
\end{equation}
Then we have the unitary property  $\|\mathcal{H}_{\mu} f\|_{L^2_{\rho^{n-1}d\rho}(\R^+)}=\|f(r)\|_{L^2_{r^{n-1}dr}(\R^+)}$.
By the functional calculus as in \eqref{ker:S},
 we also obtain the kernel $K(t,z_1,z_2)$ of the operator  $e^{itH}$
\begin{equation*}
\begin{split}
K(t,z_1,z_2)&=K(t,r_1,y_1,r_2, y_2)\\
&=\big(r_1 r_2\big)^{-\frac{n-2}2}\sum_{k\in\mathbb{N}}\varphi_{k}(y_1)\overline{\varphi_{k}(y_2)}K_{\nu_k}(t,r_1,r_2),
\end{split}
\end{equation*}
where $\overline{\varphi}_k$ denotes the complex conjugate of the eigenfunction $\varphi_k$ and
\begin{equation*}
\begin{split}
  K_{\nu_k}(t,r_1,r_2)&=\int_0^\infty e^{-it\rho^2}J_{\nu_k}(r_1\rho)J_{\nu_k}(r_2\rho) \,\rho d\rho.
  \end{split}
\end{equation*}
For $f\in L^2$, we expand 
\begin{equation}\label{f:exp}
\begin{split}
f=\sum_{k\in\mathbb{N}}c_{k}(r)\varphi_k(y),
\end{split}
\end{equation}
 then, by orthogonality and the unitarity of the Hankel transform, we obtain
\begin{equation*}
\begin{split}
\|e^{itH} f\|_{L^2(X)}&=\Big(\sum_{k\in\mathbb{N}} \big\| \mathcal{H}_{\nu_k}\big(e^{-it\rho^2} ( \mathcal{H}_{\nu_k} c_{k})\big)(r)\big\|_{L^2_{r^{n-1}dr}}^2\Big)^{1/2}\\&
=\Big(\sum_{k\in \mathbb{N}} \big\| c_{k}(r)\big\|^2_{L^2_{r^{n-1}dr}}\Big)^{1/2}
=\|f\|_{L^2(X)}.
\end{split}
\end{equation*}
So, if $\alpha\geq0$, we obtain \eqref{est:LqLq'} by interpolating \eqref{est:L2} and \eqref{est:dis-cl}. If $\alpha<0$, one can obtain \eqref{est:LqLq'} but with some weight by interpolating \eqref{est:L2} and \eqref{est:dis-weight'}. To prove \eqref{est:LqLq'0}, we need to strengthen it to get rid of the
weight when $q\in [2, q(\alpha))$. Intuitively, as in Proposition \ref{prop:Bern}, one might try to prove \eqref{est:LqLq'0} by replacing the heat kernel estimates \eqref{up-est:heat} by the estimates \eqref{est:dispersive}. Unfortunately, it doesn't work due to the lack of exponential decay in \eqref{est:dispersive}, so we have to decompose the Schr\"odinger propagator. \vspace{0.2cm}

To this end, we introduce the orthogonal projections on $L^{2}$
\begin{equation}\label{def-pro1}
  P_k:
  L^{2}(X)\to     L^{2}(r^{n-1}dr)\otimes  h_{k}(Y),
\end{equation}
and
\begin{equation}\label{def-pro}
  P_<:
  L^{2}(X)\to   
  \bigoplus_{\{k\in \mathbb{N}: \nu_k< (n-2)/2\}}   L^{2}(r^{n-1}dr)\otimes  h_{k}(Y),
  \quad
  P_{\geq }=I-P_{<}.
\end{equation}
Here the space $h_{k}(Y)$ is the linear span of $\{\varphi_k(y)\}$
defined in \eqref{eig-eq}. Then we can decompose the Schr\"odinger propagator as
\begin{equation}\label{d-propag}
e^{itH}f=
 e^{itH}P_{<}f+
 e^{itH}P_{\geq }f.
\end{equation}
By  \eqref{ker:S}, we see that the kernels 
\begin{equation}\label{ker:S-l}
\begin{split}
 e^{itH}P_{<}
&=\big(r_1 r_2\big)^{-\frac{n-2}2}\sum_{\{k\in\mathbb{N}: \nu_k<(n-2)/2\}}\varphi_{k}(y_1)\overline{\varphi_{k}(y_2)}K_{\nu_k}(t,r_1,r_2),
\end{split}
\end{equation}
and 
\begin{equation}\label{ker:S-h}
\begin{split}
 e^{itH}P_{\geq }
&=\big(r_1 r_2\big)^{-\frac{n-2}2}\sum_{\{k\in\mathbb{N}: \nu_k\geq \frac12(n-2)\}}\varphi_{k}(y_1)\overline{\varphi_{k}(y_2)}K_{\nu_k}(t,r_1,r_2).
\end{split}
\end{equation}
Since the kernel $ e^{itH}P_{\geq }$ is microlocalized to large angular momenta, we can repeat the argument of Proposition \ref{prop<} and Proposition \ref{prop>} to obtain 
$$\big| e^{itH}P_{\geq }\big|\leq C |t|^{-\frac n2}.$$
Therefore, as in the case $\alpha\geq 0$, we can prove \eqref{est:LqLq'0}  for $ e^{itH}P_{\geq }$ with $q\geq2$. Thus it remains to consider $ e^{itH}P_{<}$, in which we are restricted to small angular momenta. Due to Weyl's asymptotic formula in \eqref{est:eig} about $\nu_k$,
the summation in the kernel $ e^{itH}P_{<}$ in \eqref{ker:S-l} is finite. Hence, to prove \eqref{est:LqLq'0} for  $ e^{itH}P_{<}$,
we only need to prove \eqref{est:LqLq'0} for  $ e^{itH}P_{k}$ with each $k$ such that $\nu_k<(n-2)/2$.
By using the Littlewood-Paley square function inequality \eqref{square}
and the Minkowski inequality, it suffices to show
  \begin{equation}\label{est:q-q'-m}
  \Big\|\varphi_j(\sqrt{H}) e^{itH}P_{k} f\Big\|_{L^q(X)}\le
    C_k |t|^{-\frac n2(1-\frac 2q)} \Big\|\tilde{\varphi}_j(\sqrt{H}) P_{k} f\Big\|_{L^{q'}(X)},
  \end{equation}
  provided $q\in [2, q(\alpha))$ where we choose $\tilde{\varphi}\in C_c^\infty((0,+\infty))$ such that $\tilde{\varphi}(\lambda)=1$ if $\lambda\in\mathrm{supp}\,\varphi$
and $\tilde{\varphi}\varphi=\varphi$. In the following argument, since $\tilde{\varphi}$ has the same property as $\varphi$, we drop the tilde above $\varphi$ for brevity.   

To prove \eqref{est:q-q'-m}, we need the following proposition.
\begin{proposition}\label{prop:est-qq'}
  Let $0<\nu\leq \frac{n-2}2$ and $\sigma(\nu)=-(n-2)/2+\nu$. Let $T_\nu$ be the operator defined as
  \begin{equation}\label{Tnu-operator}
\begin{split}
(T_{\nu}g)(t,r_1)=\int_0^\infty  K^l_{\nu}(t;r_{1},r_{2}) g(r_2)\, r^{n-1}_2 dr_2
 \end{split}
\end{equation}
 and 
\begin{equation*}
\begin{split}
  K^l_{\nu}(t,r_1,r_2)&=(r_1r_2)^{-\frac{n-2}2}\int_0^\infty e^{it\rho^2}J_{\nu}(r_1\rho)J_{\nu}(r_2\rho) \varphi(\rho)\,\rho d\rho,
  \end{split}
\end{equation*}
where $\varphi$ is given in \eqref{LP-dp}.
Then, for $2\leq q<q(\sigma)$, the following estimate holds
  \begin{equation}\label{est:q-q'}
  \|T_{\nu}g\|_{L^q({r^{n-1}_1 dr_1})}\le
    C_{\nu}|t|^{-\frac n2(1-\frac 2q)}\|g\|_{L^{q'}_{r^{n-1}_2 dr_2}}.
  \end{equation}
\end{proposition}

We postpone the proof of Proposition \ref{prop:est-qq'} for a moment. Recalling \eqref{f:exp} and letting $\tilde{c}_k(r)={\varphi}_j(\sqrt{H}) c_k(r)$, similarly to \eqref{ker:S}, we write
  \begin{equation*}
    \begin{split}
\varphi_j(\sqrt{H}) e^{itH}P_{k} f&=\varphi_k(y) 2^{jn} \int_0^\infty  K^l_{\nu_k}(2^{2j}t;2^jr_{1}, 2^jr_{2}) \tilde{c}_k(r_2)\, r^{n-1}_2 dr_2\\
&=\varphi_k(y) \big(T_{\nu_k}\tilde{c}_k(2^{-j}r_2)\big)(2^{2j}t, 2^jr_1).
\end{split}
\end{equation*}
Noticing that $q(\alpha)\leq q(\sigma)$, we use \eqref{est:q-q'} and the eigenfunction estimates to obtain that 
  \begin{equation*}
  \begin{split}
  &\Big\|\varphi_j(\sqrt{H}) e^{itH}P_{k} f\Big\|_{L^q(X)}\le
    C_k \|\big(T_{\nu_k}\tilde{c}_k(2^{-j}\cdot)\big)(2^{2j}t, 2^jr_1)\|_{L^q_{r^{n-1}_1 dr_1}} \|\varphi_k(y)\|_{L^q(Y)}
    \\&\le C_{k}|t|^{-\frac n2(1-\frac 2q)}\|\tilde{c}_k(r)\|_{L^{q'}_{r^{n-1} dr}}\|\varphi_k(y)\|_{L^{q'}(Y)}\le C_k |t|^{-\frac n2(1-\frac 2q)} \Big\|{\varphi}_j(\sqrt{H}) P_{k} f\Big\|_{L^{q'}(X)},
    \end{split}
  \end{equation*}
where we used $\|\varphi_k(y)\|_{L^q(Y)} \leq C \|\varphi_k(y)\|_{L^{q'}(Y)}$ since $Y$ is compact and we are concerned with only finitely many $\varphi_k$ such that the corresponding $\nu_k \in (0,\frac{n-2}{2}]$.
This completes the proof of the desired estimate \eqref{est:q-q'-m}.

\end{proof}
  
Before proving Proposition \ref{prop:est-qq'}, we record a lemma about the properties of the Bessel function, e.g. see \cite[Lemma 5.1]{CDYZ}.
\begin{lemma}\label{lem:bessel}
  For all $r,\nu\in \mathbb{R}^+$, there exist constants $C_{\nu}$ and $C_{\nu,N}$ depending only on $\nu$ and on $\nu,N$ respectively such that
    \begin{equation}\label{eq:bess1}
    |J_{\nu}(r)|\le C_{\nu} r^{\nu}(1+r)^{-\nu-\frac 12},
  \end{equation}
  \begin{equation}\label{eq:bess2}
    |J'_{\nu}(r)|=
    |J_{\nu-1}(r)-\nu J_{\nu}(r)/r|
    \le C_{\nu}r^{\nu-1}(1+r)^{-\nu+\frac 12}.
  \end{equation}
  Moreover we can write
  \begin{equation}\label{eq:bess3}
    J_{\nu}(r)=r^{-1/2}(e^{ir}\mathfrak{a}_{+}(r)+e^{-ir}\mathfrak{a}_{-}(r))
  \end{equation}
  for two functions $\mathfrak{a}_{\pm}$ depending on $\nu,r$
  and satisfying for all $N\ge1$ and $r\ge1$
  \begin{equation}\label{eq:bess4}
    |\mathfrak{a}_{\pm}(r)|\le C_{\nu,0},
    \qquad
    |\partial_{r}^{N}\mathfrak{a}_{\pm}(r)|\le C_{\nu,N}
    r^{-N-1}.
  \end{equation}
\end{lemma}

\begin{proof}[The proof of  Proposition \ref{prop:est-qq'}]
Our proof is modified from \cite{CDYZ}, in which the dispersive estimates for the Dirac equation in Aharonov-Bohm magnetic fields were studied.
But we have to overcome the difficulties from the propagator multiplier $e^{it\rho^2}$.
Recalling $\chi\in \mathcal{C}_c^\infty ([0,+\infty))$ defined by \eqref{def:chi} and $\chi^c=1-\chi$, we decompose the kernel $ K^l_{\nu}(t;r_{1},r_{2})$ into four terms as follows:
  \begin{equation}
\begin{split}
K^l_{\nu}(t;r_{1},r_{2})=&\chi(r_1)K^l_{\nu}(t;r_{1},r_{2})\chi(r_2)+\chi^c(r_1)K^l_{\nu}(t;r_{1},r_{2})\chi(r_2)\\
&+\chi(r_1)K^l_{\nu}(t;r_{1},r_{2})\chi^c(r_2)+\chi^c(r_1)K^l_{\nu}(t;r_{1},r_{2})\chi^c(r_2).
 \end{split}
\end{equation}
This yields a corresponding decomposition for the operator $T_{\nu}=T^1_{\nu}+T^2_{\nu}+T^3_{\nu}+T^4_{\nu}$. We thus estimate separately the norms $\|T^j_{\nu}g\|_{L^q_{r_1^{n-1} dr_1}}$ for $j=1,2,3,4$.

Now we estimate $T^1_\nu$. From \eqref{eq:bess1}, one has 
  \begin{equation}
\begin{split}
|\chi(r_1)K^l_{\nu}(t;r_{1},r_{2})\chi(r_2)|\lesssim  (r_1r_2)^{\sigma} \chi(r_1)\chi(r_2).
 \end{split}
\end{equation}
Therefore, as long as $2\leq q< q(\sigma)$, if $|t|\leq 1$, we can show
  \begin{equation}\label{est:q-q'1-1}
  \begin{split}
  \|T^1_{\nu}g\|_{L^q_{r^{n-1}_1 dr_1}}&\le
    C_{\nu} \Big(\int_0^1 r^{\sigma q} r^{n-1} dr\Big)^{2/q}\|g\|_{L^{q'}_{r^{n-1}_2 dr_2}}\\
    &\le
   C_{\nu}  \|g\|_{L^{q'}_{r^{n-1}_2 dr_2}} \lesssim |t|^{-\frac n2(1-\frac2q)}  \|g\|_{L^{q'}_{r^{n-1}_2 dr_2}}.
   \end{split}
  \end{equation}

For the case that $|t|\geq 1$, we first prove that for all $N \in \mathbb{N}$, we have
\begin{align} \label{eq:Jnu-product}
\Big|\Big(\rho^{-1}\frac{\partial}{\partial\rho}\Big)^{N}\Big(J_{\nu}(r_1\rho)J_{\nu}(r_2\rho) \varphi(\rho)\rho \Big)\Big| \lesssim (r_1r_2)^{\nu}, \; r_1, r_2 \leq 1.
\end{align}
Since $\varphi(\rho)$ and all of its derivatives are supported in the region $\rho \sim 1$, the left hand side of \eqref{eq:Jnu-product} is a sum of terms of the form (modulo a uniformly bounded smooth factor)
$\Big(\frac{\partial}{\partial\rho}\Big)^{k}\Big(J_{\nu}(r_1\rho)J_{\nu}(r_2\rho) \Big)$,
and we only need to show that for any fixed $k \in \mathbb{N}$ we have
\begin{align}
\Big|\Big(\frac{\partial}{\partial\rho}\Big)^{k}\Big(J_{\nu}(r_1\rho)J_{\nu}(r_2\rho) \Big)\Big| \lesssim (r_1r_2)^{\nu}, \; r_1, r_2 \leq 1, \, \rho \sim 1.
\end{align}
This in turn follows if we can show that for any fixed $k \in \mathbb{N}$ we have
\begin{align} \label{eq:dk-Jnu-rirho}
\Big|\Big(\frac{\partial}{\partial\rho}\Big)^{k}\Big(J_{\nu}(r_i\rho) \Big)\Big| \lesssim (r_i)^{\nu}, \; r_i \leq 1, \, \rho \sim 1.
\end{align}
To this end, we first recall the recurrence relation of the Bessel function (see \cite[Section~3.2, Eq. (4)]{Watson}):
\begin{align} \label{eq:dBessel-recurrence}
J'_{\nu}(r) = \nu r^{-1}J_\nu(r) - J_{\nu+1}(r).
\end{align}
Using this recurrence relation repeatedly, an induction on the number of differentiations, combined with \eqref{eq:bess1}, shows that
\begin{equation} \label{eq:drk-Jnu}
|(\frac{d}{dr})^kJ_\nu(r)| \lesssim r^{\nu-k} , \quad  r \lesssim 1.
\end{equation}
Since $\rho \sim 1$, and $\Big(\frac{\partial}{\partial\rho}\Big)^{k}\Big(J_{\nu}(r_i\rho) \Big) = r_i^k \Big((\frac{d}{dr})^kJ_\nu(r)\Big)\Big|_{r=r_i\rho}$, we obtain \eqref{eq:dk-Jnu-rirho} after substituting $r=r_i\rho$ in \eqref{eq:drk-Jnu}.

Then we perform integration by parts in $\rho$ to obtain 
  \begin{equation}
\begin{split}
&|\chi(r_1)K^l_{\nu}(t;r_{1},r_{2})\chi(r_2)|\\
&\lesssim   \big(r_1r_2\big)^{-\frac{n-2}2} \chi(r_1)\chi(r_2)|t|^{-N}\int_0^\infty \Big| 
\Big(\rho^{-1}\frac{\partial}{\partial\rho}\Big)^{N}\Big(J_{\nu}(r_1\rho)J_{\nu}(r_2\rho) \varphi(\rho)\rho \Big)\Big| d\rho \\
&\lesssim   \big(r_1r_2\big)^{\nu-\frac{n-2}2} \chi(r_1)\chi(r_2)|t|^{-N}.
 \end{split}
\end{equation}
Finally, if $|t|\geq1$ and $N$ is large enough, as before, we obtain
  \begin{equation}\label{est:q-q'1-2}
  \begin{split}
  \|T^1_{\nu}g\|_{L^q_{r^{n-1}_1 dr_1}}&\le
    C_{\nu} |t|^{-N} \Big(\int_0^1 r^{\sigma q} r^{n-1} dr\Big)^{2/q}\|g\|_{L^{q'}_{r^{n-1}_2 dr_2}}\\
    & \lesssim |t|^{-\frac n2(1-\frac2q)}  \|g\|_{L^{q'}_{r^{n-1}_2 dr_2}}.
   \end{split}
\end{equation}
Since $T^2_{\nu}$ and $T^3_{\nu}$ are similar, we only deal with $T^3_{\nu}$. 
Using \eqref{eq:bess3}, we are reduced to estimate two integrals
  \begin{equation}\label{I+-}
    I_{\pm}=(r_1r_2)^{-\frac{n-2}2}\int_{0}^{\infty}
    \rho \varphi(\rho)J_{\nu}(r_{1}\rho)(r_{2}\rho)^{-1/2}
   e^{it\rho^2} e^{\pm i r_{2}\rho}\mathfrak{a}_{\pm}(r_{2}\rho)d \rho.
  \end{equation}
 If $|t|\leq1$, by using integration by parts and recalling $\sigma=\nu-(n-2)/2$, we obtain 
   \begin{equation*}
   \begin{split}
    |I_{\pm}| &\lesssim (r_1r_2)^{-\frac{n-2}2} r_2^{-\frac12-N}
    \int_{0}^{\infty}
   \Big| \Big(\frac{\partial}{\partial\rho}\Big)^{N}\Big(J_{\nu}(r_1\rho)\mathfrak{a}_{\pm}(r_{2}\rho)\varphi(\rho)\rho^{1/2} e^{it\rho^2}\Big)\Big|d\rho\\
  & \lesssim r_1^{\sigma} r_2^{-\frac{n-1}2-N}.
  \end{split}
  \end{equation*}
   Hence if $|t|\leq 1$ and  $2\leq q< q(\sigma)$, by choosing $N$ large enough, we have
      \begin{equation}
         \begin{split}
  \|T^3_{\nu}g\|_{L^q_{r^{n-1}_1 dr_1}}&\lesssim \Big(\int_0^1 r_1^{\sigma q} r^{n-1}_1 dr_1\Big)^{1/q}\Big(\int_{\frac12}^{+\infty} r_2^{-(\frac{n-1}2+N)q} r^{n-1}_2 dr_2\Big)^{1/q}\|g\|_{L^{q'}_{r^{n-1}_2 dr_2}}\\&\lesssim \|g\|_{L^{q'}_{r^{n-1}_2 dr_2}}   \lesssim |t|^{-\frac n2(1-\frac2q)}  \|g\|_{L^{q'}_{r^{n-1}_2 dr_2}}.
      \end{split}
  \end{equation}
 It remains to consider the region $|t|\geq 1$. 
  In this case, letting $\bar{r}_i=r_i/\sqrt{t}$ with $i=1,2$, from \eqref{I+-}, we write (we rescale $\rho$ by $\sqrt{t}$ as well and still denote the new variable by $\rho$)
   \begin{equation}\label{I-pm}
   \begin{split}
    I_{\pm}&=|t|^{-\frac n2}(\bar{r}_1\bar{r}_2)^{-\frac{n-2}2}\int_{0}^{\infty}
    \rho \varphi(\rho/\sqrt{t})J_{\nu}(\bar{r}_{1}\rho)(\bar{r}_{2}\rho)^{-1/2}
   e^{i\rho(\rho\pm \bar{r}_2)} \mathfrak{a}_{\pm}(\bar{r}_{2}\rho)d \rho\\
   &=|t|^{-\frac {n}2}\int_{0}^{\infty}
   e^{i\rho(\rho\pm \bar{r}_2)} \tilde{a}_{\pm}(t^{-\frac12}\rho, \bar{r}_{1}\rho, \bar{r}_{2}\rho) \rho^{n-1}d \rho,
   \end{split}
  \end{equation}
  where 
  \begin{equation} \label{def:ta}
  \tilde{a}_{\pm}(t^{-\frac12}\rho, \bar{r}_{1}\rho, \bar{r}_{2}\rho)= \varphi(t^{-\frac12}\rho)(\bar{r}_{1}\rho)^{-\frac{n-2}2} J_{\nu}(\bar{r}_{1}\rho)(\bar{r}_{2}\rho)^{-\frac{n-1}2} \mathfrak{a}_{\pm}(\bar{r}_{2}\rho).
  \end{equation}
Since $\bar{r}_1\rho\lesssim 1$ and $\sigma=\nu-\frac{n-2}2$, we obtain
\begin{equation} \label{est:symb}
\Big| \Big(\frac{\partial}{\partial\rho}\Big)^N \Big(\tilde{a}_{\pm}(t^{-\frac12}\rho, \bar{r}_{1}\rho, \bar{r}_{2}\rho)\Big)\Big| \lesssim  (\bar{r}_{1}\rho)^{\sigma} (\bar{r}_{2}\rho)^{-\frac{n-1}2} \rho^{-N}
\lesssim r_{1}^{\sigma} r_{2}^{-\frac{n-1}2} \rho^{-N},
\end{equation}
since $\rho\sim \sqrt{t}$ on the support of $\varphi(t^{-\frac12}\rho)$.

 \begin{lemma}\label{lem: est-interation} Let $$\tilde{I}_{\pm}(t, \bar{r}_{1}, \bar{r}_{2} )=\int_{0}^{\infty}
   e^{i\rho(\rho\pm \bar{r}_2)} \tilde{a}_{\pm}(t^{-\frac12}\rho, \bar{r}_{1}\rho, \bar{r}_{2}\rho) \rho^{n-1} d \rho,
$$
where $  \tilde{a}_{\pm}$ is given by \eqref{def:ta} and satisfies \eqref{est:symb}. Then, for $t\geq1$, the integral satisfies
   \begin{equation}\label{est:I-pm}
  \big| \tilde{I}_{\pm}(t, \bar{r}_{1}, \bar{r}_{2} )\big|\lesssim  r_{1}^{\sigma} r_{2}^{-\frac{n-1}2}+r_{1}^{\sigma} \chi_A,
  \end{equation}
where $\chi_A$ is the characteristic function on the set $A:=\{r_2\sim t\}$.
 \end{lemma}
If we could prove this lemma, then we see for $|t|\geq1$
      \begin{equation}
         \begin{split}
  &\|T^3_{\nu}g\|_{L^q_{r^{n-1}_1 dr_1}}\\
  &\lesssim |t|^{-\frac n2}\Big(\int_0^1 r_1^{\sigma q} r^{n-1}_1 dr_1\Big)^{1/q}\Big(\int_{\frac{1}{2}}^{+\infty} r_2^{-\frac{(n-1)q}{2}} r^{n-1}_2 dr_2+\int_{r_2\sim t} r^{n-1}_2 dr_2\Big)^{1/q}\|g\|_{L^{q'}_{r^{n-1}_2 dr_2}}
  \\&  \lesssim |t|^{-\frac n2} \big(1+|t|^{\frac nq}\big) \|g\|_{L^{q'}_{r^{n-1}_2 dr_2}} \lesssim |t|^{-\frac n2(1-\frac2q)}  \|g\|_{L^{q'}_{r^{n-1}_2 dr_2}}
      \end{split}
  \end{equation}
  provided $$\frac{2n}{n-1}<q<q(\sigma)=\frac{2n}{n-2-2\nu}.$$ We can extend this estimate for $2\leq q< q(\sigma)$ by interpolating this estimate with
    \begin{equation} \label{eq:T3-L2-est}
  \|T^3_{\nu}g\|_{L^2_{r^{n-1}_1 dr_1}}\le   \|\chi(r_1)\mathcal{H}_\nu e^{it\rho^2}\varphi(\rho) \mathcal{H}_\nu \chi^c(r_2)g\|_{L^2_{r^{n-1}_1 dr_1}}\leq C \|g\|_{L^{2}_{r^{n-1}_2 dr_2}}, 
  \end{equation}
which can be proved by the fact that the Hankel transform \eqref{hankel} is unitary on $L^2_{r^{n-1}dr}$.
Next we give the proof of Lemma \ref{lem: est-interation}.
 \begin{proof}[The proof of Lemma \ref{lem: est-interation}] Let $\varphi_j$ and $\phi_0$ be given by \eqref{LP-dp}. Due to the fact that $\rho\sim \sqrt{t}\geq 1$, we decompose
    \begin{equation}
    \begin{split}
    \tilde{I}_{\pm}(t, \bar{r}_{1}, \bar{r}_{2} )
    &=\int_{0}^{\infty}
   e^{i\rho(\rho\pm \bar{r}_2)} \tilde{a}_{\pm}(t^{-\frac12}\rho, \bar{r}_{1}\rho, \bar{r}_{2}\rho)  \big(1- \phi_0(4\bar{r}_2\rho)\big)  \rho^{n-1}d \rho\\
    &+\sum_{j\geq 1}\int_{0}^{\infty}
   e^{i\rho(\rho\pm \bar{r}_2)} \tilde{a}_{\pm}(t^{-\frac12}\rho, \bar{r}_{1}\rho, \bar{r}_{2}\rho) \varphi_j(\rho)  \phi_0(4\bar{r}_2\rho) \rho^{n-1} d \rho\\
   &=: \tilde{I}_{\pm, 1}+ \tilde{I}_{\pm,2}.
    \end{split}
  \end{equation} 
  Let us define 
  $$\Phi(\rho, \bar{r}_2)=\rho(\rho\pm \bar{r}_2), \quad L=L(\rho, \bar{r}_2)=(2\rho \pm \bar{r}_2)^{-1}\partial_\rho.$$
 Since the second integral on the right hand side is supported where $\rho\leq (4\bar{r}_2)^{-1}$ and $\rho\geq 1/2$, the integrand is only nonzero when $\bar{r}_2\leq 1/2$.
 Hence $|\partial_\rho \Phi|=2\rho \pm \bar{r}_2\geq\rho/2$. By \eqref{est:symb} and using integration by parts, for $N$ large enough, we obtain 
     \begin{equation}
    \begin{split}
 |\tilde{I}_{\pm,2}| &\leq \sum_{j\geq 1}\Big|\int_{0}^{\infty}
   L^N\Big(e^{i\rho(\rho\pm \bar{r}_2)}\Big) \tilde{a}_{\pm}(t^{-\frac12}\rho, \bar{r}_{1}\rho, \bar{r}_{2}\rho) \varphi_j(\rho)  \phi_0(4\bar{r}_2\rho) \rho^{n-1} d \rho\Big|\\
&\leq \sum_{j\geq 1}r_{1}^{\sigma} r_{2}^{-\frac{n-1}2}  \int_{\rho\sim 2^j}
 \rho^{-2N} \rho^{n-1} d \rho\lesssim r_{1}^{\sigma} r_{2}^{-\frac{n-1}2},
    \end{split}
  \end{equation} 
which  gives the first term of \eqref{est:I-pm}.
  Finally we consider $ \tilde{I}_{\pm,1}$. We further make a decomposition based on the size of $|\partial_\rho \Phi|$
       \begin{equation}
    \begin{split}
|\tilde{I}_{\pm,1}| &\leq \Big|\int_{0}^{\infty}
e^{i\rho(\rho\pm \bar{r}_2)} \tilde{a}_{\pm}(t^{-\frac12}\rho, \bar{r}_{1}\rho, \bar{r}_{2}\rho) \phi_0(2\rho-\bar{r}_2)  \big(1- \phi_0(4\bar{r}_2\rho)\big) \rho^{n-1}  d \rho\Big|\\
   &+ \sum_{j\geq 1}\Big|\int_{0}^{\infty}
e^{i\rho(\rho\pm \bar{r}_2)} \tilde{a}_{\pm}(t^{-\frac12}\rho, \bar{r}_{1}\rho, \bar{r}_{2}\rho) \varphi_j(2\rho-\bar{r}_2)  \big(1- \phi_0(4\bar{r}_2\rho)\big) \rho^{n-1}  d \rho\Big|\\
&:= \tilde{I}_{\pm,1<}+\tilde{I}_{\pm,1>}.
    \end{split}
  \end{equation} 
Now we estimate $\tilde{I}_{\pm,1<}$. If $\bar{r}_{2}\leq 10$, then for the integrand of $\tilde{I}_{\pm,1<}$ to be nonzero we must have $1\leq t^{1/2}\sim \rho\leq 10$, due to the supports of $\phi_0$ and $\varphi(\rho/\sqrt{t})$.
  Then $$ |\tilde{I}_{\pm,1<}|\lesssim r_{1}^{\sigma} r_{2}^{-\frac{n-1}2} \int_{\rho\sim 1}
 \rho^{n-1}d \rho\lesssim r_{1}^{\sigma} r_{2}^{-\frac{n-1}2},$$
 which  is controlled by the first term of \eqref{est:I-pm}. If $\bar{r}_{2}\geq 10$, one has $\bar{r}_{2}\sim \rho$ since $|2\rho-\bar{r}_2|\leq 1$. 
   Then it gives 
      \begin{equation}\label{est:r2-t}
      |\tilde{I}_{\pm,1<}|\lesssim r_{1}^{\sigma} r_{2}^{-\frac{n-1}2} \int_{|2\rho-\bar{r}_2|\leq 1}
 \rho^{n-1}d \rho \lesssim r_{1}^{\sigma} r_{2}^{-\frac{n-1}2} \bar{r}_2^{n-1}\lesssim r_{1}^{\sigma}\big(\frac{r_{2}}{t}\big)^{\frac{n-1}2}.
   \end{equation}
  Next we estimate $\tilde{I}_{\pm,1>}$. Integrating by parts, we show by \eqref{est:symb}
         \begin{equation}
    \begin{split}
|\tilde{I}_{\pm,1>}|&\lesssim \sum_{j\geq 1}\Big|\int_{0}^{\infty}
 L^N\Big(e^{i\rho(\rho\pm \bar{r}_2)}\Big) \tilde{a}_{\pm}(t^{-\frac12}\rho, \bar{r}_{1}\rho, \bar{r}_{2}\rho) \varphi_j(2\rho-\bar{r}_2)  \big(1- \phi_0(4\bar{r}_2\rho)\big) \rho^{n-1}d \rho\Big|\\
 &\lesssim r_{1}^{\sigma} r_{2}^{-\frac{n-1}2} \sum_{j\geq 1} 2^{-jN}\int_{|2\rho-\bar{r}_2|\sim 2^j}
(\rho^{-N}+2^{-jN}) \rho^{n-1}d \rho.
    \end{split}
  \end{equation} 
 If $\bar{r}_2\leq 2^{j+1}$, then $1\leq \rho\leq 2^{j+2} $ on the support of the integrand. Then the above is bounded by
 $$|\tilde{I}_{\pm,1>}| \lesssim r_{1}^{\sigma} r_{2}^{-\frac{n-1}2}  \sum_{j\geq 1} 2^{-jN}\int_{\rho\geq 1}
\rho^{-N+n-1} d \rho\lesssim r_{1}^{\sigma} r_{2}^{-\frac{n-1}2},$$
which is controlled by the first term of \eqref{est:I-pm}.
 Otherwise, $\bar{r}_2\geq 2^{j+1}\implies \rho\sim \bar{r}_2$, 
 then the above is bounded by
  \begin{equation}\label{est:r2-t'}
|\tilde{I}_{\pm,1>}| \lesssim r_{1}^{\sigma} r_{2}^{-\frac{n-1}2}  \sum_{j\geq 1} 2^{-2jN}\int_{|2\rho-\bar{r}_2|\sim 2^j}
 \rho^{n-1}d \rho\lesssim r_{1}^{\sigma} r_{2}^{-\frac{n-1}2}\bar{r}_2^{n-1}\lesssim r_{1}^{\sigma}\big(\frac{r_{2}}{t}\big)^{\frac{n-1}2}.
   \end{equation} 
Note that we always have $\rho\sim \sqrt{t}$ due to the factor $\varphi(\rho/\sqrt{t})$; if $\bar{r}_2\sim \rho$, then $\bar{r}_2=r_2/\sqrt{t}\sim \sqrt{t}\implies r_2\sim t$.
Hence, \eqref{est:r2-t} and \eqref{est:r2-t'} give the second term of \eqref{est:I-pm}. We remark that this term is supported on the set $A:=\{r_2\sim t\}$.
  
 \end{proof}

  We finally deal with $T^4_{\nu}$ by modifying the argument of $T^3_{\nu}$. Using \eqref{eq:bess3} again, the kernel of $T^4_\nu$ splits into two parts:
  \begin{align}\label{I+-'}
  \begin{split}
    I_{\pm} & =(r_1r_2)^{-\frac{n-2}2} \chi^c(r_1)\chi^c(r_2)\int_{0}^{\infty}
    \rho \varphi(\rho)(r_1r_{2}\rho^2)^{-1/2}
   e^{it\rho^2} e^{\pm i (r_{1}\pm r_2)\rho}\mathfrak{a}_{\pm}(r_{1}\rho) \mathfrak{a}_{\pm}(r_{2}\rho)d \rho
   \\ & =(r_1r_2)^{-\frac{n-1}2}\chi^c(r_1)\chi^c(r_2) \int_{0}^{\infty}
     \varphi(\rho) e^{it\rho^2} e^{\pm i (r_{1}\pm r_2)\rho}\mathfrak{a}_{\pm}(r_{1}\rho) \mathfrak{a}_{\pm}(r_{2}\rho)d \rho .
  \end{split}
  \end{align}
Combining \eqref{eq:bess4} with the fact that $r_1,r_2 \geq 1/2$ on the support of $I_\pm$ and $\varphi(\rho)$ is supported on the region $\rho \sim 1$, we know $\| I_\pm \|_{L^\infty(\R \times \R)} \lesssim 1$.
If $|t| \lesssim 1$, we have 
\begin{equation}\label{est:I+-'<}
  |I_\pm g| \leq \| I_\pm \|_{L^\infty(\R \times \R)} \|g\|_{L^1(r_2^{n-1}dr_2)} \lesssim |t|^{-n/2}\|g\|_{L^1(r_2^{n-1}dr_2)}.
\end{equation}
Now we consider the region $|t| \geq 1$. Changing coordinates to
\begin{equation}
  \bar{r}_i=r_i |t|^{-1/2}, \, i = 1,2; \, \rho \to |t|^{1/2}\rho,
\end{equation}
from \eqref{I+-'}, we have
   \begin{equation} \label{eq:I-pm-2}
   \begin{split}
    I_{\pm}
   &=|t|^{-\frac {n}{2} } \chi^c(r_1)\chi^c(r_2)
   \int_{0}^{\infty}
   e^{i\rho[\rho\pm (\bar{r}_1\pm \bar{r}_2)]} \tilde{a}_{\pm}(t^{-\frac12}\rho, \bar{r}_{1}\rho, \bar{r}_{2}\rho) \rho^{n-1}d \rho,
   \end{split}
  \end{equation}
   where 
  \begin{equation} \label{def:ta'}
  \tilde{a}_{\pm}(t^{-\frac12}\rho, \bar{r}_{1}\rho, \bar{r}_{2}\rho)= \varphi(t^{-\frac12}\rho)(\bar{r}_{1}\bar{r}_{2}\rho^2)^{-\frac{n-1}2} \mathfrak{a}_{\pm}(\bar{r}_{1}\rho)  \mathfrak{a}_{\pm}(\bar{r}_{2}\rho).
  \end{equation}
Using \eqref{eq:bess4}, we obtain
\begin{equation} \label{est:symb'}
\begin{split}
\Big| \Big(\frac{\partial}{\partial\rho}\Big)^N \Big(\tilde{a}_{\pm}(t^{-\frac12}\rho, \bar{r}_{1}\rho, \bar{r}_{2}\rho)\Big)\Big| &\lesssim   (\bar{r}_{1}\bar{r}_{2}\rho^2)^{-\frac{n-1}2} \rho^{-N}\\
&\lesssim (1+ r_{1})^{-\frac{n-1}2} (1+ r_{2})^{-\frac{n-1}2} \rho^{-N},
\end{split}
\end{equation}
where we used $\rho \sim \sqrt{t}$ on the support of $\varphi(t^{-\frac12}\rho)$ and all of its derivatives. 

Similarly to Lemma \ref{lem: est-interation}, we have the following lemma:
 \begin{lemma}\label{lem: est-interation'} Let $$\tilde{I}_{\pm}(t, \bar{r}_{1}, \bar{r}_{2} )=\int_{0}^{\infty}
   e^{i\rho(\rho\pm (\bar{r}_1\pm\bar{r}_2))} \tilde{a}_{\pm}(t^{-\frac12}\rho, \bar{r}_{1}\rho, \bar{r}_{2}\rho) \rho^{n-1}d \rho,
$$
where $  \tilde{a}_{\pm}$ is given by \eqref{def:ta'} and satisfies \eqref{est:symb'}. Then, for $t, r_1, r_2\gtrsim 1$, the integral satisfies
   \begin{equation}\label{est:I-pm'}
  \big| \tilde{I}_{\pm}(t, \bar{r}_{1}, \bar{r}_{2} )\big|\lesssim  (r_1r_{2})^{-\frac{n-1}2}+ (r_1r_{2})^{-\frac{n-1}2}\Big(\frac{|r_2\pm r_1|}{\sqrt{t}}\Big)^{n-1}\chi_A \lesssim 1,
  \end{equation}
where $\chi_A$ is the characteristic function on the set $A:=\{|r_2\pm r_1|\sim t\}$.
 \end{lemma}

  \begin{proof}[The proof of Lemma \ref{lem: est-interation'}] 
Firstly, $(r_1r_{2})^{-\frac{n-1}2}\Big(\frac{|r_2\pm r_1|}{\sqrt{t}}\Big)^{n-1}\chi_A \lesssim 1$ because $\max(r_1,r_2) \gtrsim t$ on $A$, hence that term is $\lesssim t^{-\frac{n-1}{2}}t^{\frac{n-1}{2}} \lesssim 1$.

Now we prove \eqref{est:I-pm'}.
  Let $\varphi_j$ and $\phi_0$ be given by \eqref{LP-dp}. Since $\rho\sim \sqrt{t}\geq 1$, let $\bar{\zeta}=(\bar{r}_1\pm \bar{r}_2)$; without loss of generality, we may assume $\bar{\zeta}\geq 0$. We decompose $\tilde{I}_\pm$ as
    \begin{equation}
    \begin{split}
    \tilde{I}_{\pm}(t, \bar{r}_{1}, \bar{r}_{2} )
    &=\int_{0}^{\infty}
   e^{i\rho(\rho\pm \bar{\zeta})} \tilde{a}_{\pm}(t^{-\frac12}\rho, \bar{r}_{1}\rho, \bar{r}_2\rho)  \big(1- \phi_0(4\bar{\zeta}\rho)\big)  \rho^{n-1}d \rho\\
    &+\sum_{j\geq 1}\int_{0}^{\infty}
   e^{i\rho(\rho\pm \bar{\zeta})} \tilde{a}_{\pm}(t^{-\frac12}\rho, \bar{r}_{1}\rho, \bar{r}_{2}\rho) \varphi_j(\rho)  \phi_0(4\bar{\zeta}\rho) \rho^{n-1} d \rho\\
   &=: \tilde{I}_{\pm, 1}+ \tilde{I}_{\pm,2}.
    \end{split}
  \end{equation} 
We define 
  $$\Phi(\rho, \bar{\zeta})=\rho(\rho\pm \bar{\zeta}), \quad L=L(\rho, \bar{\zeta})=(2\rho \pm \bar{\zeta})^{-1}\partial_\rho.$$

Consider $\tilde{I}_{\pm,2}$ first. 
 Since its integrand is supported on the region with $\rho\leq (4\bar{\zeta})^{-1}$ and $\rho\geq 1/2$, the integrand is only nonzero when $\bar{\zeta}\leq 1/2$.
 Hence $|\partial_\rho \Phi|=2\rho \pm \bar{\zeta}\geq\rho/2$. By \eqref{est:symb'}, we integrate by parts $N$ times for $N$ large enough to obtain\footnote{For terms with derivatives hitting $\phi_0$, we use $\overline{\zeta} \lesssim \rho^{-1}$. Terms introduced by differentiating $\varphi_j$ have extra $2^{-j} \sim \rho^{-1}$ factors by the definition of $\varphi_j(\rho)=\varphi(2^{-j}\rho)$. }
\begin{equation}
\begin{split}
|\tilde{I}_{\pm,2}|&\leq \Big|\sum_{j\geq 1}\int_{0}^{\infty}
   L^N\Big(e^{i\rho(\rho\pm \bar{\zeta})}\Big) \tilde{a}_{\pm}(t^{-\frac12}\rho, \bar{r}_{1}\rho, \bar{r}_{2}\rho) \varphi_j(\rho)  \phi_0(4\bar{\zeta}\rho) \rho^{n-1} d \rho\Big|\\
&\lesssim \sum_{j\geq 1}(r_{1}r_{2})^{-\frac{n-1}2}  \int_{\rho\sim 2^j}
 \rho^{-2N} \rho^{n-1} d \rho\lesssim (r_{1} r_{2})^{-\frac{n-1}2},
    \end{split}
  \end{equation} 
which  gives the first term of \eqref{est:I-pm'}.
  Finally we consider $ \tilde{I}_{\pm,1}$. We further decompose the integral in terms of the size of $|\partial_\rho \Phi|$:
       \begin{equation}
    \begin{split}
     |\tilde{I}_{\pm,1}|&\leq \Big|\int_{0}^{\infty}
e^{i\rho(\rho\pm \bar{\zeta})} \tilde{a}_{\pm}(t^{-\frac12}\rho, \bar{r}_{1}\rho, \bar{r}_{2}\rho) \phi_0(2\rho-\bar{\zeta})  \big(1- \phi_0(4\bar{\zeta}\rho)\big) \rho^{n-1}  d \rho\Big|\\
   &+ \sum_{j\geq 1}\Big|\int_{0}^{\infty}
e^{i\rho(\rho\pm \bar{\zeta})} \tilde{a}_{\pm}(t^{-\frac12}\rho, \bar{r}_{1}\rho, \bar{r}_{2}\rho) \varphi_j(2\rho-\bar{\zeta})  \big(1- \phi_0(4\bar{\zeta}\rho)\big) \rho^{n-1}  d \rho\Big|\\
&:= \tilde{I}_{\pm,1<}+\tilde{I}_{\pm,1>}.
    \end{split}
  \end{equation} 
Now we estimate $\tilde{I}_{\pm,1<}$. If $\bar{\zeta}\leq 10$, then for the integrand of $\tilde{I}_{\pm,1<}$ to be nonzero we must have $1\leq t^{1/2}\sim \rho\leq 10$, due to the supports of $\phi_0$ and $\varphi(\rho/\sqrt{t})$.
  Then $$ |\tilde{I}_{\pm,1<}|\lesssim (r_{1}r_{2})^{-\frac{n-1}2} \int_{\rho\sim 1}
 \rho^{n-1}d \rho\lesssim (r_{1} r_{2})^{-\frac{n-1}2},$$
 which  is controlled by the first term of \eqref{est:I-pm'}. If $\bar{\zeta}\geq 10$, one has $\bar{\zeta}\sim \rho$ (hence $|r_1\pm r_2| \sim t$) since $|2\rho-\bar{\zeta}|\leq 1$. Then it gives 
      \begin{equation}
      \label{est:r2-t'-6.4}
      |\tilde{I}_{\pm,1<}|\lesssim (r_{1}r_{2})^{-\frac{n-1}2} \int_{|2\rho-\bar{\zeta}|\leq 1}
 \rho^{n-1}d \rho \lesssim (r_{1} r_{2})^{-\frac{n-1}2} \bar{\zeta}^{n-1}\lesssim (r_{1}r_2)^{-\frac{n-1}2}\bar{\zeta}^{n-1}.
   \end{equation}
Note that we always have $\rho\sim \sqrt{t}$ due to the factor $\varphi(\rho/\sqrt{t})$, if $\bar{\zeta}\sim \rho$, then $\bar{\zeta}=(r_1\pm r_2)/\sqrt{t}\sim \sqrt{t}\implies r_1\pm r_2\sim t$.
So the contribution that is bounded by the right hand side of \eqref{est:r2-t'-6.4} is controlled by the second term of \eqref{est:I-pm'}.

Next we estimate $\tilde{I}_{\pm,1>}$. Integrating by parts and using \eqref{est:symb'}, we have
         \begin{equation}
    \begin{split}
|\tilde{I}_{\pm,1>}|&\lesssim \sum_{j\geq 1}\Big|\int_{0}^{\infty}
 L^N\Big(e^{i\rho(\rho\pm \bar{\zeta})}\Big) \tilde{a}_{\pm}(t^{-\frac12}\rho, \bar{r}_{1}\rho, \bar{r}_{2}\rho) \varphi_j(2\rho-\bar{\zeta})  \big(1- \phi_0(4\bar{\zeta}\rho)\big) \rho^{n-1}d \rho\Big|\\
 &\lesssim (r_{1} r_{2})^{-\frac{n-1}2} \sum_{j\geq 1} 2^{-jN}\int_{|2\rho-\bar{\zeta}|\sim 2^j}
(\rho^{-N}+2^{-jN}) \rho^{n-1}d \rho.
    \end{split}
  \end{equation} 
  
 If $\bar{\zeta}\leq 2^{j+1}$, then $1\leq \rho\leq 2^{j+2} $ on the support of the integrand. Then the right hand side of the equation above is bounded by
 $$|\tilde{I}_{\pm,1>}|\lesssim (r_{1} r_{2})^{-\frac{n-1}2}  \sum_{j\geq 1} 2^{-jN}\int_{\rho\geq 1}
\rho^{-N+n-1} d \rho\lesssim (r_{1}r_{2})^{-\frac{n-1}2},$$
which is controlled by the first term of \eqref{est:I-pm'}.
 Otherwise, $\bar{\zeta}\geq 2^{j+1}\implies \rho\sim \bar{\zeta}$, 
 then the above is bounded by
  \begin{equation}\label{est:r2-t''}
  |\tilde{I}_{\pm,1>}|\lesssim (r_{1}r_{2})^{-\frac{n-1}2}  \sum_{j\geq 1} 2^{-2jN}\int_{|2\rho-\bar{\zeta}|\sim 2^j}
 \rho^{n-1}d \rho\lesssim (r_{1}r_{2})^{-\frac{n-1}2}\bar{\zeta}^{n-1}.
   \end{equation} 
For the same reason as explained after \eqref{est:r2-t'-6.4}, the right hand side of \eqref{est:r2-t''} is controlled by the second term of \eqref{est:I-pm'} and this completes the proof.
\end{proof}

By using \eqref{est:I+-'<}, \eqref{eq:I-pm-2} and \eqref{est:I-pm'}, we obtain
\begin{equation}
  \| T^4_\nu g\|_{L^\infty_{r_1^{n-1}dr_1}} \lesssim |t|^{-n/2}\|g\|_{L^1_{r_2^{n-1}dr_2}}. 
\end{equation}
By interpolating this with the $L^2$-estimate for $T^4_{\nu}$ similar to \eqref{eq:T3-L2-est}, we obtain 
      \begin{equation}\label{est:q-q'4}
    \begin{split}
  &\|T^4_{\nu}g\|_{L^q_{r^{n-1}_1 dr_1}}\leq C |t|^{-\frac n2(1-\frac2q)}\|g\|_{L^{q'}_{r^{n-1}_2 dr_2}},\quad q\geq2 .
    \end{split}
  \end{equation}
  Collecting the estimates on the terms $T^j_\nu$ yields \eqref{est:q-q'} and the proof is concluded.
\end{proof}

\section{The decay estimates for the half-wave propagator}\label{sec:decayW}

In this section, we prove the decay estimate \eqref{dis-w}. Concretely, we mainly prove the following frequency-localized results:
\begin{proposition}\label{prop:mic} Let $\varphi$ be as in \eqref{LP-dp} and $\alpha=\nu_0-(n-2)/2$. If $\alpha\ge 0$, then there exists a constant $C$ such that
\begin{equation}\label{est: mic-decay1}
\begin{split}
\big\|\varphi(2^{-j}\sqrt{H})&e^{it\sqrt{H}}f\big\|_{L^\infty(X)}\\&\leq C 2^{jn}\big(1+2^j|t|\big)^{-\frac{n-1}2}\|\varphi(2^{-j}\sqrt{H}) f\|_{L^1(X)}.
\end{split}
\end{equation}
If $-(n-2)/2<\alpha<0$, for $q\in[2, q(\alpha))$, then
\begin{equation}\label{est: mic-decay1'}
\begin{split}
\big\|\varphi(2^{-j}\sqrt{H})&e^{it\sqrt{H}}f\big\|_{L^q(X)}\\&\leq C 2^{jn(1-\frac2q)}\big(1+2^j|t|\big)^{-\frac{n-1}2(1-\frac2q)}\|\varphi(2^{-j}\sqrt{H}) f\|_{L^{q'}(X)}.
\end{split}
\end{equation}
\end{proposition}
Indeed, if we could prove \eqref{est: mic-decay1}, then \eqref{dis-w} follows from
\begin{equation*}
\begin{split}
&\big\|e^{it\sqrt{H}}f\big\|_{L^\infty(X)}\leq \sum_{j\in\Z}\big\|\varphi(2^{-j}\sqrt{H})e^{it\sqrt{H}}f\big\|_{L^\infty(X)}\\
&\leq C |t|^{-\frac{n-1}2}\sum_{j\in\Z}2^{\frac{n+1}2j}\|\varphi(2^{-j}\sqrt{H}) f\|_{L^1(X)}\leq C |t|^{-\frac{n-1}2}\|f\|_{\dot{\mathcal{B}}^{\frac{n+1}2}_{1,1}(X)}.
\end{split}
\end{equation*}
If $-(n-2)/2<\alpha<0$, the estimate  \eqref{est: mic-decay1'} and the Littlewood-Paley square function estimate \eqref{square} show \eqref{dis-w'}
\begin{equation*}
\begin{split}
&\big\|e^{it\sqrt{H}}f\big\|^2_{L^q(X)}\leq \sum_{j\in\Z}\big\|\varphi(2^{-j}\sqrt{H})e^{it\sqrt{H}}f\big\|^2_{L^q(X)}\\
&\leq C |t|^{-(n-1)(1-\frac2q)}\sum_{j\in\Z}2^{j(n+1)(1-\frac2q)}\|\varphi(2^{-j}\sqrt{H}) f\|^2_{L^{q'}(X)}\\&\leq C  |t|^{-(n-1)(1-\frac2q)}\|f\|^2_{\dot{\mathcal{B}}^{\frac{n+1}2(1-\frac2q)}_{q',2}(X)}.
\end{split}
\end{equation*}

The rest of this section is to prove this proposition. For this purpose, we follow the argument of \cite{DPR, WZZ}, in which we need the subordination formula and Bernstein inequalities associated with the operator $H$. We state them here for convenience but omit the proof.
The following proposition about the subordination formula is from \cite[Proposition 4.1]{MS} and \cite[Proposition 2.2]{DPR}, and we use the one formulated in \cite{WZZ}.
\begin{proposition}
If $\varphi(\lambda)\in C_c^\infty(\mathbb{R})$ is supported in $[\frac{1}{2},2]$, then, for all $j\in\Z, t, x>0$ with $2^jt\geq 1$,  we can write
\begin{equation}\label{key}
\begin{split}
&\varphi(2^{-j}\sqrt{x})e^{it\sqrt{x}}\\&=\rho\big(\frac{tx}{2^j}, 2^jt\big)
+\varphi(2^{-j}\sqrt{x})\big(2^jt\big)^{\frac12}\int_0^\infty \chi(s,2^jt)e^{\frac{i2^jt}{4s}}e^{i2^{-j}tsx}\,ds,
\end{split}
\end{equation}
where $\rho(s,\tau)\in\mathcal{S}(\mathbb{R}\times\mathbb{R})$ is a Schwartz function and $\chi\in C^\infty(\mathbb{R}\times\mathbb{R})$ with $\text{supp}\,\chi(\cdot,\tau)\subseteq[\frac{1}{16},4]$ such that
\begin{equation}\label{est:chi}
\sup_{\tau\in\R}\big|\partial_s^\alpha\partial_\tau^\beta \chi(s,\tau)\big|\lesssim_{\alpha,\beta}(1+|s|)^{-\alpha},\quad \forall \alpha,\beta\geq0.
\end{equation}
\end{proposition}

If this is proven, then by the spectral theory for the non-negative self-adjoint operator $H$, we have the representation of the microlocalized half-wave propagator
\begin{equation}\label{key-operator}
\begin{split}
&\varphi(2^{-j}\sqrt{H})e^{it\sqrt{H}}\\&=\rho\big(\frac{tH}{2^j}, 2^jt\big)
+\varphi(2^{-j}\sqrt{H})\big(2^jt\big)^{\frac12}\int_0^\infty \chi(s,2^jt)e^{\frac{i2^jt}{4s}}e^{i2^{-j}tsH}\,ds.
\end{split}
\end{equation}

\begin{proof}[The proof of Proposition \ref{prop:mic}] We estimate the microlocalized half-wave propagator $$ \big\|\varphi(2^{-j}\sqrt{H})e^{it\sqrt{H}}f\big\|_{L^\infty(X)}$$
by considering two cases: $|t|2^j\geq 1$ and $|t|2^{j}\leq 1$. In the following argument, as before, we can choose $\tilde{\varphi}\in C_c^\infty((0,+\infty))$ such that $\tilde{\varphi}(\lambda)=1$ if $\lambda\in\mathrm{supp}\,\varphi$
and $\tilde{\varphi}\varphi=\varphi$. Since $\tilde{\varphi}$ has the same property as $\varphi$, we drop the tilde above $\varphi$ for brevity. 
\vspace{0.1cm}

We first consider the case that $\alpha\geq 0$. \vspace{0.1cm}

{\bf Case 1: $t2^{j}\lesssim 1$.}  By the spectral theorem, similarly to \eqref{est:L2}, one has the $L^2$-estimate
$$\|e^{it\sqrt{H}}\|_{L^2(X)\to L^2(X)}\leq C.$$
Together with this, we use the Bernstein inequality \eqref{est:Bern}
 to prove
\begin{equation*}
\begin{split}
&\big\|\varphi(2^{-j}\sqrt{H})e^{it\sqrt{H}}f\big\|_{L^\infty(X)}\\
&\lesssim 2^{\frac{nj}2}\|e^{it\sqrt{H}}\varphi(2^{-j}\sqrt{H}) f\|_{L^2(X)}\\
&\lesssim 2^{\frac{nj}2}\|\varphi(2^{-j}\sqrt{H}) f\|_{L^2(X)}\lesssim 2^{nj}\|\varphi(2^{-j}\sqrt{H}) f\|_{L^1(X)}.
\end{split}
\end{equation*}
In this case $0<t\leq 2^{-j}$, we have
\begin{equation}\label{<1}
\begin{split}
&\big\|\varphi(2^{-j}\sqrt{H})e^{it\sqrt{H}}f\big\|_{L^\infty(X)}\\
&\quad\lesssim 2^{nj}(1+2^jt)^{-N}\|\varphi(2^{-j}\sqrt{H}) f\|_{L^1(X)},\quad \forall N\geq 0,
\end{split}
\end{equation}
which shows \eqref{est: mic-decay1}.
\vspace{0.2cm}

{\bf Case 2: $t2^{j}\geq 1$.} In this case, we can use
 \eqref{key-operator} to obtain the microlocalized half-wave propagator
\begin{equation*}
\begin{split}
&\varphi(2^{-j}\sqrt{H})e^{it\sqrt{H}}\\&=\rho\big(\frac{tH}{2^j}, 2^jt\big)
+\varphi(2^{-j}\sqrt{H})\big(2^jt\big)^{\frac12}\int_0^\infty \chi(s,2^jt)e^{\frac{i2^jt}{4s}}e^{i2^{-j}tsH}\,ds.
\end{split}
\end{equation*}
We first use the spectral theorem and the Bernstein inequality  again to estimate
\begin{equation*}
\begin{split}
&\big\|\varphi(2^{-j}\sqrt{H})\rho\big(\frac{tH}{2^j}, 2^jt\big) f\big\|_{L^\infty(X)}.
\end{split}
\end{equation*}
Indeed, since $\rho\in \mathcal{S}(\R\times\R)$, we have 
$$\big\|\rho\big(\frac{tH}{2^j}, 2^jt\big)\big\|_{L^2\to L^2 }\leq C(1+2^jt)^{-N},\quad \forall N\geq 0.$$
Therefore, we use  the Bernstein inequality in Proposition \ref{prop:Bern} and  the spectral theorem to show
\begin{equation*}
\begin{split}
&\big\|\varphi(2^{-j}\sqrt{H})\rho\big(\frac{tH}{2^j}, 2^jt\big) f\big\|_{L^\infty(X)}\lesssim 2^{\frac{nj}2}\Big\|\rho\big(\frac{tH}{2^j}, 2^jt\big)\varphi(2^{-j}\sqrt{H}) f\Big\|_{L^2(X)}\\
&\lesssim 2^{\frac{nj}2}(1+2^jt)^{-N}\Big\|\varphi(2^{-j}\sqrt{H}) f\Big\|_{L^2(X)}\lesssim 2^{nj}(1+2^jt)^{-N}\Big\|\varphi(2^{-j}\sqrt{H}) f\Big\|_{L^1(X)}.
\end{split}
\end{equation*}
If $\nu_0\geq \frac{n-2}2$, i.e., $\alpha\geq0$, we use the dispersive estimates of the Schr\"odinger propagator (see \eqref{est:dis-cl})
\begin{equation*}
\begin{split}
&\big\|e^{itH} f\big\|_{L^\infty(X)}\leq C |t|^{-\frac n2}\big\| f\big\|_{L^1(X)},\quad t\neq 0,
\end{split}
\end{equation*}
 to estimate
\begin{equation*}
\begin{split}
&\big\|\varphi(2^{-j}\sqrt{H})\big(2^jt\big)^{\frac12}\int_0^\infty \chi(s,2^jt)e^{\frac{i2^jt}{4s}}e^{i2^{-j}tsH}f\,ds \big\|_{L^\infty(X)}.
\end{split}
\end{equation*}
For $t\neq 0$, we obtain
\begin{equation*}
\begin{split}
&\big\|\varphi(2^{-j}\sqrt{H})\big(2^jt\big)^{\frac12}\int_0^\infty \chi(s,2^jt)e^{\frac{i2^jt}{4s}}e^{i2^{-j}tsH}f\,ds \big\|_{L^\infty(X)}\\
&\lesssim\big(2^jt\big)^{\frac12}\int_0^\infty \chi(s,2^jt)|2^{-j}ts|^{-\frac n2}\,ds \big\|\varphi(2^{-j}\sqrt{H})f\big\|_{L^1(X)}\\
&\lesssim\big(2^jt\big)^{\frac12}(2^{-j}t)^{-\frac n2}\int_0^\infty \chi(s,2^jt)\,ds \big\|\varphi(2^{-j}\sqrt{H})f\big\|_{L^1(X)}\\
&\lesssim 2^{nj}\big(2^jt\big)^{-\frac{n-1}2}\big\|\varphi(2^{-j}\sqrt{H})f\big\|_{L^1(X)}\lesssim 2^{nj}\big(1+2^jt\big)^{-\frac{n-1}2}\big\|\varphi(2^{-j}\sqrt{H})f\big\|_{L^1(X)}
\end{split}
\end{equation*}
due to the fact that $s\in [\frac1{16}, 4]$ on the support of $\chi$.\vspace{0.1cm}

For the case that $-(n-2)/2<\alpha<0$, we repeat the above argument to prove \eqref{est: mic-decay1'} by replacing $L^\infty$ by $L^q$ and $L^1$ by $L^{q'}$ for $q \in [2,q(\alpha))$. It is worth mentioning that the $L^{q'}-L^q$ estimate \eqref{est:LqLq'0} is used to replace \eqref{est:dis-cl} in this case.

Therefore, we have completed the proof of Proposition \ref{prop:mic}.

\end{proof}

{\bf Conflicts of Interest Statement:}
The authors declare that there are no conflicts of interest relevant to the content of this manuscript. The research was conducted without any commercial or financial relationships that could be construed as a potential conflict of interest.

{\bf Data Availability Statement:}
The data supporting the findings of this study are available from the corresponding author upon reasonable request.

\begin{center}

\end{center}

\end{document}